\documentclass[11pt]{article}
\usepackage{amsmath,fullpage,amssymb,amsthm,hyperref,enumerate,mathtools}
\usepackage{tikz,color}
\usetikzlibrary{patterns}

\usepackage[utf8]{inputenc}
\usepackage[english]{babel}
\usepackage{amsfonts}
\usepackage{graphicx}
\usepackage[export]{adjustbox}
\usepackage{setspace} 
\usepackage{booktabs}
\usepackage{stmaryrd}
\usepackage{pgfplots}
\pgfplotsset{compat=newest}
\usepackage{sidecap}
\usepackage{tikz-3dplot}
\usetikzlibrary{calc,patterns,arrows,shapes.arrows,intersections,plotmarks}
\usepackage{xcolor}
\usepackage{xstring}
\usepackage{circuitikz}
\usepackage{caption}
\usepackage{listings}
\newtheorem{theorem}{Theorem}[section]
\newtheorem{proposition}[theorem]{Proposition}
\newtheorem{lemma}[theorem]{Lemma}
\newtheorem{corollary}[theorem]{Corollary}
  
\newtheorem{problem}{Problem}

\theoremstyle{definition}
  
\newtheorem{definition}[theorem]{Definition}
\newtheorem{algorithm}{Algorithm}

\theoremstyle{remark}

\usepackage[backend=biber,  url=false, doi=false,isbn=false,maxnames=5]{biblatex}
\newcommand{\cD}{\mathcal{D}}
\newcommand{\cB}{\mathcal{B}}
\newcommand{\cT}{\mathcal{T}}
\newcommand{\cI}{\mathcal{I}}
\newcommand{\cC}{\mathcal{C}}
\DeclareMathOperator{\Conv}{Conv}
\newcommand{\bbN}{\mathbb{N}}
\newcommand{\bbZ}{\mathbb{Z}}
\newcommand{\bbR}{\mathbb{R}}
\newcommand{\TYP}{\mathbb{Y}_\Delta}
\newcommand{\Wide}{\Omega}
\newcommand\I{I}
\DeclareMathOperator{\dif}{dif}
\DeclareMathOperator{\wrd}{wrd}
\DeclareMathOperator{\arm}{arm}
\DeclareMathOperator{\leg}{leg}
\newcommand\bigO{\mathcal{O}}
\newcommand\pp{\hat{p}}
\newcommand{\LL}{\mathsf{L}}
\newcommand{\HH}{\mathsf{H}}
\newcommand{\axes}[2]{\draw[->] (-.5,0)--(#1+.5,0); \draw[->] (0,-.5)--(0,#2+.5);
\draw[gray,very thin] (0,0) grid (#1,#2);}

\title{Triangular partitions: enumeration, structure, and generation}
\author{Sergi Elizalde\thanks{Department of Mathematics, Dartmouth College, Hanover, NH, USA. \texttt{sergi.elizalde@dartmouth.edu}} \ and
Alejandro B.\ Galván\thanks{Centre de Formació Interdisciplinària Superior (CFIS) -- Universitat Politècnica de Catalunya (UPC), Barcelona, Spain. \texttt{alejandrobasilio7@gmail.com}}
}

\begin{document}

\maketitle

\begin{abstract}
A {\em triangular partition} is a partition whose Ferrers diagram can be separated from its complement (as a subset of $\bbN^2$) by a straight line. Having their origins in combinatorial number theory and computer vision, triangular partitions have been studied from a combinatorial perspective by Onn and Sturmfels, and by Corteel et al.\ under the name {\em plane corner cuts}, and more recently by Bergeron and Mazin. In this paper we derive new enumerative, geometric and algorithmic properties of such partitions. 

We give a new characterization of triangular partitions and the cells that can be added or removed while preserving the triangular condition, and use it to describe the Möbius function of the restriction of Young's lattice to triangular partitions.
We obtain a formula for the number of triangular partitions whose Young diagram fits inside a square, deriving, as a byproduct, a new proof of Lipatov's enumeration theorem for balanced words. 
Finally, we present an algorithm that generates all the triangular partitions of a given size, which is significantly more efficient than previous ones and allows us to compute the number of triangular partitions of size up to~$10^5$.
    \end{abstract}



\section{Introduction}\label{sec:intro}

An integer partition is said to be triangular if its Ferrers diagram can be separated from its complement by a straight line. Triangular partitions and their higher-dimensional generalizations have been studied from different perspectives during the last five decades. They first appeared in the context of combinatorial number theory~\cite{Boshernitzan1981}, where they were called \textit{almost linear sequences}. 
Later, the closely related notion of \textit{digital straight lines} became relevant in the field of computer vision~\cite{Bruckstein1990}. From a combinatorial perspective, triangular partitions were first studied by Onn and Sturmfels~\cite{Onn1999}, who defined them in any dimension and called them \textit{corner cuts}. Soon after, Corteel et al.~\cite{Corteel1999} found an expression for the generating function for the number of plane corner cuts. More recently, motivated by work of Blasiak et al.~\cite{Blasiak2023} generalizing the shuffle theorem for paths under a line, Bergeron and Mazin~\cite{Bergeron2023} coined the term \emph{triangular partitions} and studied some of their combinatorial properties.

In this paper we obtain further enumerative, geometric, poset-theoretic, and algorithmic properties of triangular partitions. 
The paper is structured as follows. In Section~\ref{sec:background} we give basic definitions and summarize some of the previous work on triangular partitions.
In Section~\ref{sec:characterizations_triangular} we give a simple alternative characterization of triangular partitions, as those for which the convex hull of the Ferrers diagram and that of its complement (as a subset of $\bbN^2$) have an empty intersection.
We also characterize which cells can be added to or removed from the Young diagram while preserving triangularity.

In Section~\ref{sec:triangular_young_poset} we study the restriction of Young's lattice to triangular partitions. It was shown in~\cite{Bergeron2023} that this poset is a lattice. Here we completely describe its Möbius function, and we provide an explicit construction of the join and the meet of two triangular partitions. 

In Section~\ref{sec:sturmian}, we introduce a new encoding of triangular partitions in terms of balanced words, and use it to implement an algorithm which computes the number of triangular partitions of every size $n\le N$ in time $\bigO(N^{5/2})$. This allows us to produce the first $10^5$ terms of this sequence, compared to the $39$ terms that were known previously. 

In Section \ref{sec:generating-functions}, refining the approach from \cite{Corteel1999}, we obtain generating functions for triangular partitions with a given number of removable and addable cells. In Section~\ref{sec:subpartitions}, we provide a formula for the number of triangular partitions whose Young diagram fits inside a square (or equivalently, inside a staircase), which involves Euler's totient function. As a byproduct, we obtain a new combinatorial proof of a formula of Lipatov~\cite{Lipatov1982} for the number of balanced words.

We conclude by discussing some generalizations of triangular partitions in Section~\ref{sec:further}, and proposing possible directions for future research.

\section{Background}\label{sec:background}

\subsection{Triangular partitions}
A \emph{partition} $\lambda$ is a weakly decreasing sequence of positive integers, often called the {\em parts} of $\lambda$.
We will write $\lambda=(\lambda_1,\lambda_2,\dots,\lambda_k)$,
or $\lambda=\lambda_1\lambda_2\dots\lambda_k$ when there is no confusion. We call $|\lambda|=\lambda_1+\lambda_2+\dots+\lambda_k$ the {\em size} of $\lambda$. If $|\lambda|=n$, we say that $\lambda$ is a partition of~$n$.

We use $\bbN$ to denote the set of positive integers. The {\em Ferrers diagram} of $\lambda$ is the set of lattice points  
$$\{(a,b)\in\bbN^2\mid 1\le b \leq k,\;1\le a\leq\lambda_b\}.$$ 
The {\em Young diagram} of $\lambda$ is the set of unit squares (called {\em cells}) whose north-east corners are the points in the Ferrers diagram; see the examples in Figure~\ref{fig:partitions}. We identify each cell with its north-east corner, so we will also use the term cell to refer to points in the Ferrers diagram. In particular, we say that a cell lies above, below or on a line when the north-east corner does. Additionally, we will often identify $\lambda$ with its Ferrers diagram and with its Young diagram, and use notation such as $c=(a,b)\in\lambda$.

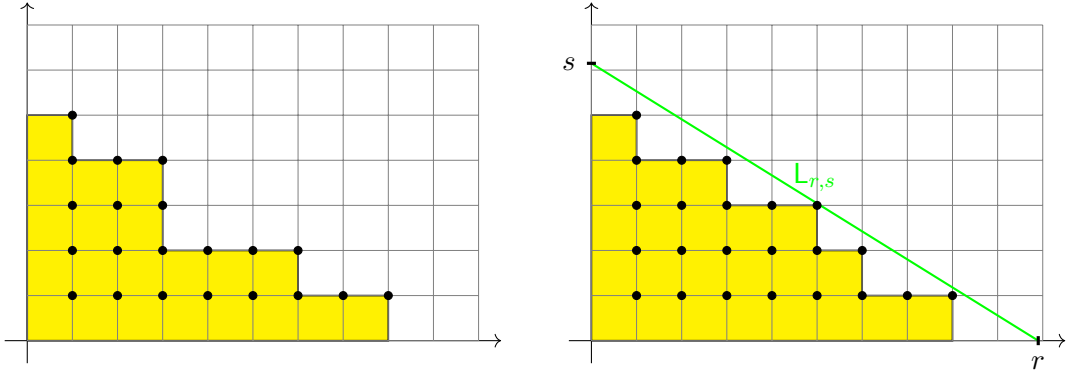
\begin{figure}[ht]
\centering
	\begin{tikzpicture}[scale=.6]
		\filldraw[color=black, fill=yellow, semithick] (0,0) -- (0,5) -- (1,5) -- (1,4) -- (3,4) -- (3,2) -- (6,2) -- (6,1) -- (8,1) -- (8,0) -- (0,0);
      \axes{10}{7}
         \foreach \c in {(1,1),(2,1),(3,1),(4,1),(5,1),(6,1),(7,1),(8,1),(1,2),(2,2),(3,2),(4,2),(5,2),(6,2),(1,3),(2,3),(3,3),(1,4),(2,4),(3,4),(1,5)} 
            {\filldraw[black] \c circle (2.5pt); }
    \begin{scope}[shift={(12.5,0)}]
		\filldraw[color=black, fill=yellow, semithick] (0,0) -- (0,5) -- (1,5) -- (1,4) -- (3,4) -- (3,3) -- (5,3) -- (5,2) -- (6,2) -- (6,1) -- (8,1) -- (8,0) -- (0,0);
        \axes{10}{7}
        \draw[green, thick] (0,6.15) -- node[above] {$\LL_{r,s}$} (9.9,0);
        \draw[very thick] (.1,6.15)--(-.1,6.15) node[left]{$s$};
         \draw[very thick] (9.9,.1)--(9.9,-.1) node[below]{$r$};
        \foreach \c in {(1,1),(2,1),(3,1),(4,1),(5,1),(6,1),(7,1),(8,1),(1,2),(2,2),(3,2),(4,2),(5,2),(6,2),(1,3),(2,3),(3,3),(4,3),(5,3),(1,4),(2,4),(3,4),(1,5)} 
            {\filldraw[black] \c circle (2.5pt); }
    \end{scope}
	\end{tikzpicture}
\caption{The Young diagram and the Ferrers diagram of partitions $\lambda = 86331$ (left) and $\tau=86531$ (right). The green cutting line shows that $\tau$ is triangular.}
\label{fig:partitions}
\end{figure}

For a partition $\lambda=\lambda_1\lambda_2\dots\lambda_k$, we call $\lambda_1$ its {\em width}, and $k$ its {\em height}.
The partition $(k,k-1,\dots ,2,1)$ will be referred to as the \emph{staircase partition} of height $k$, and denoted by $\sigma^k$. 
The {\em conjugate} of a partition $\lambda$, obtained by reflecting its Ferrers diagram along the diagonal $y=x$, will be denoted by $\lambda'$. 
Identifying $\lambda$ with its Ferrers diagram, we define its \emph{complement} to be the set $\bbN^2\setminus\lambda$.

Now we can state the definition of our main objects of study.

\begin{definition}\label{def:triangular}
A partition $\tau = \tau_1\tau_2\dots\tau_k$ is \emph{triangular} if there exist positive real numbers $r$ and $s$ such that
$$ \tau_j = \left\lfloor{r - jr/s}\right\rfloor, $$
for $1\le j\le k$, and $k = \lfloor s - s/r\rfloor$.
\end{definition}
In other words, $\tau$ is triangular if its Ferrers diagram consists of the points in $\bbN^2$ that lie on or below the line that passes through $(0,s)$ and $(r,0)$. 
This line, which has equation $x/r+y/s=1$, will be denoted by $\LL_{r,s}$ in the rest of the paper.
We say that $\LL_{r,s}$ is a \emph{cutting line} for $\tau$, or that it \emph{cuts off} $\tau$. A vector $v\in\bbR_{>0}^2$ is called a \emph{slope vector} of $\tau$ if it is perpendicular to a cutting line for $\tau$.
Unlike in the definition given in \cite{Bergeron2023}, here we do not allow $\tau$ to have parts equal to $0$, hence the condition on $k$.

Denote by $\Delta(n)$ the set of triangular partitions of $n$, and by $\Delta=\bigcup_{n\ge0}\Delta(n)$ the set of all triangular partitions. Throughout the paper, we will often use $\tau$ to denote a triangular partition.

\subsection{Enumeration}\label{sec:enumeration}

Corteel et al.~\cite{Corteel1999} prove that the set of slope vectors of any given nonempty triangular partition $\tau$ is an open cone. For the different partitions in $\Delta(n)$, these cones are disjoint. To count triangular partitions of $n$, Corteel et al.\ find the number of separating rays between these open cones. Each ray can be identified by a pair of relatively prime numbers $(a,b)$ that determine its slope. If the cone adjacent to such a ray from the right corresponds to a triangular partition $\tau$, then there is a line $L$ perpendicular to $(a,b)$ so that $\tau$ consists of the points in $\bbN$ strictly below $L$, along with the leftmost $m$ points in $L\cap\bbN^2$, for some $m$. Letting $k=|L\cap\bbN^2|$, and letting
$i+1$ be the $x$-coordinate of the leftmost cell in $L\cap\bbN^2$, and $j+1$ be the $y$-coordinate of the rightmost cell in $L\cap\bbN^2$,
the line $L$ can be encoded by the parameters $a,b,i,j,k$.
Then, $\tau$ can be decomposed as shown in Figure \ref{fig:dec_tp}, and its size $n$ is equal to
\begin{align}\label{eq:NDelta}
N_\Delta(a,b,k,m,i,j) & = (k - 1)\left(\frac{(a + 1)(b + 1)}{2} - 1\right) + \binom{k - 1}{2}ab + ij\\
&\quad + i(k - 1)a + j(k - 1)b + T(a,b,j) + T(b,a,i) + m,
\nonumber
\end{align}
where $T(a,b,j) = \sum_{r = 1}^j(\lfloor rb/a\rfloor + 1)$.

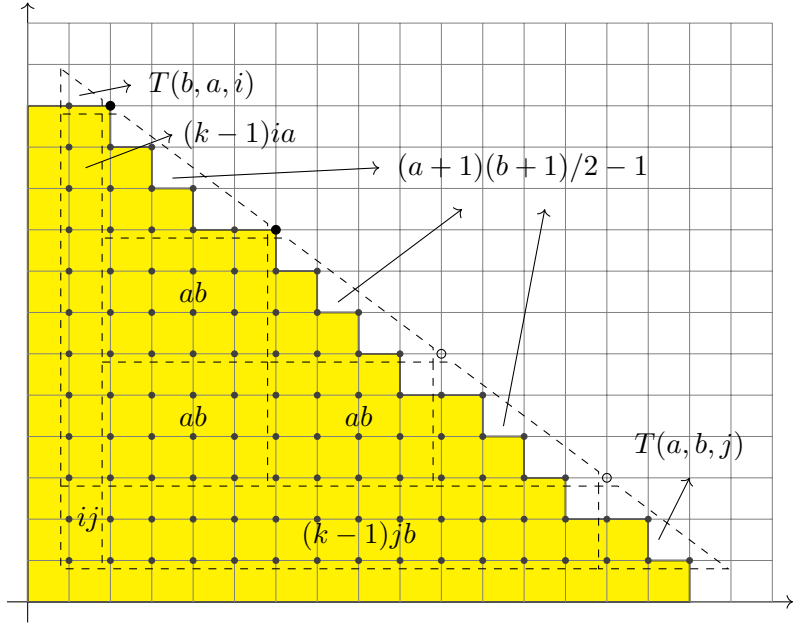
\begin{figure}[ht]
\centering
\begin{tikzpicture}[scale=.55]
		\filldraw[color=black, fill=yellow, semithick] (0,0) -- (0,12) -- (2,12) -- (2,11) -- (3,11) -- (3,10) -- (4,10) -- (4,9) -- (6,9) -- (6,8) -- (7,8) -- (7,7) -- (8,7) -- (8,6) -- (9,6) -- (9,5) -- (11,5) -- (11,4) -- (12,4) -- (12,3) -- (13,3) -- (13,2) -- (15,2) -- (15,1) -- (16,1) -- (16,0) -- (0,0);
        \axes{18}{14}
        \draw[dashed] (1 - 1/5, 12 + 9/10) -- (1 - 1/5, 12 - 1/5) -- (2 - 1/5, 12 - 1/5) -- (2 - 1/5, 12 + 3/20) -- (1 - 1/5, 12 + 9/10);
        \foreach \i in {0,1,2}
            {\draw[dashed] (2 - 1/5 + 4*\i, 12 - 1/5 - 3*\i) -- (2 - 1/5 + 4*\i, 9 - 1/5 - 3*\i) -- (6 - 1/5 + 4*\i, 9 - 1/5 - 3*\i) -- (6 - 1/5 + 4*\i, 9 + 3/20 - 3*\i) -- (2 + 4/15 + 4*\i, 12 - 1/5 - 3*\i) -- (2 - 1/5 + 4*\i, 12 - 1/5 - 3*\i);}
        \draw[dashed] (14 - 1/5, 3 - 1/5) -- (14 - 1/5, 1 - 1/5) -- (16 + 14/15, 1 - 1/5) -- (14 + 4/15, 3 - 1/5) -- (14 - 1/5, 3 - 1/5);
        \draw[dashed] (1 - 1/5, 12 - 1/5) -- (1 - 1/5, 1 - 1/5) -- (14 - 1/5, 1 - 1/5);
        \draw[dashed] (2 - 1/5, 9 - 1/5) -- (2 - 1/5, 1 - 1/5);
        \draw[dashed] (1 - 1/5, 3 - 1/5) -- (10 - 1/5, 3 - 1/5);
        \draw[dashed] (2 - 1/5, 6 - 1/5) -- (6 - 1/5, 6 - 1/5);
        \draw[dashed] (6 - 1/5, 6 - 1/5) -- (6 - 1/5, 3 - 1/5);
        \filldraw (4,7) circle (1.5pt) node[above]{$ab$};
        \filldraw (4,4) circle (1.5pt) node[above]{$ab$};
        \filldraw (8,4) circle (1.5pt) node[above]{$ab$};
        \filldraw (2,2) circle (1.5pt) node[left]{$ij$};
        \filldraw (8,1) circle (1.5pt) node[above]{$(k - 1)jb$};
        \draw[->] (1.4, 10.5) -- (3.5, 11.3) node[right]{$(k-1)ia$};
        \foreach \x in {1,...,16}
            {\filldraw[gray!150] (\x,1) circle (2pt);}
        \foreach \x in {1,...,15}
            {\filldraw[gray!150] (\x,2) circle (2pt);}
        \foreach \x in {1,...,13}
            {\filldraw[gray!150] (\x,3) circle (2pt);}
        \draw (14,3) circle (3pt);
        \foreach \x in {1,...,12}
            {\filldraw[gray!150] (\x,4) circle (2pt);}
        \foreach \x in {1,...,11}
            {\filldraw[gray!150] (\x,5) circle (2pt);}
        \foreach \x in {1,...,9}
            {\filldraw[gray!150] (\x,6) circle (2pt);}
        \draw (10,6) circle (3pt);
        \foreach \x in {1,...,8}
            {\filldraw[gray!150] (\x,7) circle (2pt);}
        \foreach \x in {1,...,7}
            {\filldraw[gray!150] (\x,8) circle (2pt);}
        \foreach \x in {1,...,5}
            {\filldraw[gray!150] (\x,9) circle (2pt);}
        \filldraw (6,9) circle (3pt);
        \foreach \x in {1,...,4}
            {\filldraw[gray!150] (\x,10) circle (2pt);}
        \foreach \x in {1,...,3}
            {\filldraw[gray!150] (\x,11) circle (2pt);}
        \filldraw[gray!150] (1,12) circle (2pt);
        \filldraw (2,12) circle (3pt);
        \draw[->] (3.5, 10.25) -- (8.5, 10.5) node[right, outer sep=3pt]{$(a + 1)(b + 1)/2 - 1$};
        \draw[->] (7.5, 7.25) -- (10.5, 9.5);
        \draw[->] (11.5, 4.25) -- (12.5, 9.5);
        \draw[->] (1.25, 12.25) -- (2.5, 12.5) node[right, outer sep=3pt]{$T(b,a,i)$};
        \draw[->] (15.25, 1.5) -- (16, 3) node[above, outer sep=3pt]{$T(a,b,j)$};
	\end{tikzpicture}
\caption{The decomposition of triangular partitions used in~\cite{Corteel1999}.}
\label{fig:dec_tp}
\end{figure}

Corteel et al.~\cite{Corteel1999} deduce that the generating function of triangular partitions with respect to their size can be expressed as follows.
\begin{theorem}[\cite{Corteel1999}]\label{thm:GDelta}
$$G_{\Delta}(z) \coloneqq \sum_{n\ge0}|\Delta(n)|z^n = \frac{1}{1 - z} + \sum_{\gcd(a, b) = 1}\sum_{\substack{0\leq j < a \\ 0\leq i < b}}\sum_{1\leq m < k}z^{N_\Delta(a,b,k,m,i,j)}.$$
\end{theorem}
The term $\frac{1}{1 - z}$ in the above generating function accounts for the open cone that lies to the left of the leftmost separating ray, along with the empty partition.
Exploiting the idea of identifying each triangular partition of $n$ with a pair of relatively prime numbers, Corteel et al.~\cite{Corteel1999} obtain the following bounds.

\begin{theorem}[\cite{Corteel1999}]\label{thm:corteel}
    There exist positive constants $C$ and $C'$ such that, for all $n > 1$,
    $$
    C n\log n < |\Delta(n)| < C'n\log n.
    $$
\end{theorem}

\subsection{Addable and removable cells} \label{subsection:bergeron}

Let $\lambda = \lambda_1\dots\lambda_k$ be a partition, and let $c = (i,j)$ be a cell of its Young diagram. Define the {\em arm length} and the {\em leg length} of $c$ to be $\arm(c) = \lambda_j - i$ and $\leg(c) = \lambda'_i-j$, that is, the number of cells to the right of $c$ in its row, and above $c$ in its column, respectively. Bergeron and Mazin~\cite{Bergeron2023} give the following characterization of triangular partitions.

\begin{lemma}[{\cite[Lemma 1.2]{Bergeron2023}}]\label{charact_bergeron}
    A partition $\lambda$ is triangular if and only if $t_\lambda^- < t_\lambda^+$, where 
    $$ t_\lambda^- = \max_{c\in\lambda}\frac{\leg(c)}{\arm(c) + \leg(c) + 1}, \quad\text{and}\quad  
    t_\lambda^+ = \min_{c\in\lambda}\frac{\leg(c) + 1}{\arm(c) + \leg(c) + 1}. $$ 
    In this case, $(t, 1 - t)$ is a slope vector of $\lambda$ if and only if $t_\lambda^- < t < t_\lambda^+$.
\end{lemma}

Bergeron and Mazin~\cite{Bergeron2023} also define a moduli space of lines, where each point $(r,s)\in\bbR_{>0}^2$ represents the line $\LL_{r,s}$. The set of lines that pass through the lattice point $(a,b)\in\bbN^2$ are then represented in the moduli space by the hyperbola $\HH_{a,b} = \{(r,s)\in\bbR_{>0}^2\mid (r - a)(s - b) = ab\}$. Every $(r,s)$ that lies above (respectively, below) $\HH_{a,b}$ represents a line which cuts off a triangular partition $\tau$ whose Young diagram contains (respectively, does not contain) the cell $(a,b)$. Therefore, we can interpret crossing the hyperbola $\HH_{a,b}$ as ``adding'' or ``removing'' the cell $(a,b)$. It is proved 
in~\cite{Bergeron2023} that there is a natural bijection between the connected components of $\bbR_{>0}^2\setminus\bigcup_{(a,b)\in\bbN^2}\HH_{a,b}$ and the set $\Delta$ of triangular partitions, as shown in Figure~\ref{fig:moduli_space_of_lines}. This interpretation motivates the following definition.

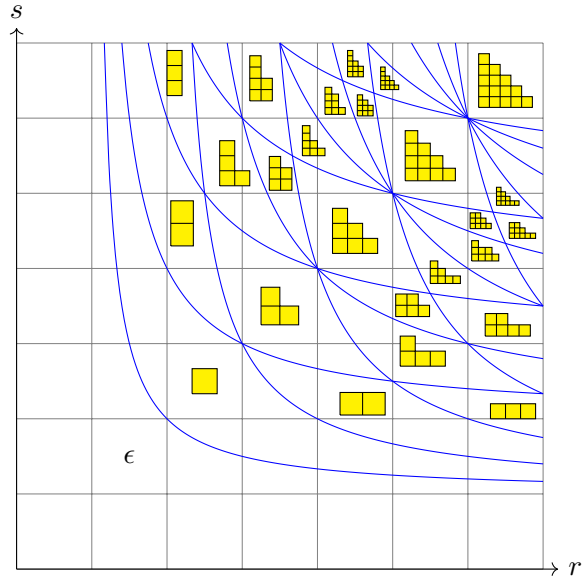
\begin{figure}[ht]
    \centering
    \begin{tikzpicture}[scale=1,domain=0:7]
      \draw[very thin,color=gray] grid (7,7);    
      \draw[->] (0,0) -- (7.2,0) node[right] {$r$};
      \draw[->] (0,0) -- (0,7.2) node[above] {$s$};    
      \draw[color=blue, domain=7/6:7, smooth, samples = 100]    plot (\x,{\x/(\x-1)});
      \draw[color=blue, domain=14/6:7, smooth, samples = 100]    plot (\x,{\x/(\x - 2)});
      \draw[color=blue, domain=7/5:7, smooth, samples = 100]    plot (\x,{2*\x/(\x - 1)});
      \draw[color=blue, domain=14/5:7, smooth, samples = 100]    plot (\x,{2*\x/(\x - 2)});
      \draw[color=blue, domain=21/6:7, smooth, samples = 100]    plot (\x,{\x/(\x - 3)});
      \draw[color=blue, domain=21/5:7, smooth, samples = 100]    plot (\x,{2*\x/(\x - 3)});
      \draw[color=blue, domain=7/4:7, smooth, samples = 100]    plot (\x,{3*\x/(\x - 1)});
      \draw[color=blue, domain=14/4:7, smooth, samples = 100]    plot (\x,{3*\x/(\x - 2)});
      \draw[color=blue, domain=21/4:7, smooth, samples = 100]    plot (\x,{3*\x/(\x - 3)});
      \draw[color=blue, domain=28/6:7, smooth, samples = 100]    plot (\x,{\x/(\x - 4)});
      \draw[color=blue, domain=28/5:7, smooth, samples = 100]    plot (\x,{2*\x/(\x - 4)});
      \draw[color=blue, domain=7/3:7, smooth, samples = 100]    plot (\x,{4*\x/(\x - 1)});
      \draw[color=blue, domain=14/3:7, smooth, samples = 100]    plot (\x,{4*\x/(\x - 2)});
      \draw[color=blue, domain=35/6:7, smooth, samples = 100]    plot (\x,{\x/(\x - 5)});
      \draw[color=blue, domain=7/2:7, smooth, samples = 100]    plot (\x,{5*\x/(\x - 1)});
      \filldraw[color=black, fill=yellow] (2 + 1/3, 2 + 1/3) -- (2 + 1/3, 2 + 2/3) -- (2 + 2/3, 2 + 2/3) -- (2 + 2/3, 2 + 1/3) -- (2 + 1/3, 2 + 1/3);
      \filldraw[color=black, fill=yellow] (2.05, 4.3) -- (2.05, 4.9) -- (2.35, 4.9) -- (2.35, 4.3) -- (2.05, 4.3);
      \draw[color=black] (2.05, 4.6) -- (2.35, 4.6);
      \filldraw[color=black, fill=yellow] (3.36, 5.04) -- (3.36, 5.49) -- (3.51, 5.49) -- (3.51, 5.34) -- (3.66, 5.34) -- (3.66, 5.04) -- (3.36, 5.04);
      \draw[color=black] (3.36, 5.34) -- (3.51, 5.34);
      \draw[color=black] (3.36, 5.19) -- (3.66, 5.19);
      \draw[color=black] (3.51, 5.34) -- (3.51, 5.04);
      \filldraw[color=black, fill=yellow] (4.53, 6.03) -- (4.53, 6.31) -- (4.6, 6.31) -- (4.6, 6.24) -- (4.67, 6.24) -- (4.67, 6.17) -- (4.74, 6.17) -- (4.74, 6.03) -- (4.53, 6.03);
      \draw[color=black] (4.53, 6.24) -- (4.6, 6.24);
      \draw[color=black] (4.53, 6.17) -- (4.67, 6.17);
      \draw[color=black] (4.53, 6.1) -- (4.74, 6.1);
      \draw[color=black] (4.6, 6.24) -- (4.6, 6.03);
      \draw[color=black] (4.67, 6.17) -- (4.67, 6.03);
      \filldraw[color=black, fill=yellow] (3.25,3.25) -- (3.25,3.75) -- (3.5,3.75) -- (3.5,3.5) -- (3.75,3.5) -- (3.75,3.25) -- (3.25,3.25);
      \draw[color=black] (3.25, 3.5) -- (3.5,3.5) -- (3.5,3.25);
      \filldraw[color=black, fill=yellow] (2,6.3) -- (2, 6.9) -- (2.2, 6.9) -- (2.2, 6.3) -- (2,6.3);
      \draw[color=black] (2,6.5) -- (2.2,6.5);
      \draw[color=black] (2,6.7) -- (2.2,6.7);
      \filldraw[color=black, fill=yellow] (2.7,5.1) -- (2.7, 5.7) -- (2.9, 5.7) -- (2.9, 5.3) -- (3.1,5.3) -- (3.1,5.1) -- (2.7,5.1);
      \draw[color=black] (2.7,5.3) -- (2.9,5.3) -- (2.9,5.1);
      \draw[color=black] (2.7,5.5) -- (2.9,5.5);
      \filldraw[color=black, fill=yellow] (4.2,4.2) -- (4.2,4.8) -- (4.4,4.8) -- (4.4,4.6) -- (4.6,4.6) -- (4.6,4.4) -- (4.8,4.4) -- (4.8,4.2) -- (4.2,4.2);
      \draw[color=black] (4.2,4.6) -- (4.4,4.6) -- (4.4,4.2);
      \draw[color=black] (4.2,4.4) -- (4.6,4.4) -- (4.6,4.2);
      \filldraw[color=black, fill=yellow] (3.1, 6.23) -- (3.1,6.83) -- (3.25, 6.83) -- (3.25, 6.53) -- (3.4, 6.53) -- (3.4, 6.23) -- (3.1, 6.23);
      \draw[color=black] (3.1, 6.53) -- (3.25, 6.53);
      \draw[color=black] (3.1, 6.38) -- (3.4, 6.38);
      \draw[color=black] (3.25, 6.53) -- (3.25, 6.23);
      \draw[color=black] (3.1, 6.68) -- (3.25, 6.68);
      \filldraw[color=black, fill=yellow] (3.8,5.5) -- (3.8,5.9) -- (3.9,5.9) -- (3.9,5.7) -- (4,5.7) -- (4,5.6) -- (4.1,5.6) -- (4.1,5.5) -- (3.8,5.5);
      \draw[color=black] (3.8,5.8) -- (3.9,5.8);
      \draw[color=black] (3.8,5.7) -- (3.9,5.7) -- (3.9,5.5);
      \draw[color=black] (3.8,5.6) -- (4,5.6) -- (4,5.5);
      \filldraw[color=black, fill=yellow] (4.1,6.05) -- (4.1,6.41) -- (4.19,6.41) -- (4.19,6.32) -- (4.28,6.32) -- (4.28,6.14) -- (4.37,6.14) -- (4.37,6.05) -- (4.1,6.05);
      \draw[color=black] (4.1,6.32) -- (4.19,6.32) -- (4.19,6.05);
      \draw[color=black] (4.1,6.23) -- (4.28,6.23);
      \draw[color=black] (4.1,6.14) -- (4.28,6.14) -- (4.28,6.05);
      \filldraw[color=black, fill=yellow] (4.4,6.55) -- (4.4,6.9) -- (4.47,6.9) -- (4.47,6.76) -- (4.54,6.76) -- (4.54,6.69) -- (4.61,6.69) -- (4.61,6.55) -- (4.4,6.55);
      \draw[color=black] (4.4,6.83) -- (4.47,6.83);
      \draw[color=black] (4.4,6.76) -- (4.47,6.76) -- (4.47,6.55);
      \draw[color=black] (4.4,6.69) -- (4.54,6.69) -- (4.54,6.55);
      \draw[color=black] (4.4,6.62) -- (4.61,6.62);
      \filldraw[color=black, fill=yellow] (5 + 1/6, 5 + 1/6) -- (5 + 1/6, 5 + 5/6) -- (5 + 1/3, 5 + 5/6) -- (5 + 1/3, 5 + 2/3) -- (5.5, 5 + 2/3) -- (5.5, 5.5) -- (5 + 2/3, 5.5) -- (5 + 2/3, 5 + 1/3) -- (5 + 5/6, 5 + 1/3) -- (5 + 5/6, 5 + 1/6) -- (5 + 1/6, 5 + 1/6);
      \draw[color=black] (5 + 1/6, 5 + 2/3) -- (5 + 1/3, 5 + 2/3) -- (5 + 1/3, 5 + 1/6);
      \draw[color=black] (5 + 1/6, 5.5) -- (5.5, 5.5) -- (5.5, 5 + 1/6);
      \draw[color=black] (5 + 1/6, 5 + 1/3) -- (5 + 2/3, 5 + 1/3) -- (5 + 2/3, 5 + 1/6);
      \filldraw[color=black, fill=yellow] (6 + 1/7, 6 + 1/7) -- (6 + 1/7, 6 + 6/7) -- (6 + 2/7, 6 + 6/7) -- (6 + 2/7, 6 + 5/7) -- (6 + 3/7, 6 + 5/7) -- (6 + 3/7, 6 + 4/7) -- (6 + 4/7, 6 + 4/7) -- (6 + 4/7, 6 + 3/7) -- (6 + 5/7, 6 + 3/7) -- (6 + 5/7, 6 + 2/7) -- (6 + 6/7, 6 + 2/7) -- (6 + 6/7, 6 + 1/7) -- (6 + 1/7, 6 + 1/7);
      \draw[color=black] (6 + 1/7, 6 + 5/7) -- (6 + 2/7, 6 + 5/7) -- (6 + 2/7, 6 + 1/7);
      \draw[color=black] (6 + 1/7, 6 + 4/7) -- (6 + 3/7, 6 + 4/7) -- (6 + 3/7, 6 + 1/7);
      \draw[color=black] (6 + 1/7, 6 + 3/7) -- (6 + 4/7, 6 + 3/7) -- (6 + 4/7, 6 + 1/7);
      \draw[color=black] (6 + 1/7, 6 + 2/7) -- (6 + 5/7, 6 + 2/7) -- (6 + 5/7, 6 + 1/7);
      \filldraw[color=black, fill=yellow] (4.84,6.38) -- (4.84,6.68) -- (4.9,6.68) -- (4.9,6.56) -- (4.96,6.56) -- (4.96,6.5) -- (5.02,6.5) -- (5.02,6.44) -- (5.08,6.44) -- (5.08,6.38) -- (4.84,6.38);
      \draw[color=black] (4.84,6.62) -- (4.9,6.62);
      \draw[color=black] (4.84,6.56) -- (4.9,6.56) -- (4.9,6.38);
      \draw[color=black] (4.84,6.5) -- (4.96,6.5) -- (4.96,6.38);
      \draw[color=black] (4.84,6.44) -- (5.02, 6.44) -- (5.02,6.38);
      \filldraw[color=black, fill=yellow] (5.04,3.36) -- (5.49,3.36) -- (5.49,3.51) -- (5.34,3.51) -- (5.34,3.66) -- (5.04,3.66) -- (5.04,3.36);
      \draw[color=black] (5.34,3.36) -- (5.34,3.51);
      \draw[color=black] (5.19,3.36) -- (5.19,3.66);
      \draw[color=black] (5.34,3.51) -- (5.04,3.51);
      \filldraw[color=black, fill=yellow] (6.03,4.53) -- (6.31,4.53) -- (6.31,4.6) -- (6.24,4.6) -- (6.24,4.67) -- (6.17,4.67) -- (6.17,4.74) -- (6.03,4.74) -- (6.03,4.53);
      \draw[color=black] (6.24,4.53) -- (6.24,4.6);
      \draw[color=black] (6.17,4.53) -- (6.17,4.67);
      \draw[color=black] (6.1,4.53) -- (6.1,4.74);
      \draw[color=black] (6.24,4.6) -- (6.03,4.6);
      \draw[color=black] (6.17,4.67) -- (6.03,4.67);
      \filldraw[color=black, fill=yellow] (5.1,2.7) -- (5.7,2.7) -- (5.7,2.9) -- (5.3,2.9) -- (5.3,3.1) -- (5.1,3.1) -- (5.1,2.7);
      \draw[color=black] (5.3,2.7) -- (5.3,2.9) -- (5.1,2.9);
      \draw[color=black] (5.5,2.7) -- (5.5,2.9);
       \filldraw[color=black, fill=yellow] (6.23,3.1) -- (6.83,3.1) -- (6.83,3.25) -- (6.53,3.25) -- (6.53,3.4) -- (6.23,3.4) -- (6.23,3.1);
      \draw[color=black] (6.53,3.1) -- (6.53,3.25);
      \draw[color=black] (6.38,3.1) -- (6.38,3.4);
      \draw[color=black] (6.53,3.25) -- (6.23,3.25);
      \draw[color=black] (6.68,3.1) -- (6.68,3.25);
      \filldraw[color=black, fill=yellow] (5.5,3.8) -- (5.9,3.8) -- (5.9,3.9) -- (5.7,3.9) -- (5.7,4) -- (5.6,4) -- (5.6,4.1) -- (5.5,4.1) -- (5.5,3.8);
      \draw[color=black] (5.8,3.8) -- (5.8,3.9);
      \draw[color=black] (5.7,3.8) -- (5.7,3.9) -- (5.5,3.9);
      \draw[color=black] (5.6,3.8) -- (5.6,4) -- (5.5,4);
      \filldraw[color=black, fill=yellow] (6.05,4.1) -- (6.41,4.1) -- (6.41,4.19) -- (6.32,4.19) -- (6.32,4.28) -- (6.14,4.28) -- (6.14,4.37) -- (6.05,4.37) -- (6.05,4.1);
      \draw[color=black] (6.32,4.1) -- (6.32,4.19) -- (6.05,4.19);
      \draw[color=black] (6.23,4.1) -- (6.23,4.28);
      \draw[color=black] (6.14,4.1) -- (6.14,4.28) -- (6.05,4.28);
      \filldraw[color=black, fill=yellow] (6.55,4.4) -- (6.9,4.4) -- (6.9,4.47) -- (6.76,4.47) -- (6.76,4.54) -- (6.69,4.54) -- (6.69,4.61) -- (6.55,4.61) -- (6.55,4.4);
      \draw[color=black] (6.83,4.4) -- (6.83,4.47);
      \draw[color=black] (6.76,4.4) -- (6.76,4.47) -- (6.55,4.47);
      \draw[color=black] (6.69,4.4) -- (6.69,4.54) -- (6.55,4.54);
      \draw[color=black] (6.62,4.4) -- (6.62,4.61);
      \filldraw[color=black, fill=yellow] (6.38,4.84) -- (6.68,4.84) -- (6.68,4.9) -- (6.56,4.9) -- (6.56,4.96) -- (6.5,4.96) -- (6.5,5.02) -- (6.44,5.02) -- (6.44,5.08) -- (6.38,5.08) -- (6.38,4.84);
      \draw[color=black] (6.62,4.84) -- (6.62,4.9);
      \draw[color=black] (6.56,4.84) -- (6.56,4.9) -- (6.38,4.9);
      \draw[color=black] (6.5,4.84) -- (6.5,4.96) -- (6.38,4.96);
      \draw[color=black] (6.44,4.84) -- (6.44,5.02) -- (6.38,5.02);
      \filldraw[color=black, fill=yellow] (6.3,2) -- (6.9,2) -- (6.9,2.2) -- (6.3,2.2) -- (6.3,2);
      \draw[color=black] (6.5,2) -- (6.5,2.2);
      \draw[color=black] (6.7,2) -- (6.7,2.2);
      \filldraw[color=black, fill=yellow] (4.3,2.05) -- (4.9,2.05) -- (4.9,2.35) -- (4.3,2.35) -- (4.3,2.05);
      \draw[color=black] (4.6,2.05) -- (4.6,2.35);
      \node at (1.5,1.5) {$\epsilon$};
    \end{tikzpicture}
    \caption{The connected regions of $\bbR_{>0}^2\setminus\bigcup_{(a,b)\in\bbN^2}\HH_{a,b}$ correspond to triangular partitions.}
     \label{fig:moduli_space_of_lines}
\end{figure}

\begin{definition}
A cell of a triangular partition $\tau$ is \emph{removable} if removing it from $\tau$ yields a triangular partition. A cell of the complement $\bbN^2\setminus\tau$ is \emph{addable} if adding it to $\tau$ yields a triangular partition.
\end{definition}

The following results are proved in~\cite{Bergeron2023}.

\begin{lemma}[{\cite[Lemma 4.5]{Bergeron2023}}] \label{lem:removable-addable}
    Every nonempty triangular partition has either one removable cell and two addable cells, two removable cells and one addable cell, or two removable cells and two addable cells.
\end{lemma}

\begin{lemma}[{\cite[Lemma 3.2]{Bergeron2023}}] \label{line_touching_removable_cell}
    Let $\tau$ be a triangular partition and let $c\in\tau$ be a removable cell. Then $\tau$ has a cutting line $\LL_{r,s}$ such that $c$ is the only point of $\tau$ on $\LL_{r,s}$.
\end{lemma}

\begin{lemma}[{\cite[Lemma 3.8]{Bergeron2023}}] \label{parallel_lines_removable_addable_cells}
    Let $\tau$ be a triangular partition with two removable and two addable cells. Then, the line containing the two removable cells is parallel to the line containing the two addable cells.
\end{lemma}

Bergeron and Mazin considered the poset $\TYP$ of triangular partitions ordered by containment of their Young diagrams; equivalently, the restriction of Young's lattice to the subset of triangular partitions. 
The lower portion of the Hasse diagram of this poset appears in Figure~\ref{fig:hasse_diagram}. We write the order relation as $\tau\le\nu$, which by definition is equivalent to $\tau\subseteq\nu$, so we use both notations interchangeably.
The covering relations in $\TYP$ can be described as follows.

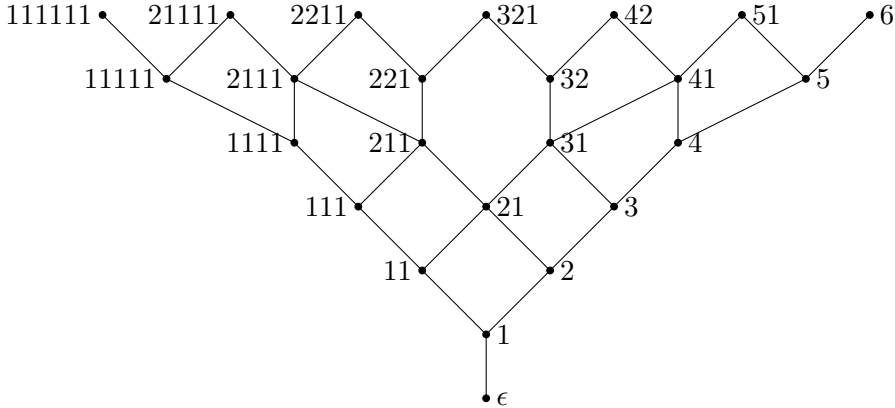
\begin{figure}[ht]
\centering
	\begin{tikzpicture}[scale=.85]
            \draw (0,0) -- (0,1) -- (-1,2) -- (-2,3) -- (-3,4) -- (-5,5) -- (-6,6);
            \draw (-5,5) -- (-4,6) -- (-3,5) -- (-3,4);
            \draw (-3,5) -- (-1,4) -- (-2,3);
            \draw (-1,4) -- (0,3) -- (-1,2);
            \draw (-3,5) -- (-2,6) -- (-1,5) -- (-1,4);
            \draw (-1,5) -- (0,6);
            \draw (0,1) -- (1,2) -- (2,3) -- (3,4) -- (5,5) -- (6,6);
            \draw (5,5) -- (4,6) -- (3,5) -- (3,4);
            \draw (3,5) -- (1,4) -- (2,3);
            \draw (1,4) -- (0,3) -- (1,2);
            \draw (3,5) -- (2,6) -- (1,5) -- (1,4);
            \draw (1,5) -- (0,6);
		\filldraw (0,0) circle (1.5pt) node[right]{$\epsilon$};
            \filldraw (0,1) circle (1.5pt) node[right]{$1$};
            \filldraw (-1,2) circle (1.5pt) node[left]{$11$};
            \filldraw (1,2) circle (1.5pt) node[right]{$2$};
            \filldraw (-2,3) circle (1.5pt) node[left]{$111$};
            \filldraw (0,3) circle (1.5pt) node[right]{$21$};
            \filldraw (2,3) circle (1.5pt) node[right]{$3$};
            \filldraw (-3,4) circle (1.5pt) node[left]{$1111$};
            \filldraw (-1,4) circle (1.5pt) node[left]{$211$};
            \filldraw (1,4) circle (1.5pt) node[right]{$31$};
            \filldraw (3,4) circle (1.5pt) node[right]{$4$};
            \filldraw (-5,5) circle (1.5pt) node[left]{$11111$};
            \filldraw (-3,5) circle (1.5pt) node[left]{$2111$};
            \filldraw (-1,5) circle (1.5pt) node[left]{$221$};
            \filldraw (1,5) circle (1.5pt) node[right]{$32$};
            \filldraw (3,5) circle (1.5pt) node[right]{$41$};
            \filldraw (5,5) circle (1.5pt) node[right]{$5$};
            \filldraw (-6,6) circle (1.5pt) node[left]{$111111$};
            \filldraw (-4,6) circle (1.5pt) node[left]{$21111$};
            \filldraw (-2,6) circle (1.5pt) node[left]{$2211$};
            \filldraw (0,6) circle (1.5pt) node[right]{$321$};
            \filldraw (2,6) circle (1.5pt) node[right]{$42$};
            \filldraw (4,6) circle (1.5pt) node[right]{$51$};
            \filldraw (6,6) circle (1.5pt) node[right]{$6$};
	\end{tikzpicture}
\caption{The lower portion of the Hasse diagram of $\TYP$.}
\label{fig:hasse_diagram}
\end{figure}

\begin{lemma}[{\cite[Lemma 4.2]{Bergeron2023}}]
    Let $\tau,\nu\in\TYP$ such that $\tau<\nu$. Then, $\tau\lessdot\nu$ if and only if $\tau$ is obtained from $\nu$ by removing exactly one cell. In particular, $\TYP$ is ranked by the size of the partitions.
\end{lemma}

The following property follows from the above description of the moduli space of lines.

\begin{lemma}[{\cite[Corollary 4.1, Lemma 4.4]{Bergeron2023}}]
The poset $\TYP$ has a planar Hasse diagram, and it is a lattice.
\end{lemma}

We conclude this section with some definitions from~\cite{Bergeron2023} that will be useful later on.

\begin{definition}\label{def:diagonal}
The \emph{diagonal} of a triangular partition $\tau$, denoted by $\partial_\tau$, is the set of cells in the segment whose endpoints are the removable cells of $\tau$ (with the convention that if there is only one removable cell, then the diagonal is just that cell).
The slope of this line will be called the \emph{diagonal slope}.
The \emph{interior} of $\tau$ is $\tau^\circ = \tau\setminus\partial_\tau$.
\end{definition}

It is shown in~\cite{Bergeron2023} that if $\tau$ is a triangular partition, then so is $\tau^\circ$. Additionally, if $\partial_\tau$ contains $k\geq2$ cells, then the Hasse diagram of the interval $[\tau^\circ,\tau]$ is a polygon with $2k$ sides. It follows that the Hasse diagram of $\TYP$ tiles a region of the plane with $2k$-gons for $k\ge2$.

\section{Characterizations of triangular partitions} \label{sec:characterizations_triangular}

Bergeron and Mazin's~\cite{Bergeron2023} characterization of triangular partitions, given in Lemma~\ref{charact_bergeron} above, requires computing some quotients of arm and leg lengths for all the cells in the partition. In this section, we introduce an alternative and arguably simpler characterization of triangular partitions in terms of convex hulls, along with various ways to identify removable and addable cells. We then use these to describe an algorithm which determines if an integer partition is triangular and finds its removable and addable cells.

The convex hull of a set $S\subseteq\bbN^2$ will be denoted by $\Conv(S)$. See Figure~\ref{fig:conv} for an illustration of the next proposition.

\begin{figure}[ht]
\centering
	\begin{tikzpicture}[scale=.7]
	    \draw[gray, pattern color = gray, pattern=crosshatch] (1,6) -- (2,5) -- (7,2) -- (9,1) -- (10,1) -- (10,7) -- (1,7) -- (1,6);
		\filldraw[black, fill=yellow, semithick] (0,0) -- (0,5) -- (1,5) -- (1,4) -- (3,4) -- (3,3) -- (5,3) -- (5,2) -- (6,2) -- (6,1) -- (8,1) -- (8,0) -- (0,0);
        \draw[gray, pattern color = gray, pattern=crosshatch] (1,1) -- (1,5) -- (5,3) -- (8,1) -- (1,1);
        \foreach \i in {1,...,10} 
            {\foreach \j in {1,...,7}  {\filldraw[black, fill=white] (\i,\j) circle (2pt); }}
        \foreach \c in {(1,1),(2,1),(3,1),(4,1),(5,1),(6,1),(7,1),(8,1),(1,2),(2,2),(3,2),(4,2),(5,2),(6,2),(1,3),(2,3),(3,3),(4,3),(1,4),(2,4),(3,4),(1,5)} 
            {\filldraw[black] \c circle (2pt); }
        \axes{10}{7}
        \draw[green, thick] (0,6.15) -- (10,0);
	\end{tikzpicture}
\caption{The convex hulls of $86531$ and its complement.}
\label{fig:conv}
\end{figure}
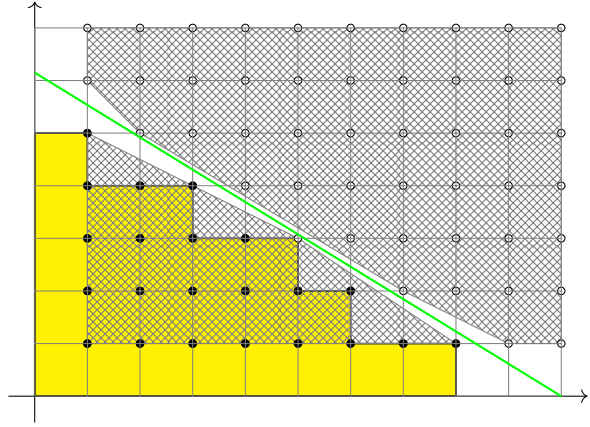

\begin{proposition}
\label{charact_triang2D}
A partition $\lambda$ is triangular if and only if $\Conv(\lambda)\cap\Conv(\bbN^2\setminus\lambda) = \emptyset$.
\end{proposition}

\begin{proof}
If $\lambda$ is triangular, there exist $r,s\in\bbR_{>0}$ so that all the points in $\lambda$ lie on or below the line $\LL_{r,s}$, and all the points in $\bbN^2\setminus\lambda$ lie above this line. It follows that all the points in 
$\Conv(\lambda)$ must lie on or below $\LL_{r,s}$, and all the points in $\Conv(\bbN^2\setminus\lambda)$ must lie above this line. We conclude that the intersection of these two convex hulls is empty.

To prove the converse, suppose that $\Conv(\lambda)\cap\Conv(\bbN^2\setminus\lambda) = \emptyset$ and that $\lambda$ is not empty. Convex hulls are closed sets, and $\Conv(\lambda)$ is bounded, hence it is compact. By the hyperplane separation theorem, two disjoint nonempty closed convex sets, one of which is compact, have a hyperplane separating them. Therefore, there exists a line separating $\lambda$ from its complement, which means that $\lambda$ is triangular.
\end{proof}

In the rest of this section, the term {\em vertex} is used in the sense of a $0$-dimensional face of a polygon; in particular, $\Conv(\tau)$ may have lattice points in its boundary that are not vertices.

If 
$\Conv(\lambda)\cap\Conv(\bbN^2\setminus\lambda)\neq\emptyset$, then one of these two convex hulls must have a vertex that belongs to the other convex hull. Thus, Proposition~\ref{charact_triang2D} is equivalent to the following statement, which will be useful in Section~\ref{sec:triangular_young_poset}.

\begin{lemma}\label{lemma_triangular_iff_vertex_on_cv}
A partition $\lambda$ is triangular if and only if no vertex of $\Conv(\lambda)$ belongs to $\Conv(\bbN^2\setminus\lambda)$ and no vertex of $\Conv(\bbN^2\setminus\lambda)$ belongs to $\Conv(\lambda)$.
\end{lemma}

\subsection{Finding removable and addable cells}\label{sec:finding-addable-removable}

This subsection gives characterizations of removable and addable cells of a triangular partition.

\begin{lemma}
\label{lemma_vertices}
In any triangular partition $\tau$, its removable cells must be vertices of $\Conv(\tau)$, and its addable cells must be vertices of $\Conv(\bbN^2\setminus\tau)$.
\end{lemma}

\begin{proof}
Suppose that $c\in\tau$ is removable, and let $L$ be a cutting line of $\tau\setminus\{c\}$. Then $c$ is the only cell of $\tau$ that lies above $L$, which implies that $c$ is a vertex of $\Conv(\tau)$. 

Similarly, if $c'\in\bbN\setminus\tau$ is addable, let $L'$ be a cutting line of $\tau\cup\{c'\}$. Then $c'$ is the only cell of $\bbN\setminus\tau$ that lies weakly below $L'$, which implies that $c'$ is a vertex of $\Conv(\bbN\setminus\tau)$.
\end{proof}

\begin{proposition}
\label{charact_trcp_tacp}
Two cells in a triangular partition $\tau$ are removable if and only if they are consecutive vertices of $\Conv(\tau)$ and the line passing through them does not intersect $\Conv(\bbN^2\setminus\tau)$.
Similarly, two cells in $\bbN\setminus\tau$ are addable if and only if they are consecutive vertices of $\Conv(\bbN^2\setminus\tau)$ and the line passing through them does not intersect $\Conv(\tau)$.
\end{proposition}

\begin{figure}[ht]
\centering
	\begin{tikzpicture}[scale=.7]
	    
	    \draw[gray, pattern color = gray, pattern=crosshatch] (1,6) -- (3,4) -- (6,2) -- (8,1) -- (9,1) -- (9,7) -- (1,7) -- (1,6);
		\filldraw[color=black, fill=yellow, semithick] (0,0) -- (0,5) -- (1,5) -- (1,4) -- (2,4) -- (2,3) -- (4,3) -- (4,2) -- (5,2) -- (5,1) -- (7,1) -- (7,0) -- (0,0);
        \axes{9}{7}
        \draw[green, thick] (0, 5 + 2/3) -- (8.5, 0);
        \draw[gray, pattern color = gray, pattern=crosshatch] (1,1) -- (1,5) -- (7,1) -- (1,1);
		\draw[green, thick] (0,6) -- (9,0);
        \filldraw[black] (1,5) circle (2.5pt);
        \filldraw[black] (7,1) circle (2.5pt);
        \filldraw[color=black, fill=white] (3,4) circle (2.5pt);
        \filldraw[black=black, fill=white] (6,2) circle (2.5pt);
	\end{tikzpicture}
\caption{In $\tau = 75421$, cells $(1,5)$, $(7,1)$ are removable, and cells $(3,4)$, $(6,2)$ are addable.}
\end{figure}
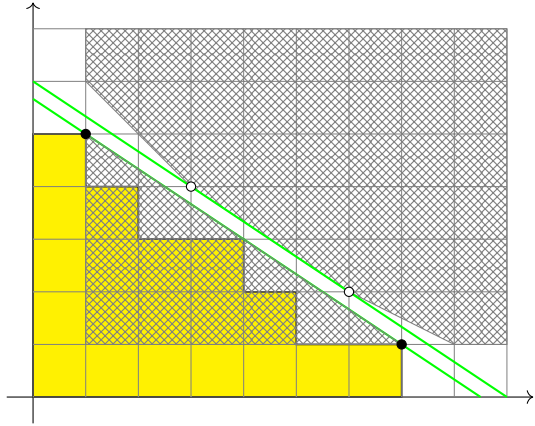

\begin{proof}
We will prove the statement for removable cells. The statement for addable cells can be proved analogously.
Let  $c_1 = (a_1,b_1)$ and $c_2 = (a_2,b_2)\in\tau$ be two different cells with $a_1\leq a_2$, and let $L$ be the line passing through them. 

To prove the forward direction, suppose that $c_1$ and $c_2$ are removable cells. Then, $a_1\neq a_2$, because otherwise the cell with lower $y$-coordinate would not be removable. By Lemma~\ref{lemma_vertices}, both must be vertices of $\Conv(\tau)$.
Suppose for contradiction that they are not consecutive vertices. Then there exists a vertex $c_3 = (a_3,b_3)$ of $\Conv(\tau)$ which lies above $L$, and such that $a_1<a_3<a_2$. 
If $a_2- a_3 \leq a_3 - a_1$, then any line that removes $c_2$ but not $c_3$ must lie above $(2a_3 - a_2, 2b_3 - b_2)$. But this cell is not in $\tau$, because $c_1$ lies strictly below and weakly to the left of it. This contradicts the fact that $c_2$ is removable. 
If, instead, $a_2 - a_3 > a_3 - a_1$, we can similarly reach a contradiction with the fact that $c_1$ is removable. It follows that $c_1$ and $c_2$ must be consecutive vertices of $\Conv(\tau)$.

Next we show that $L$ does not intersect $\Conv(\bbN^2\setminus\tau)$, by arguing that all the points in $\bbN^2\setminus\tau$ lie strictly above $L$.
Indeed, if there was a cell $c'\in\bbN^2\setminus\tau$ to the left of $c_1$ and lying weakly below $L$, then any line that removes $c_2$ but not $c_1$ would cut off a partition that contains $c'$. But that would contradict the fact that $c_2$ is removable. A similar argument shows that there cannot be a cell in $\bbN^2\setminus\tau$ to the right of $c_2$ and lying weakly below $L$. Finally, any point in $\bbN^2$ to the right of $c_1$, to the left of $c_2$ and weakly below $L$ must belong to $\tau$, since $c_1$ and $c_2$ must lie weakly below any cutting line.

To prove the backward direction, suppose that $c_1$ and $c_2$ are consecutive vertices of $\Conv(\tau)$, and that $L$ does not intersect $\Conv(\bbN^2\setminus\tau)$. In particular, $a_1 \neq a_2$, because otherwise $L$ would be vertical and it would intersect $\Conv(\bbN^2\setminus\tau)$.
We have that $\Conv(\tau)$ lies weakly below $L$, and there are no cells of $\tau$ on $L$ to the left of $c_1$ or to the right of $c_2$. Also, by hypothesis, $\Conv(\bbN^2\setminus\tau)$ lies strictly above $L$. Therefore, $L$ is a cutting line for $\tau$. 

Let us show that $c_1$ and $c_2$ are removable. Let $(t, 1 - t)$ be the slope vector of $L$. There exists $\delta > 0$ small enough so that the line passing through $c_1 = (a_1,b_1)$ with slope vector $(t - \delta, 1 - t + \delta)$ is also a cutting line for $\tau$ and does not touch any other point of $\bbN^2$, and there exists a $\varepsilon > 0$ small enough so that the line passing through $(a_1, b_1 - \varepsilon)$ with slope vector $(t - \delta, 1 - t + \delta)$ is a cutting line for $\tau\setminus\{c_1\}$. This proves that $c_1$ is removable, and a similar argument shows that so is $c_2$.
\end{proof}

An immediate consequence of Proposition~\ref{charact_trcp_tacp} is that a triangular partition can have no more than two removable cells and no more than two addable cells, agreeing with Lemma~\ref{lem:removable-addable}.
Note that the cell $(1,1)$ in a triangular partition $\tau$ is removable if and only if $|\tau| = 1$.

\begin{proposition}
\label{charact_orcp_oacp}
A cell $c = (a,b) \neq (1,1)$ in a triangular partition $\tau$ is its only removable cell if and only if it is a vertex of $\Conv(\tau)$ and both of the following conditions hold:
\begin{itemize}
    \item if $a > 1$, the line extending the edge of $\Conv(\tau)$ adjacent to $c$ from the left intersects $\Conv(\bbN^2\setminus\tau)$ to the right of $c$;
    \item if $b > 1$, the line extending the edge of $\Conv(\tau)$ adjacent to $c$ from below intersects $\Conv(\bbN^2\setminus\tau)$ above $c$.
\end{itemize}
The characterization for a cell in $\bbN^2\setminus\tau$ to be the only addable cell of $\tau$ is analogous.
\end{proposition}

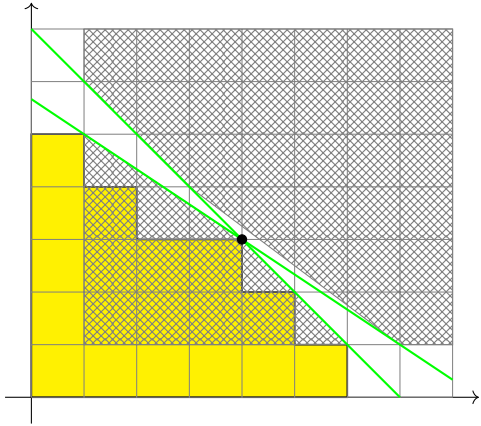
\begin{figure}[ht]
\centering
	\begin{tikzpicture}[scale=.7]
	    \draw[gray, pattern color = gray, pattern=crosshatch] (1,6) -- (3,4) -- (7,1) -- (8,1) -- (8,7) -- (1,7) -- (1,6);
		\filldraw[color=black, fill=yellow, semithick] (0,0) -- (0,5) -- (1,5) -- (1,4) -- (2,4) -- (2,3) -- (4,3) -- (4,2) -- (5,2) -- (5,1) -- (6,1) -- (6,0) -- (0,0);
  \draw[gray, pattern color = gray, pattern=crosshatch] (1,1) -- (1,5) -- (4,3) -- (6,1) -- (1,1);
        \axes{8}{7}
        \draw[green, thick] (0, 5 + 2/3) -- (8, 1/3);
		\draw[green, thick] (0,7) -- (7,0);
        \filldraw[black] (4,3) circle (2.5pt);
	\end{tikzpicture}
\caption{For partition $\tau = 65421$, the only removable cell is $(4,3)$.}
\end{figure}

\begin{proof}
We will prove the statement about removable cells. The statement about addable cells has a similar proof.

To prove the forward direction, suppose that $c$ is the only removable cell of $\tau$. By Lemma~\ref{lemma_vertices}, $c$ is a vertex of $\Conv(\tau)$. By symmetry, it suffices to prove the first condition, namely the one that assumes $a>1$. Let $c'$ be the vertex of $\Conv(\tau)$ adjacent to $c$ from the left. By Proposition \ref{charact_trcp_tacp}, the line through $c$ and $c'$ must intersect $\Conv(\bbN^2\setminus\tau)$, since $c$ is the only removable cell. Let $q$ be a point in this intersection. If $q$ lies in the segment between $c$ and $c'$, then $q\in\Conv(\tau)\cap\Conv(\bbN^2\setminus\tau)$, contradicting Proposition~\ref{charact_triang2D}. If $q$ lies to the left of $c'$, then $c'\in\Conv(\bbN^2\setminus(\tau\setminus\{c\}))$, since $c'$ lies in the segment between $c$ and $q$, but $c'$ is also in the triangular partition $\tau\setminus\{c\}$, again contradicting Proposition~\ref{charact_triang2D}. We deduce that $q$ lies to the right of $c$.

To prove the backward direction, it suffices to show that no cell other than $c$ is removable, since we know by Lemma~\ref{lem:removable-addable} that triangular partitions have at least one removable cell.

Suppose that there is a removable cell $c'$ to the left of $c$. Then $c'$ must lie weakly below the line extending the side of $\Conv(\tau)$ adjacent to $c$ from the left. By hypothesis, there is a point $q\in\Conv(\bbN^2\setminus\tau)$ to the right of $c$ on this line. But then, any cutting line $L$ for $\tau\setminus\{c'\}$ would have to pass below $c'$ and weakly above $c$, and thus also weakly above $q$. Since $q\in\Conv(\bbN^2\setminus\tau)$, there would would be a point in $\bbN^2\setminus\tau$ lying weakly below $L$, which is a contradiction.

Similarly, there cannot be a removable cell to the right of $c$. Finally, any cell of $\tau$ with the same $x$-coordinate as $c$ must have a lower $y$-coordinate because $c$ is a vertex of $\Conv(\tau)$, and thus it cannot be removable.
\end{proof}

For a triangular partition $\tau$, let $t^-\coloneqq t^-_\tau$ and $t^+\coloneqq t^+_\tau$ as defined in Lemma~\ref{charact_bergeron}.
Next we show how to use these values to find the removable cell(s) of $\tau$. Let $m^-=\max_{(a,b)\in\tau} t^-a + (1 - t^-)b$, let $C^-$ be the set of cells in $\tau$ that attain this maximum, and let $L^-$ be the line $t^-x + (1 - t^-)y=m^-$. Define $m^+$, $C^+$ and $L^+$ analogously. Let $c^-$ be the rightmost cell in $C^-$, and let $c^+$ be the uppermost cell in $C^+$. See Figure~\ref{fig:c+c-} for an example.

\begin{figure}[ht]
\centering
	\begin{tikzpicture}[scale=.8]
		\filldraw[color=black, fill=yellow, semithick] (0,0) -- (0,5) -- (1,5) -- (1,4) -- (2,4) -- (2,3) -- (3,3) -- (3,2) -- (5,2) -- (5,1) -- (6,1) -- (6,0) -- (0,0);
        \axes{8}{7}
        \draw[blue, thick] (0, 5 + 2/3) -- (8, 1/3);
		\draw[red, thick] (0,7) -- (7,0);
		\draw[blue, ->, thick] (4,3) -- (4.4, 3.6) node[above, scale=.8]{$(t^-, 1 - t^-)$};
		\draw[red, ->, thick] (4,3) -- (4.5, 3.5) node[right, scale=.8]{$(t^+, 1 - t^+)$};
        \filldraw[black] (1,5) circle (2.5pt) node[below left, outer sep=-1]{$c^-$};
        \filldraw[black] (5,2) circle (2.5pt) node[below left, outer sep=-1]{$c^+$};
        \filldraw[black] (6,1) circle (2.5pt);
        \draw[blue] (6.7,1.8) node {$L^-$};
        \draw[red] (1.9,6) node {$L^+$};
	\end{tikzpicture}
\caption{For $\tau = 65321$, we have $C^- = \{(1,5)\}$, $C^+ = \{(5,2),(6,1)\}$, and so $c^- = (1,5)$ and $c^+ = (5,2)$.}
\label{fig:c+c-}
\end{figure}
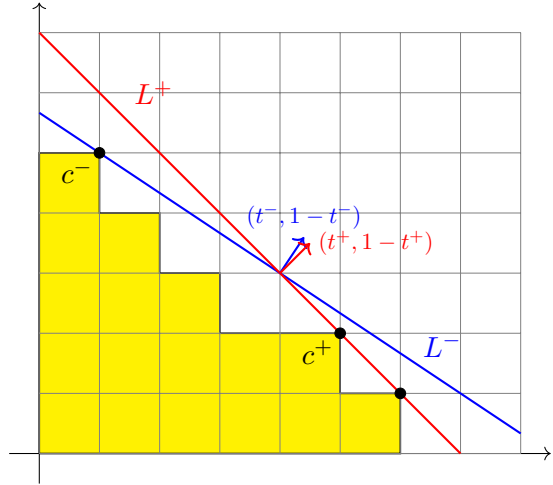

\begin{proposition}
\label{c+c-}
Let $\tau$ be a triangular partition and let $c^-$ and $c^+$ be as defined above. If $c^- = c^+$, then this is the only removable cell of~$\tau$. If $c^-\neq c^+$, then both of these cells are removable.
\end{proposition}

\begin{proof}
We have $c^-\in L^-$ by construction. If there was a point in $\bbN^2\setminus\tau$ lying weakly below $L^-$ and to the left of $c^-$, then any cutting line for $\tau$ would have a slope vector $(t, 1 - t)$ with $t \leq t^-$, in contradiction with Lemma~\ref{charact_bergeron}. Moreover, no point in $\bbN^2\setminus\tau$ can lie strictly below $L^-$ and to the right of $c^-$, because by the same lemma, there is a cutting line with slope vector $(t, 1 - t)$ for any $t^-<t<t^+$.

As a consequence, there exist $\delta,\varepsilon > 0$ small enough so that the line passing through $c^- -(0,\delta)$ with slope vector $(t^- + \varepsilon, 1 - t^- - \varepsilon)$ is a cutting line for $\tau\setminus\{c^-\}$. Therefore, $c^-$ is removable.
An analogous argument shows that $c^+$ is removable as well.

Finally, let us show that if $c^- = c^+$, then there are no more removable cells. Suppose that there is another removable cell $c$ to the left of $c^-$, and let $L$ be the line through $c$ and $c^-$, which must be a cutting line by Proposition~\ref{charact_trcp_tacp}. Since $c$ lies weakly below $L^-$, the slope vector $(t, 1 - t)$ of $L$ satisfies $t\leq t^-$, contradicting Lemma~\ref{charact_bergeron}. A symmetric argument shows that there is no removable cell to the right of $c^+$ either.
\end{proof}

\subsection{An algorithm to determine triangularity}\label{sec:triangularity_algorithm}

Next we consider the problem of determining whether a given partition $\lambda$ is triangular.
A method for this was given by Bergeron and Mazin~\cite{Bergeron2023}, as described in Lemma~\ref{charact_bergeron}, but it requires computing two values $t_\lambda^-$ and $t_\lambda^+$ for each cell in $\lambda$. Next we present a more efficient method, which also yields the removable and addable cells when the partition is triangular. We start with some preliminary results.

Let $\lambda= \lambda_1\dots\lambda_k$ be a partition. A cell in $\lambda$ is called a \emph{corner cell} if removing it from $\lambda$ yields a partition, and a cell in $\bbN^2\setminus\lambda$ is called a \emph{complementary corner cell} if adding it to $\lambda$ yields a partition. Equivalently, a corner cell is of the form 
$(\lambda_i,i)$ with either $i = k$ or $\lambda_i > \lambda_{i + 1}$; and a complementary corner cell is of the form $(\lambda_1 + 1,1)$, $(1, k + 1)$, or $(\lambda_i + 1, i)$ for $2\leq i\leq k$ such that $\lambda_{i - 1}>\lambda_i$. For any $\lambda$, the number of complementary corner cells is one more than the number of corner cells, which in turn equals the number of distinct parts of $\lambda$. Let $\cC(\lambda)$ be set of corner cells of $\lambda$, together with the cells $(1,1),(\lambda_1,1),(1,k)$. Let $\cC'(\lambda)$ be set of complementary corner cells of $\lambda$, together with the cell $(\lambda_1 + 1, k + 1)$. 
The next lemma will help us find the vertices of $\Conv(\lambda)$ and $\Conv(\bbN^2\setminus\lambda)$ efficiently.

\begin{lemma} \label{lemma-corner-cells}
The following hold:
\begin{enumerate}[(a)]
    \item $\Conv(\lambda) = \Conv(\cC(\lambda))$, 
    \item the vertices of $\Conv(\bbN^2\setminus\lambda)$ are those of $\Conv(\cC'(\lambda))$ except for $(\lambda_1 + 1, k + 1)$.
    \end{enumerate}
\end{lemma}

\begin{proof}
    Clearly $\cC(\lambda)\subseteq\lambda$, so $\Conv(\cC(\lambda))\subseteq\Conv(\lambda)$. For the reverse inclusion, note that every $(a,b)\in\lambda$ that is not in $\cC(\lambda)$ can be expressed as a convex combination of other cells in $\lambda$, since either $(a-1,b),(a+1,b)\in\lambda$ or $(a,b-1),(a,b+1)\in\lambda$. Thus, the vertices of $\Conv(\lambda)$ must be in $\cC(\lambda)$, and so $\Conv(\lambda)\subseteq\Conv(\cC(\lambda))$, proving $(a)$.

    To prove $(b)$, let us first argue that every vertex of $\Conv(\bbN^2\setminus\lambda)$ must also be a vertex of $\Conv(\cC'(\lambda))$. Indeed, if $c$ is a vertex of $\Conv(\bbN^2\setminus\lambda)$, there is a line through $c$ that leaves the rest of $\bbN^2\setminus\lambda$ strictly on one side; in particular, it leaves the rest of $\cC'(\lambda)$ strictly on one side. Thus $c$ is a vertex of $\Conv(\cC'(\lambda))$. Also, it is clear that $c\neq (\lambda_1 + 1, k + 1)$, since this point is not a vertex of $\Conv(\bbN^2\setminus\lambda)$.

    Finally, let us show that every vertex $c$ of $\Conv(\cC'(\lambda))$, other than $(\lambda_1 + 1, k + 1)$, is also a vertex of $\Conv(\bbN^2\setminus\lambda)$. This is clearly true for the vertices $(1,k+1)$ and $(\lambda_1+1,1)$, so suppose that $c=(a,b)$ with $a \le \lambda_1$ and $b \le k$. Then there there exists a line $L$ through $c$ that leaves the rest of $\cC'(\lambda)$ strictly on one side; more specifically, this line  must have negative slope and leave the rest of $\cC'(\lambda)$ strictly above it. To conclude that $c$ is a vertex of $\Conv(\bbN^2\setminus\lambda)$, it suffices to show that the rest of $\bbN^2\setminus\lambda$ also lies strictly above $L$. Indeed, if this was not the case, we could find some $c'\in\bbN^2\setminus\lambda$ with $c'\neq c$, lying weakly below $L$, and such that its sum of coordinates is smallest among all the cells with this property. Then $c'$ would be a complementary corner cell, in contradiction with the fact that all cells of $\cC'(\lambda)$ other than $c$ lie strictly above $L$.
\end{proof}

The following lemma will help us implement a binary search for removable cells.

\begin{lemma}
    \label{lemma-binary-search}
    Let $\tau$ be a triangular partition. Let $c_1,c_2\neq(1,1)$ be consecutive vertices of $\Conv(\tau)$, with $c_1$ to the left of $c_2$, and let $L$ be the line through $c_1$ and $c_2$. If there is some $c_3\in\bbN^2\setminus\tau$  lying weakly below $L$ and to the left of $c_1$ (resp.\ right of $c_2$), then any removable cells of $\tau$ 
    lie weakly to the left of $c_1$ (resp.\ right of $c_2$).

    Similarly, let $c'_1,c'_2$ be consecutive vertices of $\Conv(\bbN^2\setminus\tau)$, with $c'_1$ to the left of $c'_2$, and let $L'$ be the line through $c'_1$ and $c'_2$. If there is some $c'_3\in\tau$ lying weakly above $L'$ and to the left of $c'_1$ (resp.\ right of $c'_2$), then any addable cells 
    lie weakly to the left of $c'_1$ (resp.\ right of $c'_2$). 
\end{lemma}

\begin{proof}    
    Suppose that $c_4\in\tau$ is a removable cell other than $c_1$. Any cutting line for $\tau\setminus\{c_4\}$ must pass strictly below $c_3$ and $c_4$ and weakly above $c_1$. Since both $c_3$ and $c_4$ lie weakly below $L$, this implies that $c_4$ lies to the same side of $c_1$ as $c_3$. A similar argument proves the statement for $c_3$ to the right of $c_2$, as well as for addable cells.
\end{proof}

Our algorithm to determine if a partition $\lambda$ is triangular starts by finding its corner cells. Then, it computes $\Conv(\lambda)$ and $\Conv(\bbN^2\setminus\lambda)$, and it performs a binary search on the edges of the boundary of $\Conv(\lambda)$, using Proposition~\ref{charact_trcp_tacp} to look for a pair of removable cells. For each edge, it tries to find a point in $\bbN^2\setminus\lambda$ that lies below the line extending the edge, in order to apply Lemma~\ref{lemma-binary-search} and keep searching in the correct direction. If no pair of removable cells is found, the same procedure is applied to addable cells. Below is a more detailed description.

\begin{algorithm}[to determine whether a partition $\lambda = \lambda_1\dots\lambda_k$ is triangular]
\label{alg:triangular}
\begin{enumerate}
    \item Read the input $\lambda$ and record the cells in $\cC(\lambda)$ in counter-clockwise order starting at $(1,1)$, as well as the cells in $\cC'(\lambda)$ in clockwise order starting at $(\lambda_1 + 1, k + 1)$. Use Graham's scan~\cite{Graham1972} to find the convex hulls of $\cC(\lambda)$ and $\cC'(\lambda)$, which, by Lemma~\ref{lemma-corner-cells}, yield the vertices of $\Conv(\lambda)$ and $\Conv(\bbN^2\setminus\lambda)$. Let $w = (w_1,\dots,w_i)$ and $w' = (w'_1,\dots,w'_j)$ be the resulting lists of vertices, after deleting vertex $(1,1)$ and vertex $(\lambda_1 + 1, k + 1)$, respectively.
    \item If $w$ has at least two elements, let $c_1 = w_{\lfloor i/2 \rfloor + 1}$, $c_2 = w_{\lfloor i/2 \rfloor}$, and let $L$ be the line through $c_1$ and $c_2$. The signed distances of the vertices in $w'$ to $L$ (considering distances to be negative for vertices below $L$) form an upside-down unimodal sequence. Perform a ternary search to find the vertex $c'$ with the minimum signed distance.
    \begin{enumerate}
        \item If $c'$ lies weakly below $L$ and  to the left of $c_1$, then set $w=(w_{\lfloor i/2 \rfloor + 1},\dots,w_i)$ and repeat step~2.
        \item If $c'$ lies weakly below $L$ and to the right of $c_2$, then set $w=(w_1,\dots,w_{\lfloor i/2 \rfloor})$ and repeat step~2.
        \item If $c'$ lies weakly below $L$, to the left of $c_1$ and to the right of $c_2$, then $\lambda$ is not triangular because $c'\in\Conv(\lambda)$ (using Lemma~\ref{lemma_triangular_iff_vertex_on_cv}), and the algorithm ends.
        \item If $c'$ lies strictly above $L$, then $L$ is a cutting line for $\lambda$. In this case, $\lambda$ is triangular, and by Proposition~\ref{charact_trcp_tacp}, $c_1$ and $c_2$ are its removable cells. Check if the edge between $c'$ and the previous vertex in $w'$ is parallel to $L$, and similarly for the edge between $c'$ and the next vertex in $w'$. If so, then the two endpoints of the edge are the addable cells,
        by Proposition~\ref{charact_trcp_tacp}. If neither of these segments is parallel to $L$, then there is a line parallel to $L$ that cuts off $\lambda\cup\{c'\}$. Thus, $c'$ is an addable cell and, by Lemma~\ref{parallel_lines_removable_addable_cells}, it is the only one. We have found the addable and removable cells, and the algorithm ends.
    \end{enumerate}
    \item This step is reached when $\lambda$ has no pair of removable cells. In this case, we proceed analogously to step 2 but with the list $w'$ instead of $w$, in order to find a pair of addable cells and then a removable cell, as in step 2(d). If this step is not reached, then $\lambda$ is not triangular, by Lemma~\ref{lem:removable-addable}.
\end{enumerate}
\end{algorithm}

Let $\lambda$ be a partition of $n$ into $k$ parts, and let $m$ its number of distinct parts, which is $\bigO(\min\{k,\sqrt{n}\})$.
The complexity of finding the corner cells in step~1 is $\bigO(k)$.
Graham's scan runs in linear time in $m$, because corner cells (resp.\ complementary corner cells) are found in counter-clockwise (resp.\ clockwise) order with respect to $(1,1)$ (resp.\ $(\lambda_1 + 1, k + 1)$). Finally, steps~2 and~3 perform a binary search where each step runs a ternary search, and so they run in time $\bigO((\log m)^2)$. 

In summary, our algorithm takes time $\bigO(k)$ to read the input, and time $\bigO(m)$ to run the rest of the steps, while the space used is $\bigO(m)$.
For comparison, an algorithm based on Lemma~\ref{charact_bergeron} would run in time $\bigO(n)$ while occupying $\bigO(k)$ space.

\section{The triangular Young poset} \label{sec:triangular_young_poset} 

Bergeron and Mazin \cite{Bergeron2023} introduced the poset $\TYP$ of triangular partitions ordered by containment of their Young diagrams. An interesting property of this poset, as explained in Section~\ref{subsection:bergeron}, is that it has a planar Hasse diagram. This property is used in \cite{Bergeron2023} to deduce that $\TYP$ is a lattice, and it is ranked by the size of each partition. In this section we describe the M\"obius function of $\TYP$, and we give explicit constructions for the meet and the join of any two elements.

Our first result confirms Bergeron's conjecture (personal communication, 2022) that the M\"obius function, which we denote by $\mu$, only takes values in $\{-1,0,1\}$.

\begin{theorem}
\label{Mobius} 
Let $\tau,\nu\in\TYP$ such that $\tau\leq\nu$. Then
$$
\mu(\tau, \nu) = \begin{cases}
1 & \text{if either $\tau = \nu$ or there exist $\zeta^1\neq\zeta^2$ such that $\nu = \zeta^1\lor\zeta^2$ and $\tau\lessdot\zeta^1,\zeta^2$},\\
-1 &\text{if $\tau\lessdot\nu$},\\
0 &\text{otherwise.}
\end{cases}
$$
\end{theorem}

\begin{proof}
Trivially, if $\tau = \nu$, then $\mu(\tau,\nu) = 1$, and if $\tau\lessdot\nu$, then $\mu(\tau,\nu) = -1$. 
Otherwise, consider two possibilities depending on how many elements in the interval $[\tau,\nu]$ cover $\tau$. This number has to be one or two, since $\tau$ has at most two addable cells by Lemma~\ref{lem:removable-addable}.

If only one element $\zeta\in[\tau,\nu]$ covers $\tau$, we will show that $\mu(\tau,\nu) = 0$ by induction on $m=|\nu|-|\zeta|$. If $m=1$, then $\zeta\lessdot\nu$, and $\mu(\tau,\nu) = -\mu(\tau,\tau) - \mu(\tau,\zeta) = -1 + 1 = 0$. 
If $m>1$, then
$$\mu(\tau,\nu) = -\sum_{\tau\leq\theta<\nu}\mu(\tau,\theta) = -(1 - 1 + 0 + \dots  + 0) = 0,
$$
using the induction hypothesis.

If there are two elements $\zeta^1,\zeta^2\in[\tau, \nu]$ that cover $\tau$, let $\zeta=\zeta^1\lor\zeta^2$. If $\zeta=\nu$, then any $\theta\in[\tau,\nu]$ with $\theta<\nu$ falls in one of the above cases, and so
$$
\mu(\tau,\nu) = -\sum_{\tau\leq\theta<\nu}\mu(\tau,\theta) = -(1 - 1 - 1 + 0 + \dots  + 0) = 1.
$$
If $\zeta<\nu$, we will show that $\mu(\tau,\nu) = 0$ by induction on $m=|\nu|-|\zeta|$. If $m=1$, then $\zeta\lessdot\nu$, and
$$
\mu(\tau, \nu) = -\sum_{\tau\leq\theta<\nu}\mu(\tau,\theta) = -(1 - 1 - 1 + 0 + \dots  + 0 + 1) = 0.
$$
If $m>1$, then
$$\mu(\tau,\nu) = -\sum_{\tau\leq\theta<\nu}\mu(\tau,\theta) = -(1 - 1 - 1 + 0 + \dots  + 0 + 1 + 0 + \dots  + 0) = 0,
$$
using the induction hypothesis.
\end{proof}

As mentioned below Definition~\ref{def:diagonal}, the faces of the Hasse diagram of $\TYP$ are polygons with an even number of sides. We can interpret Theorem~\ref{Mobius} as stating that, if $\tau < \nu$ and $\nu$ does not cover $\tau$, then
$\mu(\tau, \nu)$ equals $1$ if $[\tau,\nu]$ is one of the polygonal faces (equivalently, if $\tau = \nu^\circ$), and $0$ otherwise.

Our next result explicitly characterizes the join and meet of elements of $\TYP$. See Figure~\ref{fig:join_meet} for an example.

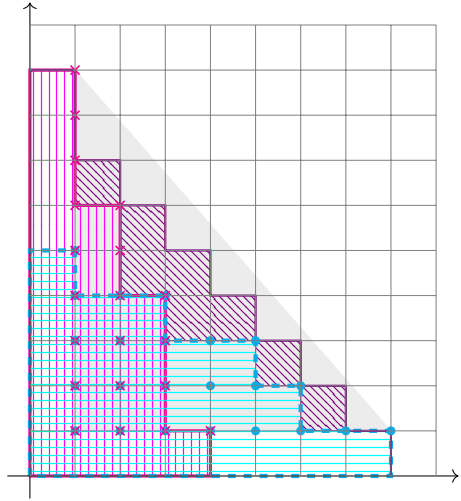
\begin{figure}[ht]
\centering
    \usetikzlibrary {patterns,patterns.meta}
	\begin{tikzpicture}[scale=.6]
 \filldraw[gray!15, very thin] (1,1) -- (1,9) -- (8,1) --  (1,1);
	    \draw[violet, pattern={north west lines},pattern color=violet] (1,7) -- (2,7) -- (2,6) -- (3,6) -- (3,5) -- (4,5) -- (4,4) -- (5,4) -- (5,3) -- (6,3) -- (6,2) -- (7,2) -- (7,1) -- (6,1) -- (6,2) -- (5,2) -- (5,3) -- (3,3) -- (3,4) -- (2,4) -- (2,6) -- (1,6) -- (1,7);
		\draw[magenta, very thick, pattern={vertical lines},pattern color=magenta]
        (0,0) -- (0,9) -- (1,9) -- (1,6) -- (2,6) -- (2,4) -- (3,4) -- (3,1) -- (4,1) -- (4,0) -- (0,0);
        \draw[cyan, dashed, ultra thick, pattern={horizontal lines},pattern color=cyan] (0,0) -- (0,5) -- (1,5) -- (1,4) -- (3,4) -- (3,3) -- (5,3) -- (5,2) -- (6,2) -- (6,1) -- (8,1) -- (8,0) -- (0,0);
        \draw[violet, thick] (0,9) -- (1,9) -- (1,7) -- (2,7) -- (2,6) -- (3,6) -- (3,5) -- (4,5) -- (4,4) -- (5,4) -- (5,3) -- (6,3) -- (6,2) -- (7,2) -- (7,1) -- (8,1) -- (8,0);
        \foreach \c in {(1,1),(2,1),(3,1),(4,1),(5,1),(6,1),(7,1),(8,1),(1,2),(2,2),(3,2),(4,2),(5,2),(6,2),(1,3),(2,3),(3,3),(4,3),(5,3),(1,4),(2,4),(3,4),(1,5)} 
            {\filldraw[cyan] \c circle (2.5pt); }
        \foreach \c in {(1,1),(2,1),(3,1),(4,1),(1,2),(2,2),(3,2),(1,3),(2,3),(3,3),(1,4),(2,4),(3,4),(1,5),(2,5),(1,6),(2,6),(1,7),(1,8),(1,9)} 
            {\draw[semithick, magenta] \c+(-.1,-.1) -- ++(.1,.1);
            \draw[semithick, magenta] \c+(-.1,.1) -- ++(.1,-.1);}
        \axes{9}{10}
	\end{tikzpicture}
\caption{The join of two partitions: $86531\lor433322111 = 876543211$.}
\label{fig:join_meet}
\end{figure}

\begin{proposition}
\label{triang_join_meet_construction}
For any $\tau,\nu\in\TYP$, we have
$$
    \tau\lor\nu = \bbN^2\cap\Conv(\tau\cup\nu) \quad\text{and}\quad
    \tau\land\nu = \bbN^2\setminus\Big(\bbN^2\cap\Conv\big(\bbN^2\setminus(\tau\cap\nu)\big)\Big).
$$
\end{proposition}

\begin{proof}
We will prove the statement for $\tau\lor\nu$. The statement for $\tau\land\nu$ has a similar proof.

Let $\zeta =\bbN^2\cap\Conv(\tau\cup\nu)$. 
The partition $\tau\lor\nu$ is triangular, so it consists of the points in $\bbN^2$ weakly below some cutting line. Therefore, since $\tau\lor\nu$ contains $\tau$ and $\nu$, it must also contain every lattice point that is a convex combination of points in $\tau$ and $\nu$. It follows that $\zeta \subseteq \tau\lor\nu$. 

On the other hand, it is clear that $\tau,\nu\subseteq\zeta$. To prove that $\tau\land\nu\subseteq\zeta$, it suffices to show that $\zeta$ is triangular. By Lemma~\ref{lemma_triangular_iff_vertex_on_cv}, this will follow if we show that no vertex of $\Conv(\zeta)$ is in $\Conv(\bbN^2\setminus\zeta)$ and vice versa.

Suppose that there is a vertex $c$ of $\Conv(\bbN^2\setminus\zeta)$ (which must be a point in $\bbN^2\setminus\zeta$) such that $c\in\Conv(\zeta)$. Note that $\Conv(\zeta) = \Conv(\tau\cup\nu)$, so $c$ is a lattice point in $\Conv(\tau\cup\nu)$, implying that $c\in\zeta$, which is a contradiction.

Suppose now that there is a vertex $c$ of $\Conv(\zeta)$ such that $c\in\Conv(\bbN^2\setminus\zeta)$. Since $\Conv(\zeta) = \Conv(\tau\cup\nu)$, every vertex must be a point in $\tau\cup\nu$, therefore either $c\in\tau$ or $c\in\nu$. Let us assume that $c\in\tau$ without loss of generality. Since $\tau\subseteq\zeta$, we have $\Conv(\bbN^2\setminus\tau)\supseteq\Conv(\bbN^2\setminus\zeta)$, and so $c\in\Conv(\bbN^2\setminus\tau)$. But this contradicts that $\tau$ is triangular.
\end{proof}

By repeatedly applying Proposition~\ref{triang_join_meet_construction}, one can show that, for any set of triangular partitions $\tau^1,\dots,\tau^k$, we have 
$$\tau^1\lor\dots \lor\tau^k = \bbN^2\cap\Conv(\tau^1\cup\dots \cup\tau^k) \quad\text{and}\quad \tau^1\land\dots \land\tau^k = \bbN^2\setminus\Big(\bbN^2\cap\Conv\big(\bbN^2\setminus(\tau^1\cap\dots \cap\tau^k)\big)\Big).$$

Recall that a non-minimum element in a lattice is called {\em join-irreducible} if it cannot be expressed as the join of other elements.

\begin{corollary}
The join-irreducible elements of $\TYP$ are the triangular partitions with one removable cell.
\end{corollary}

\begin{proof}
A triangular partition with one removable cell covers only one element in $\TYP$, so it cannot be the join of two elements other than itself. On the other hand, a triangular partition $\tau$ with two removable cells is the join of the two triangular partitions that it covers.
\end{proof}

\section{Bijections to balanced words and efficient generation} \label{sec:sturmian} 

In this section we will present two different encodings of triangular partitions in terms of factors of Sturmian words. The first one, which is hinted in~\cite{Bergeron2023}, is quite natural, and it will allow us to prove some enumeration formulas in Section~\ref{sec:subpartitions}.
The second one encodes families of triangular partitions by one single balanced word, along with two other parameters, and it will be used in Section~\ref{sec:algorithm} to implement efficient algorithms to count triangular partitions by their size.

\subsection{Sturmian words}

Sturmian words have applications in combinatorics, number theory, and dynamical systems; see \cite[Chapter 2]{Lothaire2002} for a thorough study. In the following definition, a word $w$ is called a \emph{factor} of another word $s$ if $s = uwt$ for some words $u$ and~$t$. 

\begin{definition}
    An infinite binary word $s$ is {\em Sturmian} if, for every $n\ge1$, the number of factors of $s$ of length $n$ equals $n + 1$.
\end{definition}

We will be interested in factors of Sturmian words, which have the following two useful characterizations.

\begin{proposition}[\cite{Lothaire2002}]\label{characterizations_of_sturmian_words}
    Let $w = w_1\dots w_\ell$ be a finite binary word over $\{0,1\}$. The following statements are equivalent:
    \begin{enumerate}[(a)]
        \item $w$ is a factor of some Sturmian word;
        \item $w$ is a \emph{balanced} word, that is, for any $h \leq \ell$ and $i,j \leq \ell - h + 1$, we have 
        \begin{equation}\label{eq:balanced}
        \left|\sum_{t = i}^{i + h - 1}w_t - \sum_{t = j}^{j + h - 1}w_t\right| \leq 1;
        \end{equation}
        \item $w$ is a \emph{mechanical} word, that is, there exist real numbers $0 < \alpha, \beta < 1$ such that $w_i = \lfloor i\alpha + \beta \rfloor - \lfloor (i - 1)\alpha + \beta \rfloor$ for $1\le i\le \ell$.
    \end{enumerate}
\end{proposition}

We denote by $\cB$ the set of all words satisfying any of the above equivalent conditions, and by $\cB_\ell$ the set of those of length $\ell$.
Condition~\eqref{eq:balanced} states that, for any two consecutive subwords of $w$ of the same length, the number of ones in these subwords differs by at most~$1$.

Visualizing $w$ as a lattice path $P$ that starts at the origin and has $i$th step $(1,0)$ if $w_i = 0$, and $(1,1)$ if $w_i = 1$, the condition of $w$ being a mechanical word is equivalent to $P$ being the highest path with $\ell$ steps $(1,0)$ and $(1,1)$ starting at the origin and staying weakly below the line $y = \alpha x + \beta$ (see Figure \ref{fig:mechanical_balanced}). 

\begin{figure}[ht]
\centering
	\begin{tikzpicture}[scale=.7]
        \axes{9}{7};
        \draw[brown, thick] (9,6.15) -- (0,615/1100);
        \draw[thick] (0,0) -- (1,1) -- (2,1) -- (4,3) -- (5,3) -- (6,4) -- (7,4) -- (8,5);
        \draw[very thick] (-.1, 615/1100)--++(.2,0);
        \node at (-0.5,615/1100) {$\beta$};
        \node[brown] at (3.5,4.5) {$y=\alpha x+\beta$};
        \node at (7.5,-0.4) {$1$};
        \node at (6.5,-0.4) {$0$};
        \node at (5.5,-0.4) {$1$};
        \node at (4.5,-0.4) {$0$};
        \node at (3.5,-0.4) {$1$};
        \node at (2.5,-0.4) {$1$};
        \node at (1.5,-0.4) {$0$};
        \node at (0.5,-0.4) {$1$};
        \filldraw[black] (0,0) circle (1.5pt);
        \filldraw[black] (1,1) circle (1.5pt);
        \filldraw[black] (2,1) circle (1.5pt);
        \filldraw[black] (3,2) circle (1.5pt);
        \filldraw[black] (4,3) circle (1.5pt);
        \filldraw[black] (5,3) circle (1.5pt);
        \filldraw[black] (6,4) circle (1.5pt);
        \filldraw[black] (7,4) circle (1.5pt);
        \filldraw[black] (8,5) circle (1.5pt);
	\end{tikzpicture}
\caption{The mechanical word $10110101\in\cB$ viewed as a lattice path.}
\label{fig:mechanical_balanced}
\end{figure}

The following enumeration formula for balanced words is due to Lipatov~\cite{Lipatov1982}. Throughout the paper, we use $\varphi$ to denote Euler's totient function.

\begin{theorem}[\cite{Lipatov1982}]\label{lipatov}
The number of balanced words of length $\ell$ is
$$
|\cB_\ell|=1 + \sum_{i = 1}^\ell (\ell - i + 1)\varphi(i).
$$
\end{theorem}

\subsection{Wide and tall partitions}

The following definition will be useful for our encodings of triangular partitions by balanced words.

\begin{definition}
A triangular partition is \emph{wide} (respectively \emph{tall}) if it admits a cutting line $\LL_{r,s}$ with $r > s$ (respectively $r < s$).
Denote the set of wide triangular partitions by $\Wide$.
\end{definition}

It follows from Lemma~\ref{charact_bergeron}, and in fact was already noted in~\cite{Corteel1999}, 
that the slopes of the cutting lines of any given triangular partition form an open interval.
Thus, every triangular partition must be wide, tall, or both.

It is immediate from the definition that a triangular partition $\tau$ is wide if and only if its conjugate $\tau'$ is tall. This is because $\LL_{r,s}$ is a cutting line for $\tau$ if and only if $\LL_{s,r}$ is a cutting line for $\tau'$.
Next we give alternative characterizations of wide partitions. We use the notation $[k]=\{1,2,\dots,k\}$.

\begin{lemma}
\label{lem:wide}
For any triangular partition $\tau = \tau_1\dots\tau_k$, the following are equivalent:
\begin{enumerate}[$(a)$]
    \item $\tau$ is wide,
    \item $\tau_1 \ge k$,
    \item the parts of $\tau$ are distinct.
\end{enumerate}
Separately, the following are equivalent:\begin{enumerate}[$(a')$]
    \item $\tau$ is wide and tall,
    \item $\tau_1=k$,
    \item $\tau$ is the staircase partition $\sigma^k$.
\end{enumerate}
\end{lemma}

\begin{proof}
    Let us start by proving $(b)\Rightarrow(c)$ by contrapositive. If $(c)$ does not hold, there exists $i\in[k-1]$ such that $\tau_i=\tau_{i+1}$. Let $j$ denote this value. Any cutting line for $\tau$ must pass weakly above $(j,i+1)\in\tau$ and strictly below $(j+1,i)\notin\tau$. It follows that it must pass above $(1,i+j)$ and below $(i+j,1)$, so $(1,i+j)\in\tau$ but $(i+j,1)\notin\tau$. We conclude that $k\ge i+j>\tau_1$, contradicting $(b)$.

    Next we prove that $(b')\Rightarrow(c')$. Suppose that $\tau_1=k=\tau_1'$. Applying the fact that $(b)\Rightarrow(c)$ to $\tau$ and to $\tau'$, it follows that both $\tau$ and $\tau'$ are partitions into distinct parts. This implies that $\tau$ is a staircase partition, since the upper-right boundary of its Young diagram cannot have two consecutive steps in the same direction. 
    
    To prove $(c)\Rightarrow(b)$, note that if the parts of $\tau$ are distinct, then $\tau_1 = \tau_k + \sum_{i = 1}^{k - 1}(\tau_i - \tau_{i + 1}) \ge k$. The implication $(c')\Rightarrow(b')$ is trivial.

    To prove $(a)\Rightarrow(b)$ and $(a')\Rightarrow(b')$, let $\LL_{r,s}$ be a cutting line for $\tau$. By Definition~\ref{def:triangular}, $\tau_1 = \lfloor r - r/s \rfloor$ and $k = \lfloor s - s/r \rfloor$. It follows that, if $r>s$, then $\tau_1\ge k$, and if $r<s$, then $\tau_1\le k$. In particular, if $\tau$ is wide, then $\tau_1\ge k$, and if $\tau$ is wide and tall, then $\tau_1=k$.

    Next we prove that $(b)\Rightarrow(a)$ and $(b')\Rightarrow(a')$. Suppose first that $\tau_1=k$. Then $\tau=\sigma^k$, which is wide and tall, so both $(a)$ and $(b)$ hold.
    
    Suppose now that $\tau_1>k=\tau_1'$. We claim that $\tau'$ cannot be wide in this case. Indeed, if it was, then the fact that $(a)\Rightarrow(b)$ applied to $\tau'$ would imply that $\tau'_1\ge\tau_1$, which is a contradiction. Since $\tau'$ is not wide, it must be tall, so $(a)$ holds. 
\end{proof}

\subsection{First Sturmian encoding}
\label{subsection:first-sturmian-encoding}

We are now ready to describe the first encoding of wide triangular partitions by balanced words. Given $\tau = \tau_1\dots\tau_k\in\Wide$, define the binary word
\begin{equation}\label{def:w}
\omega(\tau)=10^{\tau_1-\tau_2-1}10^{\tau_2-\tau_3-1}\dots10^{\tau_{k-1}-\tau_k-1}10^{\tau_k - 1}.
\end{equation}
The fact that all the parts of $\tau$ are distinct (by Lemma~\ref{lem:wide}) guarantees that the exponents are nonnegative. For example, we have $\omega(86531)=10110101$, as illustrated in Figure~\ref{fig:triangular-balanced}.

\begin{figure}[ht]
\centering
	\begin{tikzpicture}[scale=.7]
		\filldraw[color=black, fill=yellow, semithick] (0,0) -- (0,5) -- (1,5) -- (1,4) -- (3,4) -- (3,3) -- (5,3) -- (5,2) -- (6,2) -- (6,1) -- (8,1) -- (8,0) -- (0,0);
        \axes{10}{7}
        \draw[green, thick] (0,6.15) -- (9.9,0);
                \draw[very thick] (.1,6.15)--(-.1,6.15) node[left]{$s$};
         \draw[very thick] (9.9,.1)--(9.9,-.1) node[below]{$r$};
        \draw[black, semithick] (9,0) -- (8,1) -- (7,1) -- (5,3) -- (4,3) -- (3,4) -- (2,4) -- (1,5);
        \draw[black, semithick] (9 - 1/10, 615/1100) -- (9 + 1/10, 615/1100);
        \draw[black, semithick] (9 + 2.5/10, 615/1100) -- (9 + 4.5/10, 615/1100);
        \draw[black, semithick] (9 + 6/10, 615/1100) -- (9 + 8/10, 615/1100);
        \draw[black, semithick] (10 - 0.5/10, 615/1100) -- (10 + 1/10, 615/1100);
        \node at (10.5,615/1100) {$\beta$};
        \draw[black, semithick] (3/10, 328/55) -- (3/10, 6150/1100) -- (9/10, 6150/1100);
        \node at (-0.1,5.4) {$-\alpha$};
        \node at (1.5,-0.4) {$1$};
        \node at (2.5,-0.4) {$0$};
        \node at (3.5,-0.4) {$1$};
        \node at (4.5,-0.4) {$0$};
        \node at (5.5,-0.4) {$1$};
        \node at (6.5,-0.4) {$1$};
        \node at (7.5,-0.4) {$0$};
        \node at (8.5,-0.4) {$1$};
        \filldraw[black] (9,0) circle (1.5pt);
        \filldraw[black] (8,1) circle (1.5pt);
        \filldraw[black] (7,1) circle (1.5pt);
        \filldraw[black] (6,2) circle (1.5pt);
        \filldraw[black] (5,3) circle (1.5pt);
        \filldraw[black] (4,3) circle (1.5pt);
        \filldraw[black] (3,4) circle (1.5pt);
        \filldraw[black] (2,4) circle (1.5pt);
        \filldraw[black] (1,5) circle (1.5pt);
	\end{tikzpicture}
\caption{An example of the bijection $\omega$. The word should be read from right to left.}
\label{fig:triangular-balanced}
\end{figure}
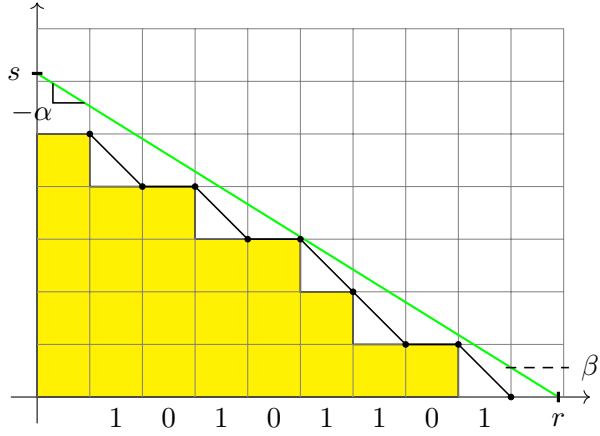

\begin{proposition}
\label{sturmian_encoding_1}
For every $k,\ell\ge1$,  the map $\omega$ from equation~\eqref{def:w} is a bijection 
$$\{\tau=\tau_1\dots\tau_k\in\Wide\mid \tau_1=\ell\}\to\{w=w_1\dots w_\ell\in\cB_\ell\mid \text{$w$ has $k$ ones and }w_1=1\}.$$
\end{proposition}

\begin{proof}
Let $\tau = \tau_1\dots\tau_k\in\Wide$.
It is clear by construction that $\omega(\tau)$ has length $\ell=\tau_1$, that it has $k$ ones, and that it starts with $1$. Let us show that this word is balanced.

Let $\LL_{r,s}$ be a cutting line for $\tau$ with $r>s$. Consider the lattice path from $(\tau_1 + 1,\; 0)$ to $(1, k)$
which passes through the highest point of each column of the Ferrers diagram of $\tau$ from right to left (see Figure \ref{fig:triangular-balanced}). Since the parts of $\tau$ are distinct, the steps of the path belong to $\{(-1,0), (-1,1)\}$, and since $\tau$ is a triangular partition, this is the highest such path that stays weakly below the line~$\LL_{r,s}$. 

By applying the vertical reflection $x\mapsto \tau_1+1-x$, the resulting path from $(0,0)$ to $(\tau_1, k)$ is the highest path with steps in $\{(1,0), (1,1)\}$ that stays weakly below the reflection of $\LL_{r,s}$, which is the line
$(\tau_1 + 1 - x)/r + y/s = 1$, or equivalently,
$$y=\frac{s}{r}\,x+s\left(1-\frac{\tau_1+1}{r}\right).$$
Note that $0 < s/r < 1$. Moreover, $0 < s(1 - (\tau_1 + 1)/r) < 1$, because it is the ordinate of the cutting line $\LL_{r,s}$ at $x = \tau_1 + 1$; this value must be positive because $(\tau_1, 1)\in\tau$ and $r > s$, and it must be less than $1$ because $(\tau_1 + 1,\; 1)\notin\tau$.

By the equivalence between balanced words and mechanical words given in Proposition~\ref{characterizations_of_sturmian_words}, this reflected path corresponds to the balanced word of length $\tau_1$ obtained by encoding steps $(1,0)$ with $0$ and steps $(1,1)$ with $1$. This word is precisely $\omega(\tau)$ by construction, proving that $\omega(\tau)\in\cB_\ell$.

Finally, we prove that $\omega$ is a bijection by constructing its inverse. Let $w\in\cB_\ell$ with $w_1=1$ and having $k$ ones, and let $1 = i_1<\dots <i_k$ be the indices of these ones. Define a partition $\tau = \tau_1\dots\tau_k$ by $\tau_j = \ell - i_j + 1$ for $j \in[k]$, and note that $\tau_1=\ell$. By Proposition~\ref{characterizations_of_sturmian_words}$(c)$, $w$ encodes a lattice path from $(0,0)$ to $(\ell,k)$ with steps in $\{(1,0),(1,1)\}$, which is the highest path staying weakly below the line $y = \alpha x + \beta$, for some $0<\alpha,\beta<1$. 

Applying the vertical reflection $x\mapsto \tau_1+1-x$, the resulting path from $(\ell + 1, 0)$ to $(1, k)$ passes through  the highest point of every column of the Ferrers diagram of $\tau$, and it is the highest path weakly below the reflected line (which is the line through $(\ell + 1, \beta)$ having slope $-\alpha$). This proves that $\tau$ is a wide triangular partition. By construction, the map that sends $w$ to $\tau$ is the inverse of $\omega$, proving that $\omega$ is a bijection.
\end{proof}

The bijection $\omega$ can be used to reduce the problem of determining whether a binary word is balanced to that of determining whether a partition is triangular, which we studied in Section~\ref{sec:triangularity_algorithm}. A naive algorithm 
to determine if a word is balanced using equation~\eqref{eq:balanced} would take quadratic time in the length of the word. However, several linear-time algorithms are known, using a variety of tools: number-theoretic methods~\cite{Boshernitzan1984}, techniques from discrete geometry and a bijection similar to the one in our Lemma~\ref{lem:phi_bijection} \cite{Bruckstein1993,Berstel1996}, a recursive  algorithm~\cite{Richomme1999}, and Lyndon words~\cite{deLuca2006,Reutenauer2015}.

Next we give yet another linear-time 
algorithm, based on the geometry of triangular partitions and the above encoding~$\omega$.
To be able to apply the inverse of this map, we need our word to start with a one. This can be easily achieved by replacing zeros with ones and ones with zeros if needed, since this operation preserves the balanced property from equation~\eqref{eq:balanced}.

\begin{algorithm}[to determine whether a binary word $w = w_1\dots w_\ell$ is balanced]
    \begin{enumerate}
        \item Read the input $w$, and record the indices $1 = i_1<\dots <i_k$ of the letters which are equal to $w_1$.
        \item Let $\lambda_j = \ell - i_j + 1$ for $j \in[k]$.
        \item Apply Algorithm~\ref{alg:triangular} to the partition $\lambda = \lambda_1\dots\lambda_k$.
    \end{enumerate}
\end{algorithm}

For a binary word of length $\ell$ with $k$ letters equal to the first one, reading the input takes time $\Theta(\ell)$ and the rest of the algorithm takes time $\Theta(k)$.

\subsection{Second Sturmian encoding}
\label{subsection:second-sturmian-encoding}

Next we give a different encoding of triangular partitions using balanced words, which appears to be new.
Let $\epsilon$ denote the empty partition, and let $\Wide'$ be the set of wide triangular partitions with at least two parts. Let $\cB^0$ denote the set of balanced words that contain at least one zero.

First we describe the possible sets that can be obtained by taking the differences of consecutive parts in a wide triangular partition.
For $\tau = \tau_1\dots\tau_k\in\Wide'$, define
$$\cD(\tau) = \{\tau_1 - \tau_2,\;\tau_2 - \tau_3, \dots,  \tau_{k-1} - \tau_k \}.$$

\begin{lemma}
\label{lemma_ell_ell+1}
For any $\tau = \tau_1\dots\tau_k\in\Wide'$, 
either $\cD(\tau) = \{d\}$ or $\cD(\tau) = \{d, d + 1\}$ for some $d\ge1$ such that $\tau_k\le d+1$.
\end{lemma}

\begin{proof}
Let $\LL_{r,s}$ be a cutting line for $\tau$. By Definition~\ref{def:triangular}, $\tau_{i + 1} = \left\lfloor{r - (i + 1)r/s}\right\rfloor$ and $\tau_i = \left\lfloor{r - ir/s}\right\rfloor$ for any $i\in[k-1]$, which implies that $r/s - 1 < \tau_i - \tau_{i + 1} < r/s + 1$. It follows that the differences $\tau_i - \tau_{i + 1}$ can take at most two values, and so either $\cD(\tau) = \{d\}$ or $\cD(\tau) = \{d, d + 1\}$ for some $d$. By Lemma~\ref{lem:wide}, the parts of $\tau$ are distinct, so $d\ge1$. Moreover, since the point $(1, k + 1)$ must lie above $\LL_{r,s}$, we have $\tau_k < r/s + 1 < d + 2$, hence $\tau_k \leq d + 1$.
\end{proof}

For $\tau = \tau_1\dots\tau_k\in\Wide'$, let $\min(\tau)=\tau_k$, and let $\dif(\tau) = \min\cD(\tau)$.
Finally, let $\wrd(\tau) = w_1\dots w_{k-1}$ where, for $i\in[k-1]$, we let $w_i = \tau_{i} - \tau_{i + 1} - \dif(\tau)$. Lemma~\ref{lemma_ell_ell+1} guarantees that $w_i\in\{0,1\}$ for all~$i$. Our second encoding is given by the map $\chi = (\min,\dif,\wrd)$. See Figure~\ref{fig:chi} for an example.

\begin{figure}[ht]
\centering
	\begin{tikzpicture}[scale=.7]
		\filldraw[color=black, fill=yellow, semithick] (0,0) -- (0,5) -- (1,5) -- (1,4) -- (4,4) -- (4,3) -- (7,3) -- (7,2) -- (9,2) -- (9,1) -- (12,1) -- (12,0) -- (0,0);
        \axes{14}{7}
        \node at (0.8,5.5) {$\min(\tau)$};
        \node at (2.5,4.5) {$\dif(\tau) + 1$};
        \node at (5.5,3.5) {$\dif(\tau) + 1$};
        \node at (8,2.5) {$\dif(\tau)$};
        \node at (10.5,1.5) {$\dif(\tau) + 1$};
        \node at (2.5,3.6) {$1$};
        \node at (5.5,2.6) {$1$};
        \node at (8,1.6) {$0$};
        \node at (10.5,0.6) {$1$};
	\end{tikzpicture}

\caption{The image of $\tau = (12,9,7,4,1)$ is $\chi(\tau) = (1, 2, 1011)$.}
\label{fig:chi}
\end{figure}
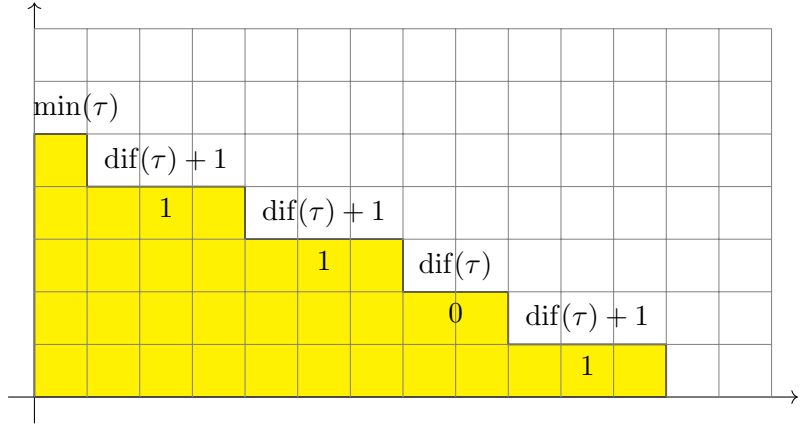

\begin{theorem}
    \label{thm:sturmian_encoding_2}
    The map $\chi = (\min,\dif,\wrd)$ is a bijection between $\Wide'$ and the set
    $$
    \cT = \{(m,d,w)\in\bbN\times\bbN\times\cB^0\mid m\le d + 1; \;  w1\in\cB^0\;\mathrm{if}\;m = d + 1\}.
    $$
    Its inverse is given by the map
    \begin{equation}\label{eq:xi}
    \xi(m,d,w_1\dots w_{k-1}) = \tau_1\dots\tau_k, \quad\text{where }\ \tau_i =  m + \sum_{j = i}^{k-1}(w_j+d)\  \text{ for }i\in[k].
    \end{equation}    
     Additionally, given $\tau\in\Wide'$ with image $\chi(\tau)=(m,d,w)$, its number of parts equals the length of $w$ plus one, and its size is
   \begin{equation}\label{eq:size}
   |\tau|=km + \binom{k}{2}d + \sum_{i = 1}^{k-1} iw_i.
   \end{equation}
\end{theorem}

\begin{proof}
    We start by proving that $\chi(\Wide') \subseteq \cT$. Let $\tau = \tau_1\dots\tau_k\in\Wide'$, and let $w = w_1\dots w_{k-1} = \wrd(\tau)$.
    First, let us show that $w\in\cB^0$. It is clear that $w$ must have a $0$, because if it consisted of only ones, we would have $\tau_i - \tau_{i + 1} = \dif(\tau) + 1$ for all $i\in[k - 1]$, by definition of $\wrd$, which would imply that $\cD(\tau) = \{\dif(\tau) + 1\}$, contradicting the definition of $\dif$. To prove that $w$ is balanced, we show that inequality~\eqref{eq:balanced} holds for any $h\le k-1$ and $i,j \leq k - h$. Since $w_t+\dif(\tau)=\tau_{t} - \tau_{t + 1}$ for all $t\in[k-1]$, we can rewrite the left-hand side as a telescoping sum
    \begin{align*}
    \left|\sum_{t = i}^{i + h - 1}w_t - \sum_{t = j}^{j + h - 1}w_t\right|&=\left|\sum_{t = i}^{i + h - 1}(w_t+\dif(\tau)) - \sum_{t = j}^{j + h - 1}(w_t+\dif(\tau))\right|=\\
    &= \left| \sum_{t = i}^{i + h - 1}(\tau_{t} - \tau_{t + 1}) - \sum_{t = j}^{j + h - 1}(\tau_{t} - \tau_{t+1}) \right| 
    = \left|(\tau_{i} - \tau_{i + h}) - (\tau_{j} - \tau_{j + h})\right|.
    \end{align*}
    By Definition~\ref{def:triangular}, there exist $r,s > 0$ such that $\tau_{t} = \lfloor r - tr/s \rfloor$ for $t\in[k]$, so we can express the above difference as
    $$
    \left|\lfloor r - ir/s \rfloor - \lfloor r - (i + h)r/s \rfloor 
    - \lfloor r - jr/s \rfloor + \lfloor r - (j + h)r/s \rfloor\right|,$$
    which is bounded above by~$1$, using that $(a - 1) - b - d + (e - 1) <\lfloor a\rfloor - \lfloor b\rfloor - \lfloor d \rfloor + \lfloor e\rfloor< a - (b - 1) - (d - 1) + e$ for any real numbers $a,b,d,e$.

    By Lemma \ref{lemma_ell_ell+1},  $\min(\tau) \le \dif(\tau) + 1$.
    It remains to show that, if $\min(\tau) = \dif(\tau) + 1$, then $w1\in\cB^0$. In this case, with the convention $w_k = 1$ and $\tau_{k + 1} = 0$, we have that $w_t = \tau_{t} - \tau_{t+ 1} - \dif(\tau)$ also for $t=k$, and so we can apply the above argument to the word $w_1\dots w_{k-1}w_k=w1$ to conclude that it belongs to $\cB^0$. This finishes the proof that $\chi(\Wide') \subseteq \cT$.\smallskip
    
    Next we show that $\xi(\cT) \subseteq \Wide'$. Let $(m,d,w)\in\cT$, where $w=w_1\dots w_{k-1}$, and let $\tau = \tau_1\dots\tau_k = \xi(m,d,w)$. Since $w$ cannot be empty, we have $k\ge2$.
    
    Consider first the case $m \le d$. Notice that $w_1\dots w_{k-1}$ is balanced if and only if $w_{k-1}w_{k-2}\dots w_1$ is balanced. By Proposition~\ref{characterizations_of_sturmian_words}$(c)$ applied to $w_{k-1}w_{k-2}\dots w_1$, there exist $0 < \alpha, \beta < 1$ such that $w_{k-i} = \lfloor \beta + i\alpha \rfloor - \lfloor \beta + (i - 1)\alpha \rfloor$ for $i\in[k - 1]$. Then, for $i\in[k]$,
    \begin{align*}
            \tau_i &= m + \sum_{j = i}^{k - 1} (w_j+d) 
            = m + (k - i)d  + \sum_{j = i}^{k - 1} \left( \lfloor \beta + (k - j)\alpha \rfloor - \lfloor \beta + (k - j - 1)\alpha \rfloor \right)  \\
            & = m + (k - i)d + \lfloor \beta + (k - i)\alpha \rfloor - \lfloor \beta \rfloor  
            = \lfloor m + \beta + (k - i)(d + \alpha)\rfloor. 
    \end{align*}
    Letting $r = m + \beta + k(d + \alpha)$ and $s = r/(d + \alpha)$, we have $\tau_i = \lfloor r - ir/s \rfloor$. Moreover, $\lfloor s - s/r \rfloor = \lfloor (m + \beta - 1)/(d + \alpha) \rfloor + k = k$, using the assumption that $m \le d$. By Definition~\ref{def:triangular}, this proves that $\tau$ is triangular with cutting line $\LL_{r,s}$. Since $r > s$, we have $\tau\in\Wide'$.

    Now consider the case $m = d + 1$. By the definition of $\cT$, we have $w' := w1\in\cB^0$. Arguing as above, there exist $0 < \alpha', \beta' < 1$ such that $w'_{k+1 - i'} = \lfloor \beta' + i\alpha' \rfloor - \lfloor \beta' + (i - 1)\alpha' \rfloor$ for $i\in[k]$. Then,
    $$
    \tau_i =  m + \sum_{j = i}^{k - 1} (w_j+d) = d+1+(k-i)d+\sum_{j = i}^{k - 1} w'_{j}
    =(k - i + 1)d + \sum_{j = i}^{k}w'_{j} = \lfloor \beta' + (k - i + 1)(d + \alpha')\rfloor.
    $$
    Letting $r' = \beta' + (k + 1)(d + \alpha')$ and $s' = r'/(d + \alpha')$, we have $\tau_i = \lfloor r' - ir'/s' \rfloor$. Moreover, $\lfloor s' - s'/r' \rfloor = \lfloor (\beta' - 1)/(d + \alpha') \rfloor + k + 1 = k$. This proves that $\tau$ is triangular with cutting line $\LL_{r',s'}$. Since $r' > s'$, we have $\tau\in\Wide'$. This finishes the proof that $\xi(\cT) \subseteq \Wide'$.\smallskip

    Next we show that $\xi$ and $\chi$ are inverses of each other. Let $(m,d,w)\in\cT$, with $w = w_1\dots w_{k-1}$, and let $\tau = \tau_1\dots\tau_k = \xi(m,d,w)$.  By definition of $\xi$, we have $\min(\tau) = \tau_k= m$ and $\cD(\tau) = \{w_i+d \mid  i\in[k-1]\}$. Since $w\in\cB^0$, there exists $i\in[k-1]$ such that $w_i = 0$, and therefore $\dif(\tau) = d$. Additionally, the $i$th entry of $\wrd(\tau)$ is $\tau_{i} - \tau_{i + 1} - \dif(\tau) = d + w_i - d = w_i$ for $i\in[k-1]$. This proves that $\chi(\xi(m,d,w)) = (m,d,w)$.

    Now take any $\tau = \tau_1\dots\tau_k\in\Wide'$. To show that $\xi(\chi(\tau))=\tau$, let $w=\wrd(\tau)$, and note that the $i$th entry of $\xi(\chi(\tau))$ is $\min(\tau) + \sum_{j = i}^{k - 1}(w_j+\dif(\tau)) = \tau_k + \sum_{j = i}^{k - 1}(\tau_j - \tau_{j + 1}) = \tau_i$.

    Finally, to prove equation~\eqref{eq:size} for any $\tau\in\Wide'$ with $\chi(\tau)=(m,d,w)$, we use the definition of the inverse map $\xi$ from equation~\eqref{eq:xi}. Adding all the parts of $\tau$, 
    $$|\tau| = |\xi(m,d,w)| = \sum_{i = 1}^k\left(m + \sum_{j = i}^{k - 1}(w_j+d)\right) = km + \binom{k}{2}d + \sum_{i = 1}^{k - 1} iw_i.\qedhere$$
\end{proof}

Before we discuss how to use Theorem~\ref{thm:sturmian_encoding_2} to generate triangular partitions, we finish this subsection describing another application of the above encoding. 

\begin{proposition}
\label{equiv_chi3_removable}
Let $\tau,\nu\in\Wide'$ with $k$ parts such that $\wrd(\tau) = \wrd(\nu)$, and suppose that the values $\min(\tau) - \dif(\tau)$ and $\min(\nu) - \dif(\nu)$ are either both equal to $1$ or both less than $1$.
Then, for every $i\in[k]$, the cell $c = (\tau_i, i)$ is a removable cell of $\tau$ if and only if $c' = (\nu_i, i)$ is a removable cell of $\nu$.
\end{proposition}

\begin{proof}
Let $\chi(\tau)=(m, d, w)$ and $\chi(\nu)=(m', d', w)$, where $\chi$ is the map from Theorem~\ref{thm:sturmian_encoding_2}.  
In the special case where $w$ consists of only zeros, all the corner cells $(\tau_i,i)$ of $\tau$ lie on the cutting line $x-m+d(y-k)=0$, and all the corner cells of $\nu$ lie on the cutting line $x-m'+d'(y-k)=0$. It follows that the removable cells of $\tau$ are $(\tau_1,1)$ and $(\tau_k,k)$, and the removable cells of $\nu$ are $(\nu_1,1)$ and $(\nu_k,k)$. Thus, the result holds in this case. In the rest of the proof we will assume that $w$ contains some one.

By symmetry, it suffices to show that if $c$ is removable in $\tau$, then $c'$ is removable in $\nu$.
By Lemma~\ref{line_touching_removable_cell}, if $c$ is a removable cell of $\tau$, there is a cutting line $L$ that passes through $c$ and so that all the other corner cells $(\tau_j,j)$ for $j\ne i$  lie strictly below $L$, and the cells $(1,k+1)$ and $(\tau_j+1,j)$ for all $j$  lie strictly above $L$. Writing the equation of $L$ as $x-\tau_i+a(y-i)=0$ for some positive real number $a$, these conditions are equivalent to 
$-1<\tau_j-\tau_i+a(j-i)<0$ for $j\neq i$, and
$1-\tau_i+a(k+1-i)>0$.

Let $L'$ be the line with equation $x-\nu_i+(a+d'-d)(y-i)=0$. This line passes through $c'$. We will show that all the other corner cells $(\nu_j,j)$ for $j\ne i$ lie strictly below $L'$, and the cells $(1,k+1)$ and $(\nu_j+1,j)$ for all $j$ lie strictly above $L'$, thus proving that $L'$ is a cutting line for $\nu$ and $c'$ is a removable cell. This is equivalent to showing that 
\begin{equation}\label{eq:nuj} -1<\nu_j-\nu_i+(a+d'-d)(j-i)<0
\end{equation} 
for $j\neq i$, and
\begin{equation}\label{eq:1k+1} 1-\nu_i+(a+d'-d)(k+1-i)>0.
\end{equation}

Equation~\eqref{eq:nuj} follows from the analogous equation for $\tau_j$, since the description of $\xi$ from equation~\eqref{eq:xi} implies that  
\begin{equation}\label{eq:nujtauj} 
\nu_j-\tau_j=m' + \sum_{\ell = j}^{k - 1}(w_\ell+d')-
m - \sum_{\ell = j}^{k - 1}(w_\ell+d)=m'-m+(k-j)(d'-d),
\end{equation}
and so
$$\nu_j-\nu_i+(a+d'-d)(j-i)=\tau_j-\tau_i+(i-j)(d'-d)+(a+d'-d)(j-i)=\tau_j-\tau_i+a(j-i).$$

To prove equation~\eqref{eq:1k+1} in the case $m-d=m'-d'=1$, we again use equation~\eqref{eq:nujtauj} to write
$$1-\nu_i+(a+d'-d)(k+1-i)=m-d-(m'-d')+1-\tau_i+a(k+1-i)=1-\tau_i+a(k+1-i)>0.$$

It remains to show that, in the case that both $m-d$ and $m'-d'$ are less than $1$, the cell $(1,k+1)$ also lies strictly above $L'$. By the above assumption, there is some $j$ such that $w_j=1$. Then, the fact that the cell $(\nu_j,j)$ lies weakly below $L'$ but the cell $(\nu_{j+1}+1,j+1)$ lies strictly above $L'$ forces the slope of $L'$ to be smaller (in absolute value) than the slope of the line between these two cells; equivalently, $a+d'-d>\nu_j-(\nu_{j+1}+1)=d'$. Thus, since the cell $(\nu_k+1,k)=(m'+1,k)$ lies strictly above $L'$, so does the cell $(m'+1-d',k+1)$, and, using that $m'-d'\le0$, the cell $(1,k+1)$ must lie strictly above~$L'$.
\end{proof}

As a consequence of the above result, we can determine the removable cells of any $\tau = \tau_1\dots\tau_k\in\Wide'$ by instead finding the removable cells of the smallest partition $\nu$ satisfying the conditions of Proposition~\ref{equiv_chi3_removable}. 
Letting $w=\wrd(\tau)$, this is the partition $\xi(2, 1, w)$ if $\min(\tau) - \dif(\tau) = 1$, or $\xi(1,1,w)$ otherwise, where $\xi$ is given by equation~\eqref{eq:xi}.
After finding the removable cells of $\nu$, we can use Proposition \ref{equiv_chi3_removable} to determine those of $\tau$.
For triangular partitions $\tau$ that are not wide, we can apply the same procedure to $\tau'$, which must be wide, noting that a cell $c = (a,b)$ is removable from $\tau$ if and only if $c' = (b,a)$ is removable from $\tau'$. For example, to find the removable cells of $\tau = (5^{576},4^{1037},3^{1037},2^{1036},1^{1037})$ (where the exponents indicate the multiplicities of the parts), we take $\tau' = (4723, 3686, 2650, 1613, 576)$, which has $\min(\tau') = 576$, $\cD(\tau') = \{1036,1037\}$, and $\dif(\tau') = 1036$. Thus, $\wrd(\tau') = 1011$, so we consider the much smaller partition $\nu = \xi(1,1,1011) = (8,6,5,3,1)$, whose only removable cell is $(5,3)$. By Proposition~\ref{equiv_chi3_removable}, the only removable cell of $\tau'$ is $(2650, 3)$, and so the only removable cell of $\tau$ is $(3, 2650)$.

\subsection{Efficient generation} \label{sec:algorithm}

Next we describe an efficient algorithm to generate triangular partitions, which relies on our new encoding of such partitions as balanced words from Section~\ref{subsection:second-sturmian-encoding}. 

At the time of writing this paper, the entry of the On-Line Encyclopedia of Integer Sequences~\cite[A352882]{oeis} for the number triangular partitions of $n$ only included values for $n\le 39$. These are the terms that appear in~\cite{Corteel1999}, where they were obtained using the generating function in Theorem~\ref{thm:GDelta}. We have checked that an algorithm based on this generating function becomes very slow when trying to compute the next terms of the sequence, and it is impractical for large~$n$.

Theorem~\ref{thm:sturmian_encoding_2} can be used to implement the following much more efficient algorithm, which can easily find the first $10^5$ terms of the sequence. 

\begin{algorithm}[to compute $|\Delta(n)|$ for $1\le n\le N$]
\begin{enumerate}
\item Perform a depth first search through the tree of balanced words of length up to $\lfloor \sqrt{2N} \rfloor$ (this is because a wide triangular partition $\tau$ with $k$ parts has size at least $\binom{k + 1}{2}$). 
The nodes of this tree correspond to balanced words, and the parent of a nonempty balanced word is the balanced word obtained by removing its last letter. Each $w\in\cB$ has one or two children, because at least one of the words $w0$ and $w1$ must be balanced. Recall from equation~\eqref{eq:balanced} that a binary word $w$ of length $\ell$ is balanced if, for every $h\le \ell$,
all the factors of length $h$ have the same number of ones, or some factors have $a$ ones and some have $a+1$ ones, for some $a$.
For $w\in\cB_\ell$, to determine whether $w0$ and $w1$ are balanced, we keep a vector that records, for each $h\le\ell$, whether all the factors of length $h$ have the same number of ones, or otherwise, whether the rightmost factor of $w$ has $a$ or $a+1$ ones. This vector 
can be used to determine whether $w0$ and $w1$ are balanced, and if they are, to compute the vectors corresponding to these words in time $\bigO(\ell)$. We also keep a variable $n_{\min} = \binom{\ell + 1}{2} + \sum_{i = 1}^{\ell} iw_i$ that is updated in constant time every time a letter is added.
\item For each $w\in\cB_\ell$ (other than $w=1^\ell$) found in step~1, after checking whether $w0$ and $w1$ are balanced and before continuing the exploration of the tree, search in lexicographic order through the pairs $(d,m)\in\bbN^2$ such that $(m, d, w)\in\cT$, as defined in Theorem~\ref{thm:sturmian_encoding_2}. For each pair, use equation~\eqref{eq:size} to calculate the size in constant time as $n_{\min} + \binom{\ell + 1}{2}(d - 1) + (\ell + 1)m$. If this size is greater than $N$, then if $m > 1$ we move on to the pair $(d+1,1)$, whereas if $m = 1$ we are done with $w$, so we move on to the next word in the tree.
\item Each triplet $(m,d,w)$ with size smaller than or equal to $N$ accounts for two triangular partitions, namely $\tau=\chi(m,d,w)$ and its conjugate $\tau'$, except when $w=0^{k-1}$ (for some $k\ge2$) and $m = d$, in which case it accounts for only one partition, the staircase $\sigma^k$.
\end{enumerate}
\end{algorithm}

A C++ implementation of this algorithm is available at \cite{website}.
In a standard laptop computer, it yields the first $10^3$ terms of the sequence $|\Delta(n)|$ in under one second, the first $10^4$ terms in under ten seconds, and the first $10^5$ terms in under one hour. Next we study the running time of our algorithm.
By comparison, Onn and Sturmfels showed in~\cite{Onn1999} that it is possible to generate triangular partitions of $n$ in time $\bigO(n^4)$.

\begin{proposition}
    The above algorithm finds $|\Delta(n)|$ for $1\le n\le N$ in time $\bigO(N^{5/2})$. Additionally, it can be modified to generate all (resp., all wide) triangular partitions of size at most $N$ in time $\bigO(N^{3}\log N)$ (resp., $\bigO(N^{5/2}\log N)$). 
\end{proposition}

\begin{proof}
    For each word $w$ of length $\ell\le\sqrt{2N}$ in the tree of balanced words, the algorithm makes $\bigO(\ell)$ comparisons to determine whether appending a $0$ or a $1$ produces a balanced word. In \cite{Akiyama2022}, it is proved that the number of balanced words of length at most $L$ is $\bigO(L^4)$. Therefore, the total number of comparisons performed in step~1 is $\bigO(N^{5/2})$.

    In step~2, for each balanced word $w$ containing at least one zero, the algorithm runs through pairs $(d,m)$ such that $(m,d,w)\in\cT$. For each one of these, the algorithm computes $|\tau|$ using equation~\eqref{eq:size} and checks if $|\tau|\le N$ in constant time. By Theorem~\ref{thm:corteel}, the number of triangular partitions of size at most $N$ is $\bigO(N^2\log N)$. Additionally, the number of  ``incorrect'' visited triplets $(m,d,w)$ for which $|\tau| > N$ is also $\bigO(N^2\log N)$, since those with $m=1$ come from different words $w$ (and we visit $\bigO(N^2)$ of them), and those with $m > 1$ came after a ``correct'' triplet $(m-1,d,w)$, i.e., corresponding to a triangular partition of size at most $N$ (and there are $\bigO(N^2\log N)$ of these). We conclude that the algorithm runs in time $\bigO(N^{5/2})$.

    To output the list of triangular partitions, for each $(m,d,w)$ that the algorithm finds, one can compute $\tau=\xi(m,d,w)$ using equation~\eqref{eq:xi} in time $\bigO(\sqrt{N})$, and its conjugate $\tau'$ in time $\bigO(N)$. Since the number of triangular partitions that are computed is $\bigO(N^2\log N)$, this construction generates the wide ones in time $\bigO(N^{5/2}\log N)$, and their conjugates in time $\bigO(N^3\log N)$.
\end{proof}

In their proof of Theorem~\ref{thm:corteel}, Corteel et al.~\cite{Corteel1999} show that $\pp(n/8)/3 \leq|\Delta(n)| \leq \pp(2n + 1)$, where $\pp(m)$ is the number of relatively prime pairs $(a,b)$ such that $ab < m$. In fact, it is possible to slightly improve the lower bound to $\pp(n/2)/3$, by tweaking the same argument that they use, but noting that if $n > 2ab$, then the set $H((a,b))$ defined in~\cite{Corteel1999} contains at least two lattice points.
 
\begin{figure}[ht]
    \centering
    \includegraphics[width=11cm]{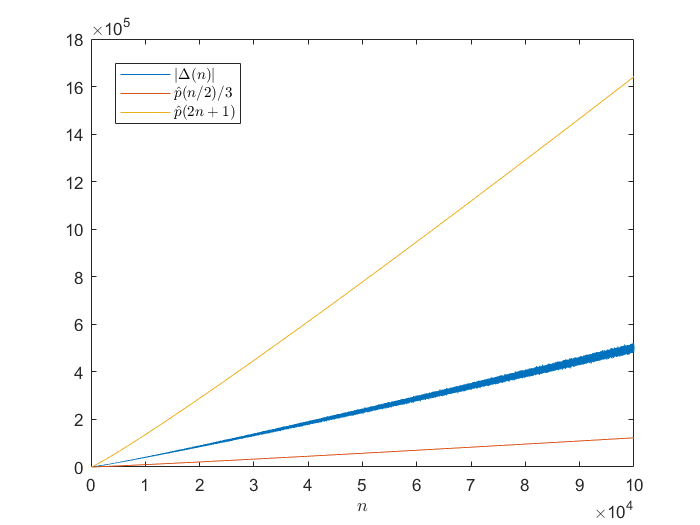}
    \caption{The sequence $|\Delta(n)|$ and the bounds $\pp(2n + 1)$ and $\pp(n/2)/3$ for $n\le10^5$.}
    \label{fig:tp_bounds}
\end{figure}

In Figure \ref{fig:tp_bounds}, the first $10^5$ terms of the sequence $|\Delta(n)|$ are plotted against these bounds $\pp(2n + 1)$ and $\pp(n/2)/3$.  While the constants $C$ and $C'$ (as in Theorem~\ref{thm:corteel}) that result from these bounds are far from tight,
Figure \ref{fig:tp_bounds_nlogn} suggests that, for large $n$, the value of 
$|\Delta(n)|/(n\log n)$ oscillates between two decreasing functions that differ by about $0.05$.

\begin{figure}[ht]
    \centering
    \includegraphics[width=11cm]{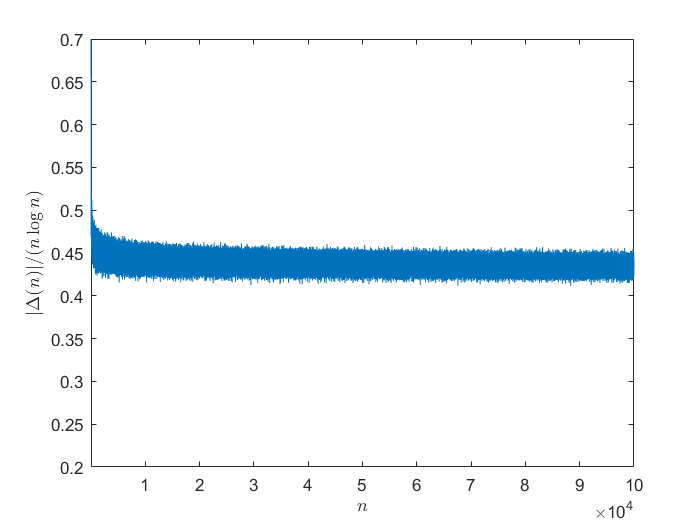}
    \caption{The values $|\Delta(n)|/(n\log n)$ for $n\le10^5$.}
    \label{fig:tp_bounds_nlogn}
\end{figure}

\section{Generating functions for subsets of triangular partitions}
\label{sec:generating-functions}

Let us introduce notation for some subsets of triangular partitions. Let $\Delta_1$ and $\Delta_2$ denote the subsets of partitions with one removable cell and with two removable cells, respectively. Let $\Delta^1$ and $\Delta^2$ denote the subsets of partitions with one addable cell and with two addable cells, respectively. Let $\Delta_2^2=\Delta_2\cap\Delta^2$.
Denote partitions of size $n$ in each subset by $\Delta_1(n)$, $\Delta_2(n)$, $\Delta^1(n)$, $\Delta^2(n)$ and $\Delta_2^2(n)$. 
In this section we give generating functions for each of these sets, refining Theorem~\ref{thm:GDelta}.

To obtain a generating function for partitions in $\Delta_2$, we modify the construction that Corteel et al.~\cite{Corteel1999} used to prove Theorem~\ref{thm:GDelta} for all triangular partitions. 
In our case, each $\tau\in\Delta_2$ is uniquely determined by the following parameters, as given by Definition~\ref{def:diagonal}: its diagonal slope (encoded by an irreducible fraction $a/b$), the number of cells in $\partial_\tau$ (denoted by $k = |\partial_\tau|$), the $x$-coordinate of the leftmost cell in $\partial_\tau$ (denoted by $i+1$), and the $y$-coordinate of the rightmost cell in $\partial_\tau$  (denoted by $j+1$).

Decomposing $\tau$ as shown in Figure~\ref{fig:Delta2} and counting cells as in Section~\ref{sec:enumeration}, we see that the size of $\tau$ is $N_{\Delta}(a,b,k,k,i,j)$,
with $N_\Delta$ as given by equation~\eqref{eq:NDelta}. Bearing in mind that $k\geq2$ because the diagonal contains both removable cells, we obtain the following generating function for $\Delta_2$.

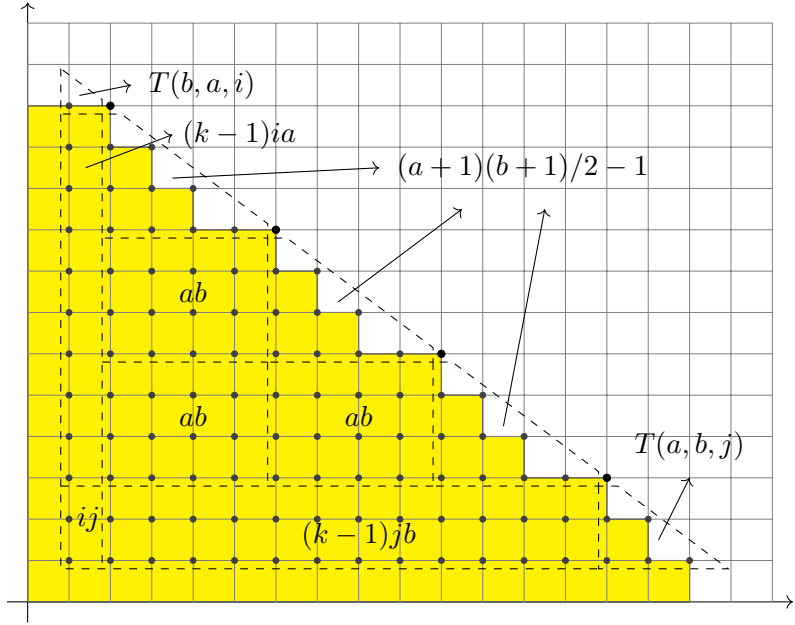
\begin{figure}[ht]
\centering
	\begin{tikzpicture}[scale=.55]
		\filldraw[color=black, fill=yellow, thin] (0,0) -- (0,12) -- (2,12) -- (2,11) -- (3,11) -- (3,10) -- (4,10) -- (4,9) -- (6,9) -- (6,8) -- (7,8) -- (7,7) -- (8,7) -- (8,6) -- (10,6) -- (10,5) -- (11,5) -- (11,4) -- (12,4) -- (12,3) -- (14,3) -- (14,2) -- (15,2) -- (15,1) -- (16,1) -- (16,0) -- (0,0);
        \axes{18}{14}
        \draw[dashed] (1 - 1/5, 12 + 9/10) -- (1 - 1/5, 12 - 1/5) -- (2 - 1/5, 12 - 1/5) -- (2 - 1/5, 12 + 3/20) -- (1 - 1/5, 12 + 9/10);
        \foreach \i in {0,1,2}
            {\draw[dashed] (2 - 1/5 + 4*\i, 12 - 1/5 - 3*\i) -- (2 - 1/5 + 4*\i, 9 - 1/5 - 3*\i) -- (6 - 1/5 + 4*\i, 9 - 1/5 - 3*\i) -- (6 - 1/5 + 4*\i, 9 + 3/20 - 3*\i) -- (2 + 4/15 + 4*\i, 12 - 1/5 - 3*\i) -- (2 - 1/5 + 4*\i, 12 - 1/5 - 3*\i);}
        \draw[dashed] (14 - 1/5, 3 - 1/5) -- (14 - 1/5, 1 - 1/5) -- (16 + 14/15, 1 - 1/5) -- (14 + 4/15, 3 - 1/5) -- (14 - 1/5, 3 - 1/5);
        \draw[dashed] (1 - 1/5, 12 - 1/5) -- (1 - 1/5, 1 - 1/5) -- (14 - 1/5, 1 - 1/5);
        \draw[dashed] (2 - 1/5, 9 - 1/5) -- (2 - 1/5, 1 - 1/5);
        \draw[dashed] (1 - 1/5, 3 - 1/5) -- (10 - 1/5, 3 - 1/5);
        \draw[dashed] (2 - 1/5, 6 - 1/5) -- (6 - 1/5, 6 - 1/5);
        \draw[dashed] (6 - 1/5, 6 - 1/5) -- (6 - 1/5, 3 - 1/5);
        \filldraw[black] (4,7) circle (1.5pt) node[above]{$ab$};
        \filldraw[black] (4,4) circle (1.5pt) node[above]{$ab$};
        \filldraw[black] (8,4) circle (1.5pt) node[above]{$ab$};
        \filldraw[black] (2,2) circle (1.5pt) node[left]{$ij$};
        \filldraw[black] (8,1) circle (1.5pt) node[above]{$(k - 1)jb$};
        \draw[->] (1.4, 10.5) -- (3.5, 11.3) node[right]{$(k-1)ia$};
        \foreach \x in {1,...,16}
            {\filldraw[gray!150] (\x,1) circle (2pt);}
        \foreach \x in {1,...,15}
            {\filldraw[gray!150] (\x,2) circle (2pt);}
        \foreach \x in {1,...,13}
            {\filldraw[gray!150] (\x,3) circle (2pt);}
        \filldraw[black] (14,3) circle (2.5pt);
        \foreach \x in {1,...,12}
            {\filldraw[gray!150] (\x,4) circle (2pt);}
        \foreach \x in {1,...,11}
            {\filldraw[gray!150] (\x,5) circle (2pt);}
        \foreach \x in {1,...,9}
            {\filldraw[gray!150] (\x,6) circle (2pt);}
        \filldraw[black] (10,6) circle (2.5pt);
        \foreach \x in {1,...,8}
            {\filldraw[gray!150] (\x,7) circle (2pt);}
        \foreach \x in {1,...,7}
            {\filldraw[gray!150] (\x,8) circle (2pt);}
        \foreach \x in {1,...,5}
            {\filldraw[gray!150] (\x,9) circle (2pt);}
        \filldraw[black] (6,9) circle (2.5pt);
        \foreach \x in {1,...,4}
            {\filldraw[gray!150] (\x,10) circle (2pt);}
        \foreach \x in {1,...,3}
            {\filldraw[gray!150] (\x,11) circle (2pt);}
        \filldraw[gray!150] (1,12) circle (2pt);
        \filldraw[black] (2,12) circle (2.5pt);
        \draw[->] (3.5, 10.25) -- (8.5, 10.5) node[right, outer sep=3pt]{$(a + 1)(b + 1)/2 - 1$};
        \draw[->] (7.5, 7.25) -- (10.5, 9.5);
        \draw[->] (11.5, 4.25) -- (12.5, 9.5);
        \draw[->] (1.25, 12.25) -- (2.5, 12.5) node[right, outer sep=3pt]{$T(b,a,i)$};
        \draw[->] (15.25, 1.5) -- (16, 3) node[above, outer sep=3pt]{$T(a,b,j)$};
	\end{tikzpicture}
\caption{Decomposition of a partition in $\Delta_2$.}
\label{fig:Delta2}
\end{figure}

\begin{proposition}\label{prop:GDelta_2}
The generating function for triangular partitions with two removable cells is
$$
G_{\Delta_2}(z) = \sum_{n\ge0}|\Delta_2(n)|z^n = \sum_{\gcd(a, b) = 1}\sum_{\substack{0\leq j < a \\ 0\leq i < b}}\sum_{k\ge2}z^{N_{\Delta}(a,b,k,k,i,j)}.
$$
\end{proposition}

Note that the term $\frac{1}{1 - z}$ that appeared in Theorem~\ref{thm:GDelta} is not included here, since its purpose was to account for partitions with one part, but those have only one removable cell.
Our next result shows how to obtain generating functions for all the other cases, in terms of the expressions in Theorem~\ref{thm:GDelta} and Proposition~\ref{prop:GDelta_2}.

\begin{proposition}
The generating functions for partitions in $\Delta_1$, $\Delta^2$, $\Delta^1$, $\Delta_2^2$ can be written in terms of $G_\Delta(z)$ and $G_{\Delta_2}(z)$ as follows:
\begin{align}
\label{eq:GDelta_1}
    G_{\Delta_1}(z) &= G_{\Delta}(z) - G_{\Delta_2}(z)-1,\\
\label{eq:GDelta^2}
    G_{\Delta^2}(z) &= \frac{1 - z}{z}G_\Delta(z) + \frac{1}{z}G_{\Delta_2}(z) - \frac{1}{z},\\
\label{eq:GDelta^1}
    G_{\Delta^1}(z) &= \frac{2z - 1}{z}G_\Delta(z) - \frac{1}{z}G_{\Delta_2}(z) + \frac{1}{z},\\
\label{eq:GDelta_2^2}
    G_{\Delta_2^2}(z) &= \frac{1 - 2z}{z}G_\Delta(z) + \frac{1 + z}{z}G_{\Delta_2}(z) - \frac{1}{z}.
\end{align}
\end{proposition}

\begin{proof}
Equation~\eqref{eq:GDelta_1} is immediate since all nonempty triangular partitions have one or two removable cells, by Lemma~\ref{lem:removable-addable}.

To prove equation~\eqref{eq:GDelta^2}, we will count in two ways the number of edges between levels $n$ and $n+1$ in the Hasse diagram of $\TYP$. Since partitions in $\Delta^1(n)$ are covered by one element, and partitions in $\Delta^2(n)$ are covered by two, the number of such edges is $|\Delta^1(n)| + 2|\Delta^2(n)| = |\Delta(n)| + |\Delta^2(n)|$. On the other hand, counting how many elements are covered by each partition of $n+1$, we see that the number of such edges is $|\Delta(n + 1)| + |\Delta_2(n + 1)|$. It follows that 
\begin{equation}\label{eq:edge_counting}
|\Delta(n)| + |\Delta^2(n)| = |\Delta(n + 1)| + |\Delta_2(n + 1)|
\end{equation}
for $n\geq1$. In terms of generating functions, we obtain
\begin{align*}
G_{\Delta^2}(z) &= \sum_{n\geq1}|\Delta^2(n)|z^n = \sum_{n\geq1}|\Delta(n + 1)|z^n + \sum_{n\geq1}|\Delta_2(n + 1)|z^n - \sum_{n\geq1}|\Delta(n)|z^n =\\
&= \frac{1}{z}(G_\Delta(z) - 1-z) + \frac{1}{z}G_{\Delta_2}(z) - (G_\Delta(z)-1) = \frac{1 - z}{z}G_\Delta(z) + \frac{1}{z}G_{\Delta_2}(z) - \frac{1}{z}.
\end{align*}

Equation~\eqref{eq:GDelta^1} follows from equation~\eqref{eq:GDelta^2} and the fact that all triangular partitions have one or two addable cells, so $G_{\Delta^1}(z) = G_{\Delta}(z) - G_{\Delta^2}(z)$.
Finally, to deduce equation~\eqref{eq:GDelta_2^2}, we use the fact that partitions with two removable cells have one or two addable cells, and that $\Delta^1(n)\subseteq\Delta_2(n)$ by Lemma~\ref{lem:removable-addable}, so $G_{\Delta_2^2}(z) = G_{\Delta_2}(z) - G_{\Delta^1}(z)$.
\end{proof}

From equation~\eqref{eq:edge_counting}, we can derive the equality
$$
2|\Delta(n)|-|\Delta^1(n)| = 2|\Delta(n + 1)| - |\Delta_1(n + 1)|
$$
for $n\ge1$, from where we obtain the upper bound
$$|\Delta(n + 1)| - |\Delta(n)| = \frac{1}{2}\left(|\Delta_1(n + 1)| - |\Delta^1(n)|\right)\le \frac{1}{2}|\Delta_1(n + 1)|.$$

The expression for $G_{\Delta_2}$ given in Proposition~\ref{prop:GDelta_2} can be used to write an algorithm to find $|\Delta_2(n)|$. We have computed the first $100$ terms of this sequence using a MATLAB implementation of this algorithm, which is available at~\cite{website}. 
The first 50 terms of the sequences $|\Delta_1(n)|$ and $|\Delta_2(n)|$ appear in Table~\ref{tab:Delta_12}, and the first 100 terms are plotted in Figure~\ref{fig:orcp_trcp}. It appears to be the case that 
$|\Delta_2(n)|>|\Delta_1(n)|$ for all $n\ge9$, although we do not have a proof of this fact. It is interesting to note that both the local maxima of $|\Delta_1(n)|$ and the local minima of $|\Delta_2(n)|$ seem to occur precisely when $n\equiv2 \pmod 3$. On the other hand, $|\Delta(n)|$ does not exhibit such periodic extrema.

\begin{table}[ht]
\centering
\begin{tabular}{ c|c|c|c|c|c|c|c|c|c|c|c|c|c|c|c|c|c|c} 
 $n$ & 1 & 2 & 3 & 4 & 5 & 6 & 7 & 8 & 9 & 10 & 11 & 12 & 13 & 14 & 15 & 16 & 17 & 18 \\
 \hline
 \hline
 $|\Delta_1(n)|$ & 1 & 2 & 2 & 2 & 4 & 2 & 2 & 6 & 4 & 2 & 8 & 2 & 4 & 10 & 6 & 2 & 10 & 6 \\ 
 \hline
 $|\Delta_2(n)|$ & 0 & 0 & 1 & 2 & 2 & 5 & 6 & 4 & 8 & 11 & 8 & 14 & 14 & 10 & 17 & 22 & 16 & 20 \\ 
 \hline
 $|\Delta(n)|$ & 1 & 2 & 3 & 4 & 6 & 7 & 8 & 10 & 12 & 13 & 16 & 16 & 18 & 20 & 23 & 24 & 26 & 26 \\ 
 \hline
\end{tabular}\medskip 

\begin{tabular}{ c|c|c|c|c|c|c|c|c|c|c|c|c|c|c|c|c} 
 $n$ & 19 & 20 & 21 & 22 & 23 & 24 & 25 & 26 & 27 & 28 & 29 & 30 & 31 & 32 & 33 & 34 \\
 \hline
 \hline
 $|\Delta_1(n)|$ & 8 & 12 & 6 & 2 & 18 & 6 & 8 & 16 & 8 & 8 & 24 & 2 & 6 & 20 & 14 & 12 \\ 
 \hline
 $|\Delta_2(n)|$ & 22 & 20 & 29 & 32 & 20 & 32 & 34 & 28 & 38 & 39 & 30 & 50 & 48 & 32 & 42 & 48 \\ 
 \hline
 $|\Delta(n)|$ & 30 & 32 & 35 & 34 & 38 & 38 & 42 & 44 & 46 & 47 & 54 & 52 & 54 & 52 & 56 & 60 \\ 
 \hline
\end{tabular}\medskip 

\begin{tabular}{ c|c|c|c|c|c|c|c|c|c|c|c|c|c|c|c|c} 
 $n$ & 35 & 36 & 37 & 38 & 39 & 40 & 41 & 42 & 43 & 44 & 45 & 46 & 47 & 48 & 49 & 50 \\
 \hline
 \hline
 $|\Delta_1(n)|$ & 26 & 6 & 6 & 24 & 22 & 6 & 30 & 6 & 20 & 30 & 10 & 8 & 36 & 14 & 18 & 32 \\ 
 \hline
 $|\Delta_2(n)|$ & 40 & 61 & 62 & 42 & 50 & 66 & 50 & 68 & 62 & 54 & 77 & 78 & 54 & 74 & 78 & 64 \\ 
 \hline
 $|\Delta(n)|$ & 66 & 67 & 68 & 66 & 72 & 72 & 80 & 76 & 82 & 84 & 87 & 86 & 90 & 88 & 96 & 96 \\ 
 \hline
\end{tabular}
\caption{The sequences $|\Delta_1(n)|$, $|\Delta_2(n)|$ and $|\Delta(n)|$ for $1\le n\le50$.}
\label{tab:Delta_12}
\end{table}

\begin{figure}[ht]
    \centering
    \includegraphics[width=11cm]{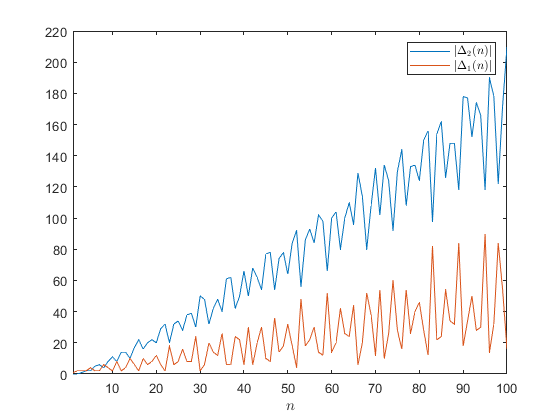}
    \caption{Plot of $|\Delta_2(n)|$ and $|\Delta_1(n)|$ for $1\le n\le100$.}
    \label{fig:orcp_trcp}
\end{figure}

\section{Triangular subpartitions and a combinatorial proof of Lipatov's enumeration formula for balanced words}
\label{sec:subpartitions}

For $\tau\in\Delta$, let $\I(\tau)=|\{\nu\in\Delta:\nu\subseteq\tau\}|$ denote the number of triangular subpartitions of $\tau$.
In this section, after giving a general recurrence for these numbers, we will present explicit formulas for $\I(\tau)$ in some particular cases. We will also derive a new, combinatorial proof of Theorem~\ref{lipatov}.

Recall from Definition~\ref{def:diagonal} that $\tau^\circ$ denotes the interior of $\tau$. Let $c^-$ and $c^+$ be the removable cells of $\tau$, which are the leftmost and rightmost cells of $\partial_\tau$. If $\tau$ has only one removable cell, then $c^-=c^+$.

\begin{lemma}
\label{lem:recurrence_triangular_subpartitions}
For any $\tau\in\Delta(n)$ with $n\geq1$, 
$$
\I(\tau) = \I(\tau\setminus\{c^-\}) + \I(\tau\setminus\{c^+\}) - \I(\tau^\circ) + 1.
$$
\end{lemma}

\begin{proof}
    If $c^-\neq c^+$, then $\tau$ covers two elements in $\TYP$, namely $\tau\setminus\{c^-\}$ and $\tau\setminus\{c^+\}$. By Proposition~\ref{triang_join_meet_construction}, the meet of these two elements is 
    $ \bbN^2\setminus\big((\bbN^2\setminus\tau)\cup\partial_\tau\big) = \tau\setminus\partial_\tau = \tau^\circ$. The formula now follows by inclusion-exclusion.

    If $\tau$ has only one removable cell $c$, then $\I(\tau) = \I(\tau\setminus\{c\})+1$. But in this case $\tau^\circ=\tau\setminus\{c\}$ by definition.
\end{proof}

The above recurrence relation, along with the base case $\I(\epsilon)=1$, allows us to compute $\I(\tau)$ for any $\tau\in\Delta$, although not very efficiently. 
In order to find explicit formulas for $\I(\tau)$ in some cases, let us consider the related problem of counting triangular partitions whose width is at most $\ell$ and whose height is at most $h$; equivalently, those whose Young diagram fits inside an $h\times\ell$ rectangle. We denote by $\Delta^{h\times\ell}$ the set of such partitions.
The next lemma, which is illustrated in Figure~\ref{fig:rectangle_equivalence}, shows how these two problems are related.

\begin{figure}[ht]
\centering
	\begin{tikzpicture}[scale=.6]
		\filldraw[color=black, fill=yellow, semithick] (0,0) -- (0,5) -- (2,5) -- (2,4) -- (4,4) -- (4,3) -- (5,3) -- (5,2) -- (7,2) -- (7,1) -- (8,1) -- (8,0) -- (0,0);
        \axes{10}{7}
        \draw[red, thick] (0,0) -- (0,5) -- (8,5) -- (8,0) -- (0,0);
        \draw[green, thick] (0,13/2) -- (10,1/4);
	\end{tikzpicture}
\caption{A triangular partition fits inside a $5\times8$ rectangle if and only if it is a subpartition of $\tau = 87542$.}
\label{fig:rectangle_equivalence}
\end{figure}
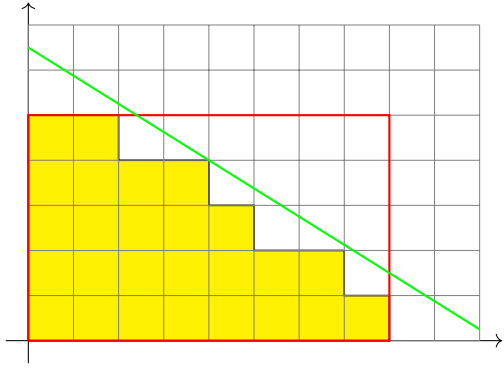

\begin{lemma}
\label{equivalence_subpartitions-rectangle}
Let $h,\ell\ge1$, and let $\nu\in\Delta$. Then $\nu\in\Delta^{h\times\ell}$ if and only if $\nu\subseteq\tau$, where $\tau=\tau_1\dots\tau_h$ is the triangular partition given by
$$
\tau_j = \left\lfloor\ell + 1 - \frac{\ell(j - 1) + 1}{h}\right\rfloor
$$
for $1\le j\le h$.
\end{lemma}

\begin{proof}
The partition $\tau$ consists of the cells that lie weakly below the line $hx+\ell y=h\ell+h+\ell-1$,
which implies that $\tau$ is triangular. Additionally, since $\tau_1=\ell$ and $\tau$ has $h$ parts, we have $\tau\in\Delta^{h\times\ell}$; hence the same is true for any subpartition of $\tau$.

It remains to show that any $\nu\in\Delta^{h\times\ell}$ must satisfy $\nu\subseteq\tau$. Consider a cell $c =(a, b) \in\nu$, and suppose for the sake of contradiction that $c\notin\tau$. Then 
$ha+\ell b>h\ell+h+\ell-1$, and so $ha+\ell b\ge h\ell+h+\ell$, that is, $(a, b)$ lies weakly above the line $hx+\ell y=h\ell+h+\ell$. This line passes through $(1, h +1)$ and $(\ell + 1, 1)$, so any cutting line for $\nu$ must lie weakly above one of these two points, contradicting the assumption that $\nu\in\Delta^{h\times\ell}$.
\end{proof}

The above lemma states that we can always write
$\Delta^{h\times\ell}$ as the set of triangular partitions contained in some suitable $\tau$. However, the converse is not true. For example, if $\tau=31$, the set of triangular subpartitions of $\tau$ is not of the form $\Delta^{h\times\ell}$ for any $h,\ell$, since a rectangle containing $\tau$ must also contain the partition $32$.

Our next goal is to give a formula for $\I(\sigma^\ell)$, which, by Lemma~\ref{equivalence_subpartitions-rectangle}, equals the number of triangular partitions that fit inside an $\ell\times \ell$ square. The proof uses the bijection $\omega$ from equation~\eqref{def:w}. 

\begin{lemma}
\label{lem:h<=l,w=l}
The following hold for $\ell\ge1$:
\begin{enumerate}[(a)]
    \item the number of triangular partitions of width exactly $\ell$ and height at most $\ell$ is $|\cB_\ell|/2$,
    \item $\left|\Delta^{\ell\times\ell}\setminus\Delta^{(\ell-1)\times(\ell-1)}\right|=\I(\sigma^\ell)-\I(\sigma^{\ell-1})=|\cB_\ell|-1$.
\end{enumerate}
\end{lemma}

\begin{proof}
By Lemma~\ref{lem:wide}, a triangular partition of width $\ell$ has height at most $\ell$ if and only if it is wide. By allowing $k$ to vary within $[\ell]$ in Proposition \ref{sturmian_encoding_1}, we get a bijection between such partitions and balanced words of length $\ell$ that start with~$1$. By the definition in equation~\eqref{eq:balanced}, it is clear that the operation on binary words that replaces the ones with zeros and the zeros with ones preserves the property of being balanced. 
Therefore, the number of balanced words of length $\ell$ that start with with~$1$ is half of the total number of balanced words of length $\ell$. This proves $(a)$.

To prove $(b)$, note that partitions that fit inside an $\ell\times \ell$ square but not inside an $(\ell-1)\times (\ell-1)$ square must have height or width exactly equal to $\ell$. By part $(a)$, $|\cB_\ell|/2$ is the number of triangular partitions of width $\ell$ and height at most $\ell$, and by conjugation, also the number of triangular partitions of height $\ell$ and width at most $\ell$. By Lemma~\ref{lem:wide}, the only partition that has width and height equal to $\ell$ is the staircase $\sigma^\ell$. Thus, $$\left|\Delta^{\ell\times\ell}\setminus\Delta^{(\ell-1)\times(\ell-1)}\right|=2\,|\cB_\ell|/2-1=|\cB_\ell|-1.$$ 
On the other hand, we have $\left|\Delta^{\ell\times\ell}\right|=\I(\sigma^\ell)$ by Lemma \ref{equivalence_subpartitions-rectangle}, and so
$\left|\Delta^{\ell\times\ell}\setminus\Delta^{(\ell-1)\times(\ell-1)}\right|=\I(\sigma^\ell)-\I(\sigma^{\ell-1})$.
\end{proof}

Lemma~\ref{lem:h<=l,w=l}$(a)$, combined with Theorem~\ref{lipatov}, implies that the number of triangular partitions of width exactly $\ell$ and height at most $\ell$ is
\begin{equation}\label{eq:h<=l,w=l}
\frac{|\cB_\ell|}{2}=\frac{1}{2} + \frac{1}{2}\sum_{i = 1}^\ell(\ell - i + 1)\varphi(i).
\end{equation}

\begin{theorem}
\label{thm:staircase}
For $\ell\ge0$, 
$$
\left|\Delta^{\ell\times\ell}\right|=\I(\sigma^\ell) = 1+\sum_{i = 1}^\ell\binom{\ell - i + 2}{2}\varphi(i).
$$
\end{theorem}
 
\begin{proof}
By Lemma~\ref{lem:h<=l,w=l}$(b)$, we have 
$\I(\sigma^j)-\I(\sigma^{j-1})=|\cB_j|-1$ for all $j\ge1$. Summing over $j\in[\ell]$, including the empty partition, and using Theorem~\ref{lipatov}, we obtain
\begin{align*}
\I(\sigma^\ell) &= 1+\sum_{j = 1}^\ell(|\cB_j|-1)= 1+\sum_{j = 1}^\ell\sum_{i = 1}^j(j - i + 1)\varphi(i) = 1+\sum_{i = 1}^\ell\varphi(i)\sum_{j = i}^\ell(j - i + 1) =\\
&= 1+\sum_{i = 1}^\ell\varphi(i)\sum_{k = 1}^{\ell - i + 1}k = 1+\sum_{i = 1}^\ell\binom{\ell - i + 2}{2}\varphi(i).\qedhere
\end{align*}
\end{proof}

The first few terms of the sequence $\left|\Delta^{\ell\times\ell}\right|=\I(\sigma^\ell)$ for $\ell\ge0$ are $1, 2, 5, 12, 25, 48, 83,\dots$, which did not appear in~\cite{oeis} at the time of writing this paper.

The above proof of Theorem~\ref{thm:staircase} relies on Lipatov's enumeration formula for balanced words (Theorem~\ref{lipatov}), and it does not give a conceptual understanding of why the terms $\binom{\ell - i + 2}{2}$ and $\varphi(i)$ appear. Theorem~\ref{lipatov}, first proved by Lipatov in~\cite{Lipatov1982}, has been rediscovered several times over the years, along with different proofs \cite{Berenstein88,Mignosi1991,Berstel93,Berstel1996,Cassaigne2001}. These proofs are quite technical, and do not easily provide a conceptual explanation of our formula for $\I(\sigma^\ell)$.
In the rest of this section, we give a direct, combinatorial proof of Theorem~\ref{thm:staircase} that does not rely on Lipatov's formula.
As an added benefit, our argument contributes a new proof of Lipatov's formula.

Let $\cI$ denote the set of partitions with all parts equal to $1$, including the empty partition.
We will encode partitions in $\Delta\setminus\cI$ by four integers, using ideas similar to those in~\cite{Corteel1999}.
For a nonempty triangular partition $\tau$, let $c= (a,b)$ be its rightmost removable cell. 

If $a=1$, there cannot be another removable cell to the left of $c$, so this is the only one. By Proposition \ref{charact_orcp_oacp}, either $b = 1$ or the line containing the edge of $\Conv(\tau)$ adjacent to $c$ from below must intersect $\Conv(\bbN^2\setminus\tau)$ above $c$, which implies that $\tau\in\cI$. 

If $a > 1$, consider the unique pair of relatively prime positive integers $(d,e)$ such that $(a - d, b + e)\in\bbN^2 \setminus \tau$ and $(a - d', b + e')\notin\bbN^2 \setminus \tau$ for any $d',e'\in\bbN$ with $e'/d' < e/d$ (see Figure~\ref{fig:abde}). Note that such a pair must always exist. For any $\tau\in\Delta\setminus\cI$, define $\phi(\tau) = (a,b,d,e)$.

\begin{figure}[ht]
\centering
	\begin{tikzpicture}[scale=0.6]
		\filldraw[color=black, fill=yellow, semithick] (0,0) -- (0,6) -- (1,6) -- (1,5) -- (2,5) -- (2,4) -- (4,4) -- (4,3) -- (6,3) -- (6,2) -- (7,2) -- (7,1) -- (9,1) -- (9,0);
        \axes{11}{8}
        \draw[black] (3,5) -- (6,5) -- (6,3);
        \node at (4.5,5.5) {$d$};
        \node at (6.3,4) {$e$};
        \draw[orange, thick] (0,7) -- (10.5,0);
        \node at (10.5, -0.4) [orange]{$L$};
        \draw[green, thick] (0,6.6) -- (11,0);
        \node at (11.25, -0.4) [green]{$L_\varepsilon$};        
        \draw[black] (3,5) circle (3pt) node[below left]{$(a-d,b+e)$};
        \filldraw[black] (6,3) circle (3pt) node[right, outer sep=3pt]{$(a,b)$};
        \filldraw[black] (9,1) circle (2.5pt);
	\end{tikzpicture}
\caption{A triangular partition $\tau$ with $\phi(\tau) =(a,b,d,e)$, and the cutting line $L_\varepsilon$.}
\label{fig:abde}
\end{figure}
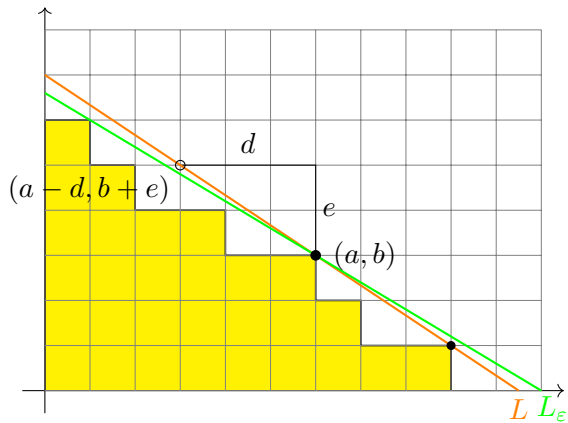

\begin{lemma}
\label{lem:phi_bijection}
The map $\phi$ is a bijection between $\Delta\setminus\cI$ and $$Q=\{(a,b,d,e)\in\bbN^4\mid d < a,\;\gcd(d,e) = 1\}.$$
\end{lemma}

\begin{proof}
It is clear from the definition of $\phi$ that if $\tau\in\Delta\setminus\cI$, then $\phi(\tau)\in Q$.

Let us first prove that $\phi$ is surjective. Given $(a,b,d,e)\in Q$, let $L$ be the line passing through $(a,b)$ and $(a-d,b+e)$, which has equation $e(x-a) + d(y-b) = 0$. 
Let $L_\varepsilon$ be the line with equation $(e - \varepsilon)(x-a) + d(y-b) = 0$, where $\varepsilon$ is a positive irrational number small enough so that there are no lattice points in the open region between $L_\varepsilon$ and $L$ in the first quadrant.
Let $\tau$ be the triangular partition cut off by $L_\varepsilon$. We claim that $\phi(\tau) = (a,b,d,e)$.

Indeed, since $\varepsilon$ is irrational, $c = (a,b)$ is the only lattice point in $L_\varepsilon$. It follows that $c$ is removable in $\tau$, since a cutting line for $\tau\setminus\{c\}$ can be obtained with a small perturbation of $L_\varepsilon$. Clearly, $(a - d, b + e)\in\bbN^2\setminus\tau$ because this point lies above $L_\varepsilon$, and $d < a$. Additionally, for any $d',e'\in\bbN$ with $e'/d' < e/d$, we have $(a - d', b + e')\notin\bbN^2\setminus\tau$ because there are no lattice points between $L_\varepsilon$ and $L$. To show that $c$ is the rightmost removable cell of $\tau$, suppose for contradiction that there was another removable cell $c'$ to the right of $c$. Then $c'$ must lie weakly below $L$, since there are no lattice points (other than $c$) on $L_\varepsilon$ or between $L_\varepsilon$ and $L$. Any cutting line for $\tau\setminus\{c'\}$ would pass below $c'$ and above $c$, forcing it to pass above $(a - d, b + e)$, which is not in $\tau$, reaching a contradiction. 

To prove that $\phi$ is injective, we will argue that if $\nu\in\Delta\setminus\cI$ is such that $\phi(\nu) = (a,b,d,e)$, then $\nu=\tau$. By the definition of $\phi$, all points in $\bbN^2$ lying strictly below $L$ and weakly to the left of $c=(a,b)$ belong to $\nu$. By Lemma~\ref{line_touching_removable_cell}, there exists a cutting line passing through $c$, and the point $(a - d, b + e)$ must lie above this line. Thus, all points in $\bbN^2$ lying weakly to the left of $c$ and weakly above $L$ belong to $\bbN^2\setminus\nu$, while those lying weakly below $L$ and to the right of $c$ belong to $\nu$. If $b = 1$, this implies that $L_\varepsilon$ cuts off $\nu$, and $\nu = \tau$ follows. Otherwise, let $L'$ be the line through $c$ and the vertex of $\Conv(\nu)$ adjacent to $c$ from the right. Since $c$ is the rightmost removable cell of $\nu$, the line $L'$ must intersect $\Conv(\bbN^2\setminus\nu)$ by Proposition~\ref{charact_trcp_tacp}. Additionally, by Proposition~\ref{charact_orcp_oacp} and Lemma~\ref{lemma-binary-search}, this intersection must occur to the left of $c$, and so the point $(a - d, b + e)$ must be weakly below $L'$.
In other words, the slope of $L'$, in absolute value, is greater than or equal to that of $L$. Since all of $\nu$ lies weakly below $L'$, all the cells lying to the right of $c$ and strictly above $L$ must belong to $\bbN^2\setminus\nu$. It follows that $L_\varepsilon$ cuts off $\nu$, implying that $\nu = \tau$.
\end{proof}

Next we will determine the possible values of the image $\phi(\tau)=(a,b,d,e)$ for partitions $\tau$ that fit inside an $\ell\times\ell$ square. 
The next two lemmas characterize, for given $d,e\le\ell$ with $\gcd(d,e)=1$, what are the possible values of $a$ and $b$ that are obtained.
We treat the cases $d<e$ and $d\ge e$ separately, and they are
illustrated in Figures \ref{example_triangle_1} and \ref{example_triangle_2}, respectively.

For positive integers $d,e,\ell$ with 
$d<e$, define the triangle
$$T_{d,e,\ell}^<=\{(x,y)\in\bbR^2\mid x\ge d+1,\,y\ge1,\,ex+dy\le e+d(\ell+1)\}.$$

\begin{lemma}
\label{LxL_triangle1}
Let $\tau\in\Delta\setminus\cI$ with $\phi(\tau)=(a,b,d,e)$, and suppose that $d<e$. Then $\tau\in\Delta^{\ell\times\ell}$  if and only if $(a,b)\in T_{d,e,\ell}^<$.
\end{lemma}

\begin{proof}
Let $L$ be the line described in the proof of Lemma~\ref{lem:phi_bijection}. As shown in that proof, a point belongs to $\bbN^2\setminus\tau$ if and only if it lies strictly to the left of $(a,b)$ and weakly above $L$ or weakly to the right of $(a,b)$ and strictly above $L$. Since $d < e$, requiring that $(1, \ell + 1)$ lies weakly above $L$ forces $(\ell + 1, 1)$ to lie strictly above $L$.

We deduce that $\tau\in\Delta^{\ell\times\ell}$ if and only if $(1, \ell + 1)$ lies weakly above $L$, that is, $e + d(\ell + 1) \geq ae + bd$. Noting that  $b\ge1$ and $a\ge d+1$ (since $d<a$), this inequality is  satisfied precisely when $(a,b)\in T_{d,e,\ell}^<$.
\end{proof}

\begin{figure}[ht]
\centering
	\begin{tikzpicture}[scale=.56]
        \filldraw[color=green, thick, fill=green!40] (3,3) -- (3,1) -- (13/3, 1) -- (3,3);
        \filldraw[color=black, fill=yellow, semithick] (0,0) -- (0,3) -- (1,3) -- (1,2) -- (2,2) -- (2,1) -- (3,1) -- (3,0) -- (0,0);
        \axes{6}{7}
        \draw[red, thick] (0,0) rectangle (5,5);
        \filldraw[gray] (3,3) circle (2.5pt);
        \filldraw[gray] (3,2) circle (2.5pt);
        \filldraw[gray] (3,1) circle (2.5pt);
        \filldraw[gray] (4,1) circle (2.5pt);
        \draw[orange, thick] (0,5.5) -- (11/3,0);
        \draw[black] (1,4) circle (3pt);
        \filldraw[black] (3,1) circle (3pt);

\begin{scope}[shift={(7.5,0)}]
        \filldraw[color=green, thick, fill=green!40] (3,3) -- (3,1) -- (13/3, 1) -- (3,3);
        \filldraw[color=black, fill=yellow, semithick] (0,0) -- (0,4) -- (1,4) -- (1,3) -- (2,3) -- (2,2) -- (3,2) -- (3,0) -- (0,0);
        \axes{6}{7}
        \draw[red, thick] (0,0) rectangle (5,5);
        \filldraw[gray] (3,3) circle (2.5pt);
        \filldraw[gray] (3,2) circle (2.5pt);
        \filldraw[gray] (3,1) circle (2.5pt);
        \filldraw[gray] (4,1) circle (2.5pt);
        \draw[orange, thick] (0,6.5) -- (13/3,0);
        \draw[black] (1,5) circle (3pt);
        \filldraw[black] (3,2) circle (3pt);
\end{scope}
 
 \begin{scope}[shift={(15,0)}]
        \filldraw[color=green, thick, fill=green!40] (3,3) -- (3,1) -- (13/3, 1) -- (3,3);
        \filldraw[color=black, fill=yellow, semithick] (0,0) -- (0,5) -- (1,5) -- (1,4) -- (2,4) -- (2,3) -- (3,3) -- (3,1) -- (4,1) -- (4,0) -- (0,0);
        \axes{6}{7}
        \draw[red, thick] (0,0) rectangle (5,5);
        \filldraw[gray] (3,3) circle (2.5pt);
        \filldraw[gray] (3,2) circle (2.5pt);
        \filldraw[gray] (3,1) circle (2.5pt);
        \filldraw[gray] (4,1) circle (2.5pt);
        \draw[orange, thick] (1/3,7) -- (5,0);
        \draw[black] (1,6) circle (3pt);
        \filldraw[black] (3,3) circle (3pt);
\end{scope}

\begin{scope}[shift={(22.5,0)}]
        \filldraw[color=green, thick, fill=green!40] (3,3) -- (3,1) -- (13/3, 1) -- (3,3);
        \filldraw[color=black, fill=yellow, semithick] (0,0) -- (0,5) -- (1,5) -- (1,3) -- (2,3) -- (2,2) -- (3,2) -- (3,1) -- (4,1) -- (4,0) -- (0,0);
        \axes{6}{7}
        \draw[red, thick] (0,0) rectangle (5,5);
        \filldraw[gray] (3,3) circle (2.5pt);
        \filldraw[gray] (3,2) circle (2.5pt);
        \filldraw[gray] (3,1) circle (2.5pt);
        \filldraw[gray] (4,1) circle (2.5pt);
        \draw[orange, thick] (0,7) -- (14/3,0);
        \draw[black] (2,4) circle (3pt);
        \filldraw[black] (4,1) circle (3pt);
\end{scope}
	\end{tikzpicture}
\caption{A triangular partition $\tau$ with $\phi(\tau) = (a,b,2,3)$ fits in a $5\times5$ square if and only if the point $(a,b)$ is in the triangle $T_{2,3,5}^<$, colored in green.}
\label{example_triangle_1}
\end{figure}
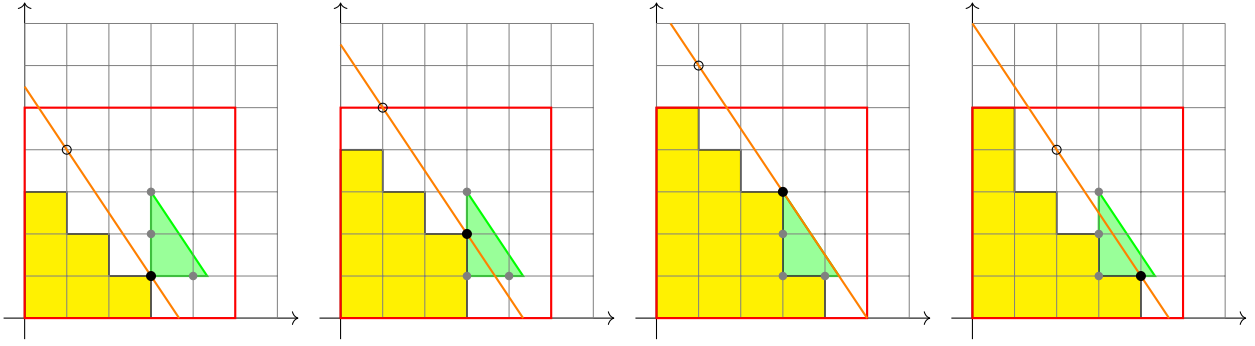

For positive integers $d,e,\ell$ with $e\le d$, define the triangle
$$T_{d,e,\ell}^\ge=\{(x,y)\in\bbR^2\mid x\ge d+1,\,y\ge1,\,ex+dy< e(\ell+1)+d\}.$$
Note that the last inequality is strict, unlike in the definition of $T_{d,e,\ell}^<$.

\begin{lemma}
\label{LxL_triangle2}
Let $\tau\in\Delta\setminus\cI$ with $\phi(\tau)=(a,b,d,e)$, and suppose that $e\le d$. Then $\tau\in\Delta^{\ell\times\ell}$ if and only if $(a,b)\in T_{d,e,\ell}^\ge$.
\end{lemma}

\begin{proof}
The proof is analogous to that of Lemma~\ref{LxL_triangle1}, except that since now $e \leq d$, requiring that $(\ell + 1, 1)$ lies strictly above $L$ forces $(1, \ell + 1)$ to lie strictly above $L$.
We deduce that $\tau\in\Delta^{\ell\times\ell}$ if and only if $(\ell + 1, 1)$ lies strictly above $L$, that is, $e(\ell + 1) + d > ae + bd$. Noting that $b \ge 1$ and $a\ge d + 1$, this is equivalent to requiring  $(a,b)\in T_{d,e,\ell}^\ge$.
\end{proof}

\begin{figure}[ht]
\centering
	\begin{tikzpicture}[scale=.6]
        \filldraw[color=white, fill=blue!40] (6,1) -- (4,1) -- (4, 7/3) -- (6,1);
        \filldraw[color=black, fill=yellow, semithick] (0,0) -- (0,2) -- (2,2) -- (2,1) -- (4,1) -- (4,0) -- (0,0);
        \draw[blue, thick, dashed] (6,1) -- (4,7/3);
        \draw[blue, thick] (6,1) -- (4,1) -- (4,7/3);
        \axes{7}{6}
        \draw[red, thick] (0,0) rectangle (5,5);
        \filldraw[gray] (5,1) circle (2.5pt);
        \filldraw[gray] (4,1) circle (2.5pt);
        \filldraw[gray] (4,2) circle (2.5pt);
        \draw[orange, thick] (0,11/3) -- (5.5,0);
        \draw[black] (1,3) circle (3pt);
        \filldraw[black] (4,1) circle (3pt);

\begin{scope}[shift={(9,0)}]
        \filldraw[color=white, fill=blue!40] (6,1) -- (4,1) -- (4, 7/3) -- (6,1);
        \draw[blue, thick, dashed] (6,1) -- (4,7/3);
        \draw[blue, thick] (6,1) -- (4,1) -- (4,7/3);
        \filldraw[color=black, fill=yellow, semithick] (0,0) -- (0,3) -- (2,3) -- (2,2) -- (4,2) -- (4,1) -- (5,1) -- (5,0) -- (0,0);
        \axes{7}{6}
        \draw[red, thick] (0,0) rectangle (5,5);
        \filldraw[gray] (5,1) circle (2.5pt);
        \filldraw[gray] (4,1) circle (2.5pt);
        \filldraw[gray] (4,2) circle (2.5pt);
        \draw[orange, thick] (0,14/3) -- (7,0);
        \draw[black] (1,4) circle (3pt);
        \filldraw[black] (4,2) circle (3pt);
\end{scope}

\begin{scope}[shift={(18,0)}]
        \filldraw[color=white, fill=blue!40] (6,1) -- (4,1) -- (4, 7/3) -- (6,1);
        \draw[blue, thick, dashed] (6,1) -- (4,7/3);
        \draw[blue, thick] (6,1) -- (4,1) -- (4,7/3);
        \filldraw[color=black, fill=yellow, semithick] (0,0) -- (0,3) -- (1,3) -- (1,2) -- (3,2) -- (3,1) -- (5,1) -- (5,0) -- (0,0);
        \axes{7}{6}
        \draw[red, thick] (0,0) rectangle (5,5);
       \filldraw[gray] (5,1) circle (2.5pt);
        \filldraw[gray] (4,1) circle (2.5pt);
        \filldraw[gray] (4,2) circle (2.5pt);
        \draw[orange, thick] (0,13/3) -- (6.5,0);
        \draw[black] (2,3) circle (3pt);
        \filldraw[black] (5,1) circle (3pt);
\end{scope}
	\end{tikzpicture}
\caption{A triangular partition $\tau$ with $\phi(\tau) = (a,b,3,2)$ fits in a $5\times5$ square if and only if the point $(a,b)$ is in the triangle $T_{3,2,5}^\ge$, colored in blue.}
\label{example_triangle_2}
\end{figure}

To count the number of lattice points in the above triangles, it is convenient to pair up  $T_{d,e,\ell}^<$ and $T_{e, e - d, \ell}^\ge$ for $d<e$. The next lemma could be proved by induction on $\ell$, using some elementary number theory, but we prefer to instead present a geometric argument, illustrated in Figure \ref{example_gluing_triangles}, that provides more intuition for the resulting binomial coefficient.

\begin{lemma}
\label{LxL_sum_of_triangles}
Let $\ell,d,e\in\bbN$ such that $1\le d < e \leq \ell$. Then,
$$
\left| T_{d,e,\ell}^<\cap\bbN^2 \right| + \left| T_{e, e - d, \ell}^\ge\cap\bbN^2 \right| = \binom{\ell - e + 2}{2}.
$$
\end{lemma}

\begin{proof}
Consider the affine transformation $\bbR^2\to\bbR^2$ given by
$$\begin{cases} u=-x-y+d+\ell+3,\\
v=x-e.\end{cases}$$
This transformation is bijective, with inverse
$$\begin{cases} x=v+e,\\
y=-u-v+d-e+\ell+3,\end{cases}$$
and it preserves the lattice $\bbZ^2$. The image of $T_{e, e - d, \ell}^\geq$ under this transformation is
$$T'_{d, e, \ell}\coloneqq\{(u,v)\in\bbR^2\mid v\ge1,\,u+v\le d-e+\ell+2,\,eu+dv>d(\ell+1)+e\}.$$

The triangles $T_{d,e,\ell}^<$ and $T'_{d,e,\ell}$ are disjoint, and their union is the triangle 
$$T_{d, e, \ell}\coloneqq T_{d,e,\ell}^<\sqcup T'_{d,e,\ell} =\{(x,y)\in\bbR^2\mid x\ge d+1,\,y\ge1,\,x+y\le d-e+\ell+2\}.$$
The lattice points in $T_{d, e, \ell}$ consist of horizontal rows of cardinality $\ell - e + 1$, $\ell - e$, \dots , $2$, $1$. Thus,
$$
\left| T_{d,e,\ell}^<\cap\bbN^2 \right| + \left| T_{e, e - d, \ell}^\geq\cap\bbN^2 \right| = \left| T_{d,e,\ell}^<\cap\bbN^2 \right| + \left|T'_{d, e, \ell}\cap\bbN^2 \right|
= \left| T_{d, e, \ell}\cap\bbN^2 \right| = \binom{\ell - e + 2}{2}.\qedhere
$$
\end{proof}

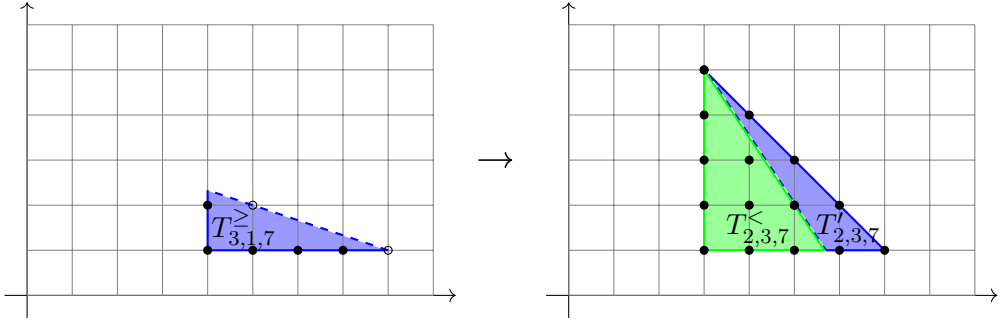
\begin{figure}[ht]
\centering
	\begin{tikzpicture}[scale=.6]
        \filldraw[color=white, fill=blue!40] (8,1) -- (4,1) -- (4, 7/3) -- (8,1);
        \axes{9}{6}
        \draw[blue, thick, dashed] (8,1) -- (4,7/3);
        \draw[blue, thick] (8,1) -- (4,1) -- (4,7/3);
        \draw[black] (8,1) circle (2.5pt);
        \draw[black] (5,2) circle (2.5pt);
        \filldraw[black] (7,1) circle (2.5pt);
        \filldraw[black] (6,1) circle (2.5pt);
        \filldraw[black] (5,1) circle (2.5pt);
        \filldraw[black] (4,1) circle (2.5pt);
        \filldraw[black] (4,2) circle (2.5pt);
        \node at (4.8,1.47) {$T_{3,1,7}^\geq$};

        \draw[->, semithick] (10, 3) -- (10.75, 3);
        
 \begin{scope}[shift={(12,0)}]
        \filldraw[color=white, fill=blue!40] (3.1,4.9) -- (5.7,1) -- (7, 1) -- (3.1,4.9);
        \filldraw[color=white, fill=green!40] (3,1) -- (3,5) -- (17/3, 1) -- (3,1);
        \axes{9}{6}
        \draw[blue, thick, dashed] (3.1,4.9) -- (5.7,1);
         \draw[blue, thick] (5.7,1) -- (7,1) -- (3.1,4.9);
         \draw[green, thick] (3,1) -- (3,5) -- (17/3,1)  -- (3,1);
        \draw[black] (3,5) circle (2.5pt);
        \draw[black] (5,2) circle (2.5pt);
        \filldraw[black] (6,1) circle (2.5pt);
        \filldraw[black] (7,1) circle (2.5pt);
        \filldraw[black] (6,2) circle (2.5pt);
        \filldraw[black] (5,3) circle (2.5pt);
        \filldraw[black] (4,4) circle (2.5pt);
        \node at (6.2,1.47) {$T'_{2,3,7}$};
      
        \filldraw[black] (3,5) circle (2.5pt);
        \filldraw[black] (3,4) circle (2.5pt);
        \filldraw[black] (3,3) circle (2.5pt);
        \filldraw[black] (3,2) circle (2.5pt);
        \filldraw[black] (3,1) circle (2.5pt);
        \filldraw[black] (4,3) circle (2.5pt);
        \filldraw[black] (4,2) circle (2.5pt);
        \filldraw[black] (4,1) circle (2.5pt);
        \filldraw[black] (5,2) circle (2.5pt);
        \filldraw[black] (5,1) circle (2.5pt);
        \node at (4.2,1.47) {$T_{2,3,7}^<$};
\end{scope}
	\end{tikzpicture}
\caption{Example of the proof of Lemma \ref{LxL_sum_of_triangles} for $d=2$, $e=3$ and $\ell=7$.}
\label{example_gluing_triangles}
\end{figure}

We are now ready to give a self-contained proof of our formula for $\I(\sigma^\ell)=\left|\Delta^{\ell\times\ell}\right|$.

\begin{proof}[Combinatorial proof of Theorem \ref{thm:staircase}]
Let $\tau\in\Delta\setminus\cI$ and $\phi(\tau) = (a,b,d,e)\in Q$, as defined in Lemma~\ref{lem:phi_bijection}. By Lemmas \ref{LxL_triangle1} and \ref{LxL_triangle2}, we have $\tau\in\Delta^{\ell\times\ell}$ if and only if $d < e$ and $(a,b)\in T_{d,e,\ell}^<$, or $d\geq e$ and $(a,b)\in T_{d,e,\ell}^\geq$. Accounting for the $\ell+1$ partitions in $\cI\cap\Delta^{\ell\times\ell}$ (including the empty partition), we get
$$
\left|\Delta^{\ell\times\ell}\right| = \ell +1 + \sum_{\substack{1\leq d < e \leq \ell \\ \gcd(d,e) = 1}}\left| T_{d,e,\ell}^<\cap\bbN^2 \right| + \sum_{\substack{1\leq e' \leq d' \leq \ell \\ \gcd(d',e') = 1}}\left|T_{d',e',\ell}^\geq \cap\bbN^2 \right|.
$$

The bijection $$\{(d,e)\in\bbN^2\mid d < e,\;\gcd(d,e) = 1 \}\to\{(d',e')\in\bbN^2\mid d' > e',\;\gcd(d',e') = 1 \}$$ given by $(d,e)\mapsto (e, e - d)$ allows us to combine the summations as
$$
\left|\Delta^{\ell\times\ell}\right| = \ell +1 + \left| T_{1,1,\ell}^\geq\cap\bbN^2 \right| + \sum_{\substack{1\leq d < e \leq \ell \\ \gcd(d,e) = 1}}\left( \left| T_{d,e,\ell}^<\cap\bbN^2 \right| + \left| T_{e, e - d, \ell}^\geq\cap\bbN^2 \right| \right).
$$
The lattice points in the triangle $T_{1,1,\ell}^\geq$ consist of horizontal rows of cardinality $\ell - 1, \ell - 2, \dots , 2, 1$, and so $\left|T_{1,1,\ell}^\geq\right| =  \binom{\ell}{2}$. Using Lemma \ref{LxL_sum_of_triangles}, we obtain
$$
\left|\Delta^{\ell\times\ell}\right| = \ell +1 + \binom{\ell}{2} + \sum_{\substack{1\leq d < e \leq \ell \\ \gcd(d,e) = 1}}\binom{\ell - e + 2}{2} = 1+ \sum_{e = 1}^\ell\binom{\ell - e + 2}{2}\varphi(e).\qedhere
$$
\end{proof}

We can now easily deduce Lipatov's enumeration formula for balanced words from our results. Indeed, by Lemma~\ref{lem:h<=l,w=l}$(b)$ and Theorem~\ref{thm:staircase}, 
$$|\cB_\ell|=1+\I(\sigma^\ell)-\I(\sigma^{\ell-1})=1+\sum_{i = 1}^{\ell}\binom{\ell - i + 2}{2}\varphi(i) - \sum_{i = 1}^{\ell - 1}\binom{\ell - i + 1}{2}\varphi(i) = 1+\sum_{i = 1}^\ell(\ell - i + 1)\varphi(i),
$$
giving a combinatorial proof of Theorem~\ref{lipatov}.

In the rest of this section, we show that similar formulas for the number of triangular partitions inside other rectangles can be derived from Theorem~\ref{thm:staircase}. However, we do not have a general formula for $\left|\Delta^{h\times\ell}\right|$ for arbitrary $h$ and $\ell$.

\begin{corollary}
\label{formula_LxL-1_rectangle}
For $\ell\ge2$,
$$
\left|\Delta^{\ell\times(\ell-1)}\right|=
\I(\sigma^\ell\setminus\{(\ell,1)\})=\frac{1}{2} + \frac{1}{2}\sum_{i = 1}^\ell(\ell - i + 1)^2\varphi(i).
$$
\end{corollary}

\begin{proof}
By Lemma \ref{equivalence_subpartitions-rectangle}, the triangular subpartitions of $\sigma^\ell\setminus\{(\ell,1)\}$ are the triangular partitions that fit inside an $\ell\times(\ell-1)$ rectangle; equivalently, those that fit inside an $\ell\times \ell$ square and do not have width exactly $\ell$. Those having width exactly $\ell$ are counted by equation~\eqref{eq:h<=l,w=l}. Subtracting this formula from $\left|\Delta^{\ell\times\ell}\right|$, which is given by Theorem \ref{thm:staircase}, we get
$$\left|\Delta^{\ell\times(\ell-1)}\right| = 1+\sum_{i = 1}^\ell\binom{\ell - i + 2}{2}\varphi(i) - \frac{1}{2} - \frac{1}{2}\sum_{i = 1}^\ell(\ell - i + 1)\varphi(i)
= \frac{1}{2} + \frac{1}{2}\sum_{i = 1}^\ell(\ell - i + 1)^2\varphi(i).\qedhere
$$
\end{proof}

To give a formula for the number of partitions that fit inside an $\ell\times (\ell - 2)$ rectangle, we need the following lemma.

\begin{lemma}
\label{length_exactly_L_height_exactly_L-1}
For $\ell\ge2$, the number of triangular partitions of width $\ell-1$ and height $\ell$ is $\ell - 1$.
\end{lemma}

\begin{proof}
A triangular partition $\tau$ of width $\ell-1$ and height $\ell$ must contain cells $(1, \ell)$ and $(\ell-1, 1)$, and hence all the cells lying weakly below the line segment between these two cells. In particular, letting $\nu=(\ell-1,\ell-2,\dots,1,1)$, we have $\nu\subseteq\tau$. On the other hand, since $\tau\in\Delta^{\ell\times(\ell-1)}$,  Lemma~\ref{equivalence_subpartitions-rectangle} implies that $\tau\subseteq\sigma^\ell\setminus\{(\ell,1)\}$. As a consequence, $\tau = \nu\cup C$, where $C$ is a subset of the cells that belong to $\sigma^\ell\setminus\{(\ell,1)\}$ but not to $\nu$. Denote these cells by $c_1,\dots,c_{\ell-2}$, where $c_i=(i+1,\ell-i)$.

If $c_i\in\tau$ for some $i\in[\ell-2]$, then $c_j\in\tau$ for all $j\in[i]$, since these cells lie on the line segment between $c_i$ and $(1,\ell)$. Thus, $C=\emptyset$ or $C=\{c_1,c_2,\dots,c_i\}$ for some $i\in[\ell-2]$. In the latter case, it is easy to see that the partition $\nu\cup C$ is triangular, since we can find a cutting line through $c_i$ with slope vector $(1/2 + \varepsilon, 1/2 - \varepsilon)$ for small enough $\varepsilon > 0$. Since there are $\ell - 1$ choices for $C$ in total, the result follows.
\end{proof}

\begin{corollary}
\label{formula_LxL-2_rectangle}
For $\ell\ge3$, 
$$\left|\Delta^{\ell\times(\ell-2)}\right|=
1-\ell + \frac{1}{2}\sum_{i = 1}^\ell\big((\ell - i + 1)(\ell - i) + 1\big)\varphi(i).
$$
\end{corollary}

\begin{proof}
Partitions in $\Delta^{\ell\times(\ell-2)}$ are precisely those in $\Delta^{\ell\times(\ell-1)}$ (which were counted in Corollary~\ref{formula_LxL-1_rectangle}) whose width is not $\ell - 1$. 
Partitions in $\Delta^{\ell\times(\ell-1)}$ of width $\ell - 1$ either have height at most $\ell-1$ (counted in equation~\eqref{eq:h<=l,w=l} with $\ell-1$ playing the role of $\ell$), or they have height $\ell$ (counted in Lemma~\ref{length_exactly_L_height_exactly_L-1}). Putting these formulas together,
\begin{align*}\left|\Delta^{\ell\times(\ell-2)}\right|&=
\frac{1}{2} + \frac{1}{2}\sum_{i = 1}^\ell(\ell - i + 1)^2\varphi(i)  - \left(\frac{1}{2} + \frac{1}{2}\sum_{i = 1}^{\ell - 1}(\ell - i)\varphi(i)\right) - (\ell - 1)\\
&= 1-\ell + \frac{1}{2}\sum_{i = 1}^\ell\big((\ell - i + 1)(\ell - i) + 1\big)\varphi(i).\qedhere
\end{align*}
\end{proof}

\section{Further directions}\label{sec:further}

In this section we discuss possible generalizations of our work and avenues of further research.

\subsection{Triangular Young tableaux}

Given a partition $\lambda$ of $n$, a \emph{standard Young tableau} of shape $\lambda$ is a filling of the cells of the Young diagram of $\lambda$ with the numbers $1,2,\dots,n$ so that each number appears once, the rows are increasing from left to right, and the columns are increasing from bottom to top. The last two conditions are equivalent to the requirement that for all $i\in[n]$, the cells with labels at most $i$ form the Young diagram of a partition.
One can consider the following triangular analogue of this notion.

\begin{definition}
    Let $\tau$ be a triangular partition of size $n$. A \emph{triangular Young tableau} of shape $\tau$ is a filling of the cells of the Young diagram of $\tau$ with the numbers $1,2,\dots,n$ so that, for all $i\in[n]$, the cells with labels at most $i$ form the Young diagram of a triangular partition.
\end{definition}

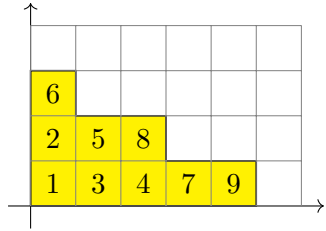
\begin{figure}[ht]
\centering
	\begin{tikzpicture}[scale=.6]
		\filldraw[color=black, fill=yellow, semithick] (0,0) -- (0,3) -- (1,3) -- (1,2) -- (3,2) -- (3,1) -- (5,1) -- (5,0) -- (0,0);
        \axes{6}{4}
        \node at (0.5,0.5) {$1$};
        \node at (0.5,1.5) {$2$};
        \node at (1.5,0.5) {$3$};
        \node at (2.5,0.5) {$4$};
        \node at (1.5,1.5) {$5$};
        \node at (0.5,2.5) {$6$};
        \node at (3.5,0.5) {$7$};
        \node at (2.5,1.5) {$8$};
        \node at (4.5,0.5) {$9$};
	\end{tikzpicture}
\caption{A triangular Young tableau of shape $531$.}
\end{figure}

Similarly to how standard Young tableaux can be viewed as walks in Young's lattice, we can interpret triangular Young tableaux of shape $\tau$ as walks of length $n$ in the Hasse diagram of $\TYP$ from the empty partition to $\tau$.

It is natural to ask if there is a triangular analogue of the hook-length formula counting standard Young tableaux of a given shape. 

\begin{problem}
    Find a formula for the number of triangular Young tableaux of a given shape.
\end{problem}

We can solve this problem in the special case of two-row shapes.

\begin{proposition}
\label{counting_TYT}
    For a two-part triangular partition $\tau = \tau_1\tau_2$, the number of triangular Young tableaux of shape $\tau$ is
    $$
    \frac{\tau_1 - 2\tau_2 + 2}{\tau_1 + 2}\binom{\tau_1 + \tau_2 + 1}{\tau_2}.
    $$
\end{proposition}

The proof of this result relies on the observation that an integer partition $\lambda = \lambda_1\lambda_2$ is triangular if and only if $\lambda_1 \geq 2\lambda_2 - 1$. We deduce that triangular Young tableaux of shape $\tau = \tau_1\tau_2$ are in bijection with lattice paths from $(0,0)$ to $(\tau_1,\tau_2)$, with steps $(1,0)$ and $(0,1)$, and staying weakly below the line $x=2y-1$. The bijection simply makes the $i$th step of the path be $(1,0)$ (resp.\ $(0,1)$) if entry $i$ is in the bottom (resp.\ top) row of the tableau. These paths can then be counted similarly to ballot paths to deduce the above formula.

Using the Robinson--Schensted bijection, triangular Young tableaux with at most two rows can also be interpreted as $321$-avoiding involutions with certain additional restrictions.

\subsection{Pyramidal partitions}

There is a higher-dimensional analogue of triangular partitions. 
A {\em $d$-dimensional pyramidal partition} is a finite subset of $\bbN^d$ that can be separated from its complement by a hyperplane in $\mathbb{R}^d$.
These objects were first considered in \cite{Onn1999}, and some bounds on their growth were given in \cite{Wagner2002}.
By construction, $2$-dimensional pyramidal partitions are the same as triangular partitions, whereas $3$-dimensional pyramidal partitions can be viewed as a subset of plane partitions.

One could ask which of the results from this paper generalize to pyramidal partitions. For example, the argument in the proof of Proposition~\ref{charact_triang2D} also shows that a finite subset $\pi\subset\bbN^d$ is a $d$-dimensional pyramidal partition if and only if $\Conv(\pi)\cap\Conv(\bbN^2\setminus\pi) = \emptyset$. In addition, some experimentation suggests that the M\"obius function of the poset of $3$-dimensional pyramidal partitions ordered by containment only takes values in $\{-1,0,1\}$, as in the case of triangular partitions.

However, some other properties of triangular partitions do not extend to higher dimensions. 
In work in progress, we have shown that for any fixed $d\ge3$, there are $d$-dimensional pyramidal partitions with an arbitrary large number of removable or addable cells, in contrast with Lemma~\ref{lem:removable-addable}.
It has also been observed by Vincent Pilaud (personal communication, 2023) that the poset of $3$-dimensional pyramidal partitions is no longer a lattice. Understanding the intervals of this poset is an interesting avenue of research.

\subsection{Convex and concave partitions}

Two other generalizations of triangular partitions are convex partitions and concave partitions.
Convex partitions are mentioned in OEIS entry A074658~\cite{oeis}, due to Dean Hickerson, where they are counted for size up to $55$. The notion of concave partitions appears in work of Blasiak et al.~\cite[Conjecture 7.1.1]{Blasiak2023} in connection to Schur positivity.

Convex (resp.\ concave) partitions can be defined as intersections (resp.\ unions) of triangular partitions, or equivalently as finite subsets of $\bbN^2$ that lie below a convex (resp.\ concave) polygonal curve.
In work in progress, we have shown that a partition is triangular if and only if it is convex and concave, and we have extended some of the results from Section~\ref{sec:triangular_young_poset} to the lattices of convex and concave partitions.

\subsection*{Acknowledgements}
The authors thank Fran\c{c}ois Bergeron for introducing them to triangular partitions and for helpful discussions.
SE was partially supported by Simons Collaboration Grant \#929653. AG was partially supported by the mobility grants of CFIS-UPC, Generalitat de Catalunya and Gobierno de Navarra.

\printbibliography

\end{document}